\documentclass[reqno,11pt]{amsart}
\usepackage{amssymb, amsmath,latexsym,amsfonts,amsbsy, amsthm}
\usepackage{color}
\usepackage{txfonts}
\usepackage{mdframed}
\usepackage[hypertexnames=false]{hyperref}

\newtheorem{theorem}{Theorem}[section]

\newtheorem{lemma}[theorem]{Lemma}
\newtheorem{proposition}[theorem]{Proposition}
\newtheorem{corollary}[theorem]{Corollary}
\newtheorem{remark}[theorem]{Remark}

\def\bR{\mathbb R}
\def\bN{\mathbb N}
\def\bZ{\mathbb Z}

\def\bT{\mathbb T}

\def\la{\lambda}

\def\t{\tilde}
\def\q{\quad}
\def\qq{\qquad}
\def\th{\theta}

\def\dl{\delta}
\def\Dl{\Delta}
\def\lt{\left}
\def\rt{\right}

\def\i{\infty}
\def\e{\epsilon}

\def \ls{\lesssim}
\def\gs{\gtrsim}
\def\p{\partial}
\def\f{\frac}
\def\na{\nabla}
\def\al{\alpha}

\def\O{\Omega}
\def\o{\omega}

\def\s{\sqrt}
\def\bl{\boldsymbol}

\def\red{\color{red}}

\def\nn{\nonumber}
\def\be{\begin{equation}}
\def\ee{\end{equation}}
\def\bes{\begin{equation*}}
\def\ees{\end{equation*}}
\def\bali{\begin{aligned}}
\def\eali{\end{aligned}}
\def\beas{\begin{eqnarray*}}
\def\eeas{\end{eqnarray*}}

\def\pf{\noindent {\bf Proof. \hspace{2mm}}}
\def\ef{ \hfill $ \Box $ \vskip 3mm}

\allowdisplaybreaks[4]

\begin{document}
\title[Vanishing viscosity limit of 2D stationary NS outside a disc and its application]
{Vanishing viscosity limit of the 2D stationary Navier-Stokes equations outside a rotating disc and its application}

\author{Xinghong Pan}
\address[X. Pan]{School of Mathematics and Key Laboratory of MIIT, Nanjing University of Aeronautics and Astronautics, Nanjing 211106, China}
\email{xinghong\_87@nuaa.edu.cn}

\author{Jianfeng Zhao}
\address[J. Zhao]{School of Mathematics and Information Science, Guangxi University, Nanning 530004, China}
\email{zhaojianfeng@amss.ac.cn}




\thanks{X. Pan is supported by National Natural Science Foundation of China (No. 12471222).}

\begin{abstract}
   We justify the vanishing-viscosity limit for 2D stationary Navier-Stokes flows in the exterior of a rotating disc. The boundary data are a small perturbation of a nonzero rigid rotation, and the velocity vanishes at spatial infinity. Guided by Prandtl-Batchelor theory, we select the limiting Euler flow $\f{A}{r}e_\th$, where $A$ is fixed by the Batchelor-Wood formula. In the regime of small viscosity, we construct stationary Navier-Stokes solutions through a higher-order matched asymptotic expansion and prove the validity of the corresponding boundary-layer expansion. We also identify the far-field asymptotics of the Navier-Stokes solution. This gives a partial resolution of Problem 11b in Yudovich [{\it Eleven great problems of mathematical hydrodynamics}, Mosc. Math. J. 3 (2003), no. 2, 711--737]. As an application, for fixed viscosity we prove the existence of stationary solutions in the same exterior domain under a large perturbation of a fast rigid rotation at the boundary.
\end{abstract}

\maketitle

\numberwithin{equation}{section}


\indent

\section{Introduction}

\indent

 We consider the stationary Navier-Stokes equations in the exterior of the unit disc $\O=\bR^2\setminus B$,
\begin{equation}\label{ns}
\left \{
\begin {array}{ll}
\bl{u}^\e\cdot\na \bl{u}^\e+\na p^\e-\e^2\Dl \bl{u}^\e=0,\\ [5pt]
\na\cdot\bl{u}^\e=0,
\end{array}
\right.
\end{equation}
with the boundary condition
\be\label{nsboundary}
\bali
&\bl{u}^\e\cdot \boldsymbol{\tau}\big|_{\p B}=\o +\delta f(x,y),\q \bl{u}^\e\cdot \boldsymbol{n}\big|_{\p B}=0,\q \bl{u}^\e \big|_{x,y\rightarrow\i}=0,
\eali
\ee
Here $(x,y)\in \bR^2$, and $B$ is the unit disc centered at the origin. The vector field $\bl{u}^\e$ denotes the velocity and $p^\e$ the pressure. The parameter $\e^2>0$ is the viscosity, i.e. the reciprocal of the Reynolds number. The constant $\o>0$ is fixed, $\dl$ is small, and $f$ is a smooth function on $\p B$. On $\p B$, $\boldsymbol{n}$ and $\boldsymbol{\tau}$ denote the unit outward normal and unit tangential vectors, respectively, with orientation consistent with the Euclidean basis $(\bl{e}_1,\bl{e}_2)$.

Our main objective is to study the vanishing-viscosity limit for system \eqref{ns}, motivated by one of the open problems raised by V. I. Yudovich in \cite[Eleven great problems of mathematical hydrodynamics]{Yudovich:2003MMJ}. We recall this problem below.
\begin{mdframed}
\textbf{Problem 11b.} Determine the limit of a stationary solution to the Navier-Stokes system as the viscosity approaches to zero. In particular, find the asymptotics of the stationary flow of a viscous fluid past a rigid body.
\end{mdframed}

Formally, letting $\e\to0$ in \eqref{ns}, we obtain the 2D steady Euler system
\begin{equation}\label{equeuler}
\left \{
\begin {array}{ll}
\bl{u}^e\cdot\na \bl{u}^e+\na p^e=0,\\ [5pt]
\na\cdot\bl{u}^e=0,
\end{array}
\right.
\end{equation}
with the boundary condition
\bes
\bl{u}^e\cdot\bl{n}\big|_{\p B}=0, \q \bl{u}^e\cdot\bl{n} \big|_{x,y\rightarrow\i}=0.
\ees

We prove the existence of solutions $\bl{u}^\e$ to \eqref{ns} that converge, as $\e\to0$, to a solution of \eqref{equeuler}. We also derive the asymptotic behavior of $\bl{u}^\e$ at spatial infinity. This provides a partial answer to \textbf{Problem 11b}.
Since the domain is exterior to a disc, it is convenient to rewrite both \eqref{ns} and \eqref{equeuler} in polar coordinates. For $(r,\th)$ with $x=r\cos\th$ and $y=r\sin\th$, the orthonormal polar basis
\bes
\bl{e}_r=(\cos\th,\sin\th),\q \bl{e}_\th=(-\sin\th,\cos\th)
\ees
has the same orientation as $(\bl{e}_1,\bl{e}_2)$. We write
\bes
\bl{u}=(u,v)=u\bl{e}_\th+v\bl{e}_r.
\ees
Then \eqref{ns}--\eqref{nsboundary} becomes
\be\label{nspolar}
\left\{
\begin{array}{l}
u^{\epsilon} u_\theta^{\epsilon}+r v^{\epsilon} u_r^{\epsilon}+u^{\epsilon} v^{\epsilon}+p_\theta^{\epsilon}-\epsilon^2\left(\frac{u_{\theta \theta}^{\epsilon}}{r}+r u_{r r}^{\epsilon}+u_r^{\epsilon}+\frac{2}{r} v_\theta^{\epsilon}-\frac{u^{\epsilon}}{r}\right)=0, \\
u^{\epsilon} v_\theta^{\epsilon}+r v^{\epsilon} v_r^{\epsilon}-\left(u^{\epsilon}\right)^2+r p_r^{\epsilon}-\epsilon^2\left(\frac{v_{\theta \theta}^{\epsilon}}{r}+r v_{r r}^{\epsilon}+v_r^{\epsilon}-\frac{2}{r} u_\theta^{\epsilon}-\frac{v^{\epsilon}}{r}\right)=0, \\
u_\theta^{\epsilon}+\left(r v^{\epsilon}\right)_r=0, \\
\lt(u^{\epsilon},v^\e\rt)(\theta, 1)=(\o+\delta f(\theta),0), \q (u^{\epsilon},v^\e)(\theta,r)|_{r\rightarrow+\i}=0,
\end{array}
\right.
\ee
where $\bl{u}^\e=(u^\e,v^\e)=u^\e\bl{e}_\th+v^\e\bl{e}_r$. Likewise, \eqref{equeuler} is rewritten as
\bes
\left\{
\begin{array}{l}
u^{e} u_\theta^{e}+r v^{e} u_r^{e}+u^{e} v^{e}+p_\theta^{e}=0, \\
u^{e} v_\theta^{e}+r v^{e} v_r^{e}-\left(u^{e}\right)^2+r p_r^{e}=0, \\
u_\theta^{e}+\left(r v^{e}\right)_r=0, \\
v^e(\theta, 1)=v^e(\theta, +\i)=0,
\end{array}
\right.
\ees
where $\bl{u}^e=(u^e,v^e)=u^e\bl{e}_\th+v^e\bl{e}_r$.

{\bf\noindent Choice of the limiting Euler flow}

In this work, we consider limiting Euler flows without stagnation points, i.e. $|\bl{u}^e|>0$ in the exterior domain. By \cite[Theorem 1.3]{WangZ:2023ARXIV}, such a flow must be a tangential shear flow $u^e(r)\bl{e}_\th$. Moreover, Prandtl-Batchelor theory (see, e.g., \cite[Appendix]{FeiGLT:2024ADV}) yields
\bes
\Dl u^e(r)=0,
\ees
 which implies that the only possible Euler flow is the Taylor--Couette profile
\bes
\bl{u}^e=\lt(\t{A} r+\f{A}{r}\rt)\bl{e}_\th.
\ees

 Since $u^\e\to0$ as $r\to+\i$, we impose the same far-field behavior on $u^e$, and thus $\t A=0$. The constant $A$ is then determined by the Batchelor--Wood formula; see \cite{FeiGLT:2023CMP,DrivasIN:2023ARXIV} for derivations, as well as Lemma \ref{lembw} and Section \eqref{sec431} below. The selection principle for $A$ is as follows:
{\it The integral average of the square of the tangential components to the NS solution and the limiting Euler solution on the boundary must be the same.}

This means that
\be \label{BatchelorWood}
\begin{aligned}
&\f{1}{2\pi}\int^{2\pi }_0 \f{A^2}{r^2}\Big|_{r=1}d\th=\f{1}{2\pi}\int^{2\pi }_0 \lt(\o+\dl f(\th)\rt)^2d\th:=\t{\o}^2.
\end{aligned}
\ee
From \eqref{BatchelorWood}, we see that
\bes
A=\t{\o}=\s{\f{1}{2\pi}\int^{2\pi }_0 \lt(\o+\dl f(\th)\rt)^2d\th}.
\ees

{\bf\noindent Statement of the main results}

Next, we introduce the leading-order stationary boundary-layer system near $r=1$. Let $(u_p^{(0)},v_p^{(1)})$ denote the leading boundary-layer profile, which satisfies
\be\label{boundarylayeruppermain}
\left\{
\begin {array}{ll}
\big(\t{\o}+u_p^{(0)}\big)\partial_\th u_p^{(0)}+\big( v_p^{(1)}- v_p^{(1)}(\th,0)\big)\partial_\zeta u_p^{(0)}-\partial^2_{\zeta}u_p^{(0)}=0,\\[5pt]
\partial_\th u_p^{(0)}+\partial_\zeta v_p^{(1)}=0,\\[5pt]
(u_p^{(0)},v_p^{(1)})(\th,\zeta)=(u_p^{(0)},v_p^{(1)})(\th+2\pi,\zeta),\\[5pt]
u_p^{(0)}\big|_{\zeta=0}=\lt(\o+\dl f(\th)-\t{\o}\rt),\q \lim\limits_{\zeta\rightarrow +\infty}u_p^{(0)}=\lim\limits_{\zeta\rightarrow +\infty}v_p^{(1)}=0,
\end{array}
\right.
\ee
 where $\zeta:=\f{r-1}{\e}$ is the stretched normal variable. System \eqref{boundarylayeruppermain} is derived via matched asymptotic expansions, and its well-posedness is established in Proposition \ref{propdcu0}.
The following is the main result of our paper.
\begin{theorem}\label{thmain}
Assume that $f(\th)$ is a smooth $2\pi$-periodic function. Then there exist two constants $\e_0$ and $\dl_0$ such that for any $\e\in(0,\e_0]$ and $\dl\in(0,\dl_0]$, the system \eqref{nspolar} has a solution $(u^\e,v^\e)$ satisfying
\be\label{linfinityesti}
\bali
\lt|u^\e(\th,r)-(\t{\o}+\mathcal{O}(\e\dl))r^{-1}-u_p^{(0)}(\th,\zeta)\chi(r)\rt|+\lt|v^\e(\th,r)\rt|\leq C\e\dl r^{-2},
\eali
\ee
where $\mathcal{O}(\e\dl)$ is a $\e\dl$-compared constant and $\chi(r)\in C^\i_c([1,+\i))$ satisfies
\bes
\chi(r)=\lt\{
\bali
& 1, r\in [1,2],\\
& 0, r\geq 3.
\eali
\rt.
\ees
The constant $C$ is independent of $\e$ and $\dl$. \qed
\end{theorem}

\begin{remark}

 By Proposition \ref{propdcu0}, the boundary-layer solution $u_p^{(0)}(\th,\zeta)$ is of size $\dl$, i.e. $|u_p^{(0)}(\th,\zeta)|\ls \dl$. Hence, \eqref{linfinityesti} implies that
\be\label{solestix}
\lt|u^\e(\th,r)-(\t{\o}+\mathcal{O}(\e\dl))r^{-1}\rt|\ls C\dl r^{-2},\q \lt|v^\e(\th,r)\rt|\leq C\e\dl r^{-2}.
\ee

\end{remark}

 In the presence of boundaries, rigorously justifying the vanishing-viscosity limit from the Navier--Stokes equations with no-slip boundary conditions to the Euler equations is a fundamental and difficult problem because of boundary-condition mismatch; see, for example, \cite{OleinikS:1999,SchlichtingG:2017}. Prandtl \cite{Prandtl:1904} introduced the formal boundary-layer expansion in 1904 to explain this mismatch. In the unsteady setting, rigorous validity results for Prandtl expansions are mainly available in analytic frameworks; see \cite{FTZ2018, KVW2020, M2014, NN2018, SC1998-1, SC1998-2, WWZ2017} and references therein. By contrast, validity in Sobolev spaces for unsteady flows remains open. We also mention Sobolev-space instability results for shear-type Prandtl expansions \cite{G2000,GGN2016, GN2019}, as well as nonlinear asymptotic stability and transition-threshold analysis near Taylor--Couette flow in \cite{AnHL:2024CMP}.

   For stationary inviscid limits, the problem is in some sense more delicate than in the unsteady case: the steady Euler system admits infinitely many solutions, and selecting the physically relevant limiting flow is the first issue. For general domains, no complete selection principle is currently known. Prandtl \cite{Prandtl:1904} observed for steady incompressible viscous flows that, in simply connected regions with closed streamlines, the Euler vorticity should be constant along each closed streamline. This property was later rediscovered by Batchelor and Wood \cite{Batchelor:1956, Wood:1957JFM}, who further related the constant value to boundary data. These ideas are now known as Prandtl-Batchelor theory; see \cite{Kim-thesis, Kim:1998SIAM, KimC:2001SIAM, WV:2007} and references therein.

In recent years, substantial progress has been made on asymptotic expansions and inviscid limits for steady Navier-Stokes equations. Guo-Nguyen \cite{GuoN:2017ANNPDE} established Prandtl expansions for steady flow over a moving plate in $(0,L)\times\bR_+$, and Iyer \cite{Iyer:2017ARMA} generalized this to a rotating disk. The case where the Euler flow is a perturbation of a shear flow was treated in \cite{Iyer:2019SIAM}. For a motionless plate, Guo-Iyer \cite{GuoI:2023CPAM} proved validity with shear Euler flow, and Gao-Zhang \cite{GaoZ:2024JMP} extended this to non-shear Euler flows. These results are obtained in narrow domains with $L\ll 1$.

For large-width plates, Iyer \cite{Iyer:2019Peking} proved global Prandtl expansions in $(0,+\i)\times(0,+\i)$ with boundary velocity $1-\dl$ ($\dl$ small) and limiting Euler flow $(1,0)$. Under motionless no-slip boundary conditions, Gao-Zhang \cite{GaoZ:2023SCM} established boundary expansions in $(0,L)\times(0,+\i)$ for any $L>0$, with nondegenerate shear Euler flow and concave Prandtl profile in the $y$ direction. Fei-Gao-Lin-Tao \cite{FeiGLT:2023CMP, FeiGLT:2024ADV} studied Prandtl expansions in a disk and an annulus, where the swirl boundary data are small perturbations of rigid rotation. Gao-Xin \cite{GX2023} considered an infinitely long convergent channel. More recently, Fei-Pan-Zhao \cite{FeiPZ:23ARXIV} and Chen-Fei-Lin-Zhao \cite{ChenFLZ:24arxiv} addressed vanishing-viscosity limits for 2D stationary Navier-Stokes flows in a periodic strip with one motionless boundary and in a disk with a center vortex.
We also refer to \cite{DM:2019, WangZ:2021AIHP,ShenWZ:2021, GuoWZ:2023ANNPDE,WangZ:2023MATHANN} for dynamic stability, asymptotic behavior, global $C^\i$ regularity, and boundary-layer separation for steady Prandtl equations, and to \cite{MasmoudiR:2012ARMA,WangXZ:2012JMFM,BedrossianHIW:2023ARXIV,BedrossianHIW:2024ARXIV,PanZ:2025preprint} for vanishing-viscosity limits of stationary and non-stationary Navier-Stokes equations with Navier-slip boundary conditions.
Using the strategy for Theorem \ref{thmain}, we also obtain an existence result for large perturbations of fast rotation in the fixed-viscosity stationary problem outside a disc. We formulate it as follows. Consider the stationary Navier-Stokes system in the exterior of the unit disc $\O=\bR^2\setminus B$,\begin{equation}\label{ns1}
\left \{
\begin {array}{ll}
\bl{u}\cdot\na \bl{u}+\na p-\Dl \bl{u}=0,\\ [5pt]
\na\cdot\bl{u}=0
\end{array}
\right.
\end{equation}
with the boundary condition
\bes
\bali
&\bl{u}\cdot \boldsymbol{\tau}\big|_{\p B}=\la +\la\dl f(\th),\q \bl{u}\cdot \boldsymbol{n}\big|_{\p B}=0,\q \bl{u} \big|_{r\rightarrow\i}=0.
\eali
\ees
We have the following result for system \eqref{ns1}.
\begin{theorem}\label{thmain1}
Assume that $f(\th)$ is a smooth $2\pi$-periodic function. Then there exist two constants $\la_0$ and $\dl_0$, independent of each other, such that for any $\la\in [\la_0,+\i)$ and $\dl\leq \dl_0$,  the system \eqref{ns1} has a solution $\bl{u}$ satisfying
\be\label{xxx}
\lt|\bl{u}\cdot\bl{e}_\th-\t{\la}r^{-1}\rt|\ls C{\la}\dl r^{-2},\q \lt|\bl{u}\cdot\bl{e}_r\rt|\leq C\s{\la}\dl r^{-2},
\ee
for some constant $\t{\la}\thickapprox \la$.  The constant $C$ is independent of $\la$. \qed

\end{theorem}

\begin{remark}
The smallness threshold $\dl_0$ is independent of the largeness threshold $\la_0$. Hence, even when the boundary data are a relatively large perturbation of a fast rigid rotation, system \eqref{ns1} still admits a (large-amplitude) solution.\end{remark}

The existence theory for \eqref{ns1} is closely related to the classical 2D exterior-domain problem of finding a solution to
\bes
\left \{
\begin {array}{ll}
\bl{u}\cdot\na \bl{u}+\na p-\Dl \bl{u}=0,\\ [5pt]
\na\cdot\bl{u}=0,\\
\bl{u}\big|_{\p B}=\bl{a}^\ast,\q \bl{u}\big|_{r\rightarrow+\i}=\la \bl{e}_1,\\
\end{array}
\right.
\ees
where $\bl{a}^\ast$ is smooth on $\p B$ and $\la$ is a constant. For general $\bl{a}^\ast$ and $\la$, this existence question was listed by Yudovich in \cite{Yudovich:2003MMJ} as one of the ``Eleven Great Problems in Mathematical Hydrodynamics'' (Problem 2):\begin{mdframed}
\textbf{Problem 2a.} A global existence theorem for a solution to the 2D problem on a viscous fluid flow past a rigid body. The velocity at infinity is assumed to be given and equal to a prescribed constant
vector.
\end{mdframed}
 We summarize several representative results.
\begin{itemize}
\item $a^\ast=0$, $\la$ small: Existence by Finn and Smith in \cite{FinnS:1967ARMA}; Existence and Uniqueness in $D$-solution class by Korobkov and Ren in \cite{KorobkovR:2021ARMA, KorobkovR:2022JMPA}.
\item General $a^\ast$, $\la=0$: no results even for the small general $a^\ast$.
\item $a^\ast=\f{\al}{r}\bl{e}_r$+small perturbation,  $\la=0$: Existence by Hillairet and Wittwer in \cite{HillairetW:2013JDE} for $\al>\s{48}$.
\end{itemize}
Compared with \cite{HillairetW:2013JDE}, our result allows larger perturbations, and may provide insight into the unresolved case of general data $a^\ast$ with $\la=0$.
 We now outline the main ideas in the proofs and introduce notation.

\begin{itemize}
\item Throughout the paper, $C_{a,b,c,...}$ denotes a positive constant depending on $a,\,b,\, c,\,...$ which may vary from line to line. We also apply $A\lesssim_{a,b,c,\cdots} B$ to denote $A\leq C_{a,b,c,...}B$, while $A \approx_{a,b,c,...}B$ means $A\leq C_{a,b,c,...}B$ and $B\leq C_{a,b,c,...}A$.

\item For a norm $\|\cdot\|$, we use $\|(f,g,\cdots)\|$ to denote $\|f\|+\|g\|+\cdots$.  For a function $f(\th,r)$ and $1\leq p,q \leq +\i$,  define
\bes
\|f\|_{L^p_\th L^q_r}:=\lt(\int^{+\i}_1 \lt(\int_{\bT} |f|^p d\th\rt)^{q/p}dr \rt)^{1/q}.
\ees
If $p=q$, we simply write it as $\|f\|_{L^p}$ or $\|f\|_{p}$. If $p=q=2$, we use $\|f\|$ to denote the $\|f\|_{L^2}$.

\item For a function $f$ depending on $(\th,r)$ or $(\th,s)$, where $s=\ln r$ is a newly introduced variable, $f_{,\th}$, $f_{,r}$, and $f_{,s}$ denote derivatives with respect to $\th$, $r$, and $s$, respectively. If no ambiguity is caused, we omit the comma and write $f_{\th}$, $f_r$, and $f_s$.
\item $f^2_{,\text{subscript}}$ denote $(f_{,\text{subscript}})^2$, which is different from $(f^2)_{,\text{subscript}}$.
\item For a $\th$ periodic function $f(\th,\cdot)$, we denote its Fourier series by
\bes
f(\th,\cdot)= f_0(\cdot)+\sum_{1\leq k\in\bN} A^f_k(\cdot)\cos k\th+ B^f_k(\cdot)\sin k\th:=f_0+\sum_{1\leq k\in\bN} f_k(\th,\cdot),
\ees
\end{itemize}
 with
 \begin{align*}
 & f_0=\f{1}{2\pi} \int_{\bT}f(\th,\cdot)d\th=\fint f(\th,\cdot)d\th,\\
 &A^f_k(\cdot)=\f{1}{\pi}\int^{2\pi}_0 f(\th,\cdot)\cos k\th d\th,\q B^f_k(\cdot)=\f{1}{\pi}\int^{2\pi}_0 f(\th,\cdot)\sin k\th d\th.
 \end{align*}
$f_0$ and $f_k$ represent Fourier zero mode and Fourier $k$ mode of $f$. $f_{\neq}=\sum_{1\leq k\in\bN} f_k(\th,\cdot)$ and $f_{h}=\sum_{2\leq k\in\bN} f_k(\th,\cdot)$ represent Fourier non-zero mode and high-frequency mode of $f$.

{\bf\noindent Strategy of proving Theorem \ref{thmain}}

To prove Theorem \ref{thmain}, we first construct an approximate solution to \eqref{nspolar} via asymptotic expansions, with the same boundary data as \eqref{nspolar}, and then estimate the remainder around this approximate profile. The strategy is summarized below.\subsubsection*{Construction of an approximation solution}

We construct an approximate solution $(u^a,v^a)$ consisting of an Euler part $(u_e^a,v_e^a)$ and a boundary-layer part $(u_p^a,v_p^a)$. The Euler part is harmonic, i.e. $(\Dl u_e^a,\Dl v_e^a)=0$, admits a Fourier representation, and its $n$-th mode decays like $r^{-(n+1)}$ as $r\to+\i$. It satisfies
\begin{align}
&\lt\|r^{2+k}\p^j_\th\p^k_r\lt(u^a_e-(A+\mathcal{O}(\e\dl))r^{-1}\rt)\rt\|_\infty+ \|r^{2+k}\p^j_\th\p^k_rv^a_e\|_\infty\leq C_{j,k}\epsilon(\dl+\epsilon). \label{eulerapprox}
\end{align}
While for the boundary layer part, near the boundary $r=1$, we have for any $j,k,\ell\in \{0\}\cup\bN$,
\begin{align}
\|\zeta^\ell\p^{j}_\th\partial_\zeta^k u_p^a\|_\infty\leq C_{j,k,\ell}(\dl+\epsilon), \ \|\zeta^\ell\p^{j}_\th\partial_\zeta^k v_p^a\|_\infty\leq C_{j,k,\ell} \epsilon(\dl+\epsilon), \label{prandtlapprox1}
\end{align}
where $\zeta:=\f{r-1}{\e}$ is the stretched variable near $r=1$. Estimates \eqref{eulerapprox} and \eqref{prandtlapprox1} are repeatedly used in the error analysis. The construction and asymptotic analysis of the approximate solution are technical and are deferred to Section \ref{secappro}.
\subsubsection*{Weighted ${H}^1$ estimate of the linearized error Navier-Stokes equations}

In the polar coordinate system \eqref{nspolar}, define the error variables by
\bes
u:=u^\epsilon-u^a,\ v:=v^\epsilon-v^a,\ p:=p^\epsilon-p^a,
\ees
 where $p^a$ is the approximate pressure. Then $(u,v,p)$ satisfies the following linear forced Navier-Stokes system

\begin{align}\label{errorequation}
\left\{
\begin{array}{ll}
-\epsilon^2\left(\frac{u_{\theta \theta}}{r}+r u_{r r}+u_r+\frac{2}{r} v_\theta-\frac{u}{r}\right)+\p_\th p +S_u=R_u, &(\th,r)\in \bT\times(1,+\i),\\[5pt]
-\epsilon^2\left(\frac{v_{\theta \theta}}{r}+r v_{r r}+v_r-\frac{2}{r} u_\theta-\frac{v}{r}\right)+r\p_rp+S_v=R_v,&(\th,r)\in \bT\times(1,+\i),\\[5pt]
\p_\th u+\p_r(rv)=0,  &(\th,r)\in \bT\times(1,+\i),\\[5pt]
( u^a,v^a)(\th,r)=( u^a,v^a)(\th+2\pi,r), &(\th,r)\in \bT\times(1,+\i), \\[5pt]
u(\th,1)=v(\th,1)=u(\th,+\i)=v(\th,+\i)=0, &\th\in \bT,
 \end{array}
\right.
\end{align}
where
\bes
\begin{aligned}
S_u:&=u^a \p_\th u +v^ar\p_ru+u \p_\th u^a+vr\p_ru^a+uv^a+vu^a,\\
S_v:&=u^a \p_\th v +v^ar\p_rv+u \p_\th v^a+v r\p_r v^a-2u u^a,
\end{aligned}
\ees
 along with $R_u$ and $R_v$ being the remainders.

We next derive weighted $\dot{H}^1$ estimates for this linearized system. A direct weighted estimate in \eqref{errorequation} is inconvenient because of the pressure term. We therefore use incompressibility to eliminate pressure and pass to a vorticity formulation.
Define the vorticity by $\o=\f{v_\th-(ru)_r}{r}$. Subtracting $\p_r$\eqref{errorequation}$_1$ from $\f{1}{r}\p_\th$\eqref{errorequation}$_2$, we obtain
\begin{align}
&-\epsilon^2 r\Dl \o+ (u^a\p_\th +rv^a\p_r)\o+  (u\p_\th +rv\p_r)\o^a=\f{1}{r}\p_\th R_v-\p_rR_u:=R_{\text{comb}}. \label{erroreqution1}
\end{align}
Although pressure is removed in \eqref{erroreqution1}, the vorticity boundary condition remains nontrivial. We therefore rewrite the equation in terms of a stream function. Define $\Phi=\int_{\red 1}^r u\,dr$. Then
\bes
\Phi_r=u,\q \Phi_\th=-(rv),\q -\Dl\Phi=\o.\q (\Phi,\Phi_r, \Phi_\th)\big|_{r={\red 1}}=0.
\ees
Then \eqref{erroreqution1} in the stream function reformulation is
\be \label{streamform}
\lt\{
\bali
&\epsilon^2 r\Dl^2 \Phi-(u^a\p_\th +rv^a\p_r)\Dl \Phi+  (u\p_\th +rv\p_r)\Dl\Phi^a=R_{\text{comb}},\\
&(\Phi,\Phi_r,\Phi_\th)\big|_{r=1}=0.
\eali
\rt.
\ee
After subtracting the Euler zero mode, the approximate solution decays like $r^{-2}$ at spatial infinity. We therefore expect the nonzero mode of the error $(u_\neq,v)$ to have the same decay. In variables $(\th,r)$, however, this decay is difficult to recover directly from \eqref{streamform}. To overcome this, we introduce the change of variables
\be
\text{Let}\q s=\ln r,\q r\in[1,+\i),\q \text{and}\q (\hat{u}, \hat{v},\hat{\o},\hat{\Phi})(\th,s)=(u,v,\o,\Phi)(\th,r)=(u,v,\o,\Phi)(\th,e^s).
\ee
In the new variables $(\th,s)$, system \eqref{errorequation} is rewritten as
\begin{align}\label{vcerrorequation}
\left\{
\begin{array}{ll}
-\epsilon^2e^{-s}\left(\Dl_s\hat{u}+2\hat{v}_\th-\hat{u}\right)+\hat{p}_\th +S_{\hat{u}}=R_{\hat{u}}, &(\th,s)\in \bT\times(0,+\i),\\[5pt]
-\epsilon^2e^{-s} \left(\Dl_s\hat{v}-2\hat{u}_\th-\hat{v}\right)+\hat{p}_s+S_{\hat{v}}=R_{\hat{v}},&(\th,s)\in \bT\times(0,+\i),\\[5pt]
\hat{u}_\th+\hat{v}_s+ \hat{v}=0,  &(\th,s)\in \bT\times(0,+\i),\\[5pt]
(\hat{u},\hat{v})(\th,s)=( \hat{u},\hat{v})(\th+2\pi,s), &(\th,s)\in \bT\times(0,+\i), \\[5pt]
\hat{u}(\th,0)=\hat{v}(\th,0)=\hat{u}(\th,+\i)=\hat{v}(\th,+\i)=0, &\th\in \bT,
 \end{array}
\right.
\end{align}
with
\bes
\begin{aligned}
S_{\hat{u}}:&=\hat{u}^a  \hat{u}_\th +\hat{v}^a \hat{u}_s+\hat{u} \hat{u}^a_\th+\hat{v}\hat{u}^a_s+\hat{u}\hat{v}^a+\hat{v}\hat{u}^a,\\
S_{\hat{v}}:&=\hat{u}^a  \hat{v}_\th +v^a\hat{v}_s+\hat{u} \hat{v}^a_\th+\hat{v} \hat{v}^a_s-2\hat{u} \hat{u}^a.
\end{aligned}
\ees
And the stream function equation \eqref{streamform} is reformulated into
{\small
\be \label{vcstreamform}
\lt\{
\bali
&\e^2 e^{-s}\lt(\Dl^2_s\hat{\Phi}-4\Dl_s\hat{\Phi}_s+4\Dl_s\hat{\Phi}\rt)-\lt[(\hat{u}^a\p_\th+\hat{v}^a\p_s)\Dl_s \hat{\Phi}-2v^a\Dl_s\hat{\Phi}\rt]-\lt[(\hat{u}\p_\th+\hat{v}\p_s)\Dl_s \hat{\Phi}^a-2\hat{v}\Dl_s\hat{\Phi}^a\rt]=e^{2s}\hat{R}_{\text{comb}},\\
&\hat{\Phi}_\th=-e^s\hat{v}, \q \hat{\Phi}_s=e^s \hat{u},\q (\hat{\Phi},\hat{\Phi}_s,\hat{\Phi}_\th)\big|_{s=0}=0.
\eali
\rt.
\ee
}
Here, $\Dl_s=\p^2_\th+\p^2_s$ and $\Dl_s \hat{\Phi}^a=(e^s \hat{u}^a)_s-(e^s \hat{v}^a)_\th$ with $(\hat{u}^a, \hat{v}^a,\hat{\o}^a,\hat{\Phi})(\th,s)=(u^a,v^a,\o^a,\Phi^a)(\th,r)=(u^a,v^a,\o^a,\Phi^a)(\th,e^s)$.

Weighted ${H}^1$ estimates for the original error are equivalent to weighted $\dot{H}^1$-$\dot{H}^2$ estimates for $\hat{\Phi}$. The $r^{-2}$ decay of nonzero modes corresponds to $e^{-2s}$ decay of $(\hat{u}_\neq,\hat{v})$. To derive $e^{-2s}$ decay for $(\hat{u}_\neq,\hat{v})$ and $e^{-s}$ decay for $\hat{u}_0$ through \eqref{vcstreamform}$_2$, we need $\dot{H}^i$ estimates ($i=1,2,3$) for $e^s\Phi_\neq$ and $\Phi_0$. We now perform weighted energy estimates for $\hat{\Phi}$.

  The weighted $\dot{H}^2$ analysis combines an $\e$-independent coercive contribution from the linear transport term $-\hat{u}^a\Dl_s\hat{\Phi}_\th$ and an $\e$-dependent weighted contribution from the diffusion term $\e^2\Dl_s^2\hat{\Phi}$. A direct $e^s$-weighted estimate for the Fourier-one mode of \eqref{vcstreamform} is too weak, so we first derive a weaker $e^{\frac12 s}$-weighted bound for this mode and then refine it.

Multiplying \eqref{vcstreamform} by $e^{3s}\hat{\Phi}_{h,\th}+e^{2s}\hat{\Phi}_{1,\th}$, where $\hat{\Phi}_{h}$ and $\hat{\Phi}_{1}$ denote the Fourier modes $k\ge2$ and $k=1$, and integrating over $\hat{\O}:=\bT\times[0,+\i)$, we obtain the estimate
\be\label{linear1}
\bali
&\|e^{s}(\hat{\Phi}_{h,\th\th},\hat{\Phi}_{h,\th s})\|^2_{L^2(\hat{\O})}+\|e^{0.5s}(\hat{\Phi}_{1,\th\th},\hat{\Phi}_{1,\th\th})\|^2_{L^2(\hat{\O})}\\
\ls& \e^2\dl\lt(\|\hat{\Phi}_{0,ss},\hat{\Phi}_{0,s}, e^s\hat{\Phi}_{h,ss}, e^{0.5s}\hat{\Phi}_{1,ss} \|^2_{L^2(\hat{\O})}\rt)+ \cdots,
\eali
\ee
where $``\cdots"$ denotes remainder terms. Here $\hat{\Phi}_{0}$ is the Fourier zero mode of $\hat{\Phi}$ in the $\th$ variable.

Control of $\hat{\Phi}_{0}$ on the right-hand side of \eqref{linear1} follows from an explicit equation for $\hat{u}_0$. Taking the $\th$-average of \eqref{vcerrorequation}$_1$, we get
\be\label{vcu0formulax}
\lt\{
\bali
&\e^2\left(\hat{u}_{0,ss}-\hat{u}_{0,s}\right)=S_{\hat{u}0}-R_{\hat{u}0},\\
& \hat{u}_{0}(0)=\hat{u}_0 (+\i)=0,
\eali
\rt.
\ee
where
\bes
S_{\hat{u}0}=\fint_\bT (v^a \hat{u}_{\neq,s}+\hat{v} \hat{u}^a_{\neq,s}+\hat{u}_{\neq}\hat{v}^a+\hat{u}^a_{\neq}\hat{v})d\th,\q R_{\hat{u}0}=\fint_\bT R_{\hat{u}} d\th.
\ees
Multiplying \eqref{vcu0formulax} by $-e^{s}\hat{\Phi}_{0,s}+e^s\hat{\Phi}_0$, integrating in $s$, and using properties of the approximate solution, we obtain
\be\label{linearx1}
\e^2\|(\hat{\Phi}_{0,ss},\hat{\Phi}_{0,s})\|^2_{L^2}\ls \|(e^{s}\hat{\Phi}_{h,\th\th},e^{s}\hat{\Phi}_{h,\th s},\hat{\Phi}_{1,\th\th},e^{0.5s}\hat{\Phi}_{1,\th\th})\|^2_{L^2(\hat{\O})}+ \cdots.
\ee

Control of nonzero modes on the right-hand side of \eqref{linear1} comes from nonzero-frequency energy estimates. Multiplying \eqref{vcstreamform} by $e^{3s}\hat{\Phi}_{h}+e^{2s}\hat{\Phi}_{1}$ and integrating over $\hat{\O}$ yields
\be\label{linearx2}
\bali
&\e^2\|(e^s\hat{\Phi}_{hss},e^{0.5s}\hat{\Phi}_{1,ss})\|^2_{L^2(\O)}\\
\ls& \e^2\dl\|(\hat{\Phi}_{0,ss},\hat{\Phi}_{0,s})\|^2_{L^2}+ \|(e^{s}\hat{\Phi}_{h,\th\th},e^{s}\hat{\Phi}_{h,\th s},\hat{\Phi}_{1,\th\th},e^{\f{1+\beta }{2}s}\hat{\Phi}_{1,\th\th})\|^2_{L^2(\hat{\O})}+ \cdots,
\eali
\ee

Finally, combining \eqref{linear1}, \eqref{linearx1}, and \eqref{linearx2}, we close the following linear $H^1$ stability estimate
\be\label{linearx3}
\bali
&\e^2\|(e^s\hat{\Phi}_{h,ss},e^{0.5s}\hat{\Phi}_{1,ss},\hat{\Phi}_{0,s},\hat{\Phi}_{0,ss})\|^2_{L^2(\O)}+\|(e^s\hat{\Phi}_{h,\th \th},e^s\hat{\Phi}_{h,\th s}, e^{0.5s}\hat{\Phi}_{1,\th s},e^{0.5s}\hat{\Phi}_{1,\th \th})\|^2_{L^2(\O)}\ls \cdots.
\eali
\ee

\subsubsection*{Weighted ${H}^2$ estimate of the linearized error Navier-Stokes equations}

To obtain $r^{-2}$-decay $L^\i$ bounds for the Fourier-one mode $(u_1,v_1)$, equivalently $e^{-2s}$ decay of $(\hat{u}_1,\hat{v}_1)$, we require refined $\dot{H}^3$ decay estimates for the Fourier-one component of $\hat{\Phi}$. This in turn requires local spatial $\dot{H}^4$ estimates for $\Phi$. We proceed in three steps.
1. As in \eqref{linearx2}, multiply \eqref{vcstreamform}$_1$ by $-\e^2\lt(e^{3s}\hat{\Phi}_{h,\th\th}+e^{2s}\hat{\Phi}_{1,\th\th}\rt)$ and by $\e^4\lt(e^{3s}\hat{\Phi}_{h,\th\th\th\th}+e^{2s}\hat{\Phi}_{1,\th\th\th\th}\rt)$, and integrate. This gives \be\label{linearx4}
 \bali
 &\e^4\|(e^s\hat{\Phi}_{h,\th\th\th},e^s\hat{\Phi}_{h,\th\th s}, e^s\hat{\Phi}_{h,\th ss},e^{0.5s}\hat{\Phi}_{1,\th\th\th},e^{0.5s}\hat{\Phi}_{1,\th\th s}, e^{0.5s}\hat{\Phi}_{1,\th ss} )\|^2_{L^2(\O)}\\
 \ls&\e^2\|(e^s\hat{\Phi}_{h,\th\th},e^s\hat{\Phi}_{h,\th s}, e^s\hat{\Phi}_{h,ss},e^{0.5s}\hat{\Phi}_{1,\th\th},e^{0.5s}\hat{\Phi}_{1,\th s}, e^{0.5s}\hat{\Phi}_{1,ss} )\|^2_{L^2(\O)}+\e^2\|(\hat{\Phi}_{0,s},\hat{\Phi}_{0,ss})\|+\cdots.
 \eali
 \ee
 and
  \be\label{linearx4x}
 \bali
 &\e^6\|(e^s\hat{\Phi}_{h,\th\th\th\th},e^s\hat{\Phi}_{h,\th\th\th s}, e^s\hat{\Phi}_{h,\th\th ss},e^{0.5s}\hat{\Phi}_{1,\th\th\th\th},e^{{0.5s}}\hat{\Phi}_{1,\th\th\th s}, e^{0.5s}\hat{\Phi}_{1,\th \th ss} )\|^2_{L^2(\O)}\\
 \ls&\e^2\|(e^s\hat{\Phi}_{h,\th\th\th},e^s\hat{\Phi}_{h,\th\th s}, e^s\hat{\Phi}_{h,\th ss},e^{0.5s}\hat{\Phi}_{1,\th\th\th},e^{0.5s}\hat{\Phi}_{1,\th\th s}, e^{0.5s}\hat{\Phi}_{1,\th ss} )\|^2_{L^2(\O)}+\e^2\|(\hat{\Phi}_{0,s},\hat{\Phi}_{0,ss})\|+\cdots.
 \eali
 \ee
However, \eqref{linearx3}, \eqref{linearx4}, and \eqref{linearx4x} do not yet imply the required $e^{-s}$ decay of the Fourier-one mode $\hat{\Phi}_1$. We therefore derive sharper decay estimates for $\hat{\Phi}_1$, relying on local spatial $\dot{H}^4$ bounds for $\Phi$.
In particular, to estimate $\hat{\Phi}_{sss}$ we exploit \eqref{vcerrorequation}$_1$.
2. Using the Bogovskii mapping lemma and a domain cutting-gluing argument, we derive a weighted $L^2$ estimate for $p_\th$, namely $\|\hat{p}_\th e^{0.5s}\|^2_{L^2(\O)}$, which yields
\bes
\e^{10}\|\Phi_{0,sss},e^{0.5s} \Phi_{\th sss} \|^2\ls  \e^6\|(\hat{\Phi}_{0,s},\hat{\Phi}_{0,ss}\|^2+\e^6\sum_{\substack{ i+j\leq 4\\i\neq0,j\leq 2 }}\|e^{0.5s}\p^i_\th\p^j_s\hat{\Phi}\|^2+\cdots.
\ees
A direct use of \eqref{vcstreamform} then gives
\bes
\e^{16}\| \Phi_{0,ssss}, e^{0.5s}\Phi_{\neq,ssss} \|^2\ls \e^{10}\lt(  \|\Phi_{0,s},\Phi_{0,ss},\Phi_{0,sss}\|^2+\sum_{\substack{i+j\leq 4\\i\neq 0,j\neq 4}}\|e^{0.5s}\p^i_\th\p^j_s\hat{\Phi}\|^2\rt)+\cdots.
\ees

  3. Multiplying \eqref{vcstreamform}$_1$ by $-e^{3s}\lt((\Dl_s\hat{\Phi}_1)_{,\th}+\Dl_s\hat{\Phi}_1\rt)$, a basic energy estimate implies
 \begin{align}
 &\e^{20} \|(\Dl_s\hat{\Phi}_1)_{,s}\|^2+ \e^{18}\| (\Dl_s\hat{\Phi}_1)_{,\th}\|^2\ls  \e^{16}\sum^4_{i=0}\|\hat{\Phi}^{(k)}_{0}\|^2+\e^{16}\sum_{\substack{0\neq i+j\leq 4 \\i\neq 0}}\|e^{0.5s}\p^i_\th\p^j_s\Phi\|^2+\cdots. \label{linearx6}
 \end{align}
  The left-hand side of \eqref{linearx6} yields $e^{-2s}$ decay of the Fourier-one mode of $(\hat{u},\hat{v})$, namely $|\hat{u}_1,\hat{v}_1|\ls e^{-2s}$.

\subsubsection*{Contraction mapping of the nonlinear error Navier-Stokes equations}

  Using the a priori weighted estimates up to third-order derivatives, a standard contraction-mapping argument yields an $L^\i$ solution with $e^{-2s}$ decay for the nonzero Fourier modes of \eqref{vcerrorequation}. Returning to \eqref{errorequation}, we obtain an $r^{-2}$-decay solution for the nonzero mode $(u_\neq,v)$.

{\bf\noindent Strategy of proving Theorem \ref{thmain1}}

This result follows directly from the scaling $\bl{u}^\la=\f{1}{\la}\bl{u}$ together with an application of Theorem \ref{thmain} to the rescaled system.
The paper is organized as follows. In Section \ref{sec2}, we present the approximate system and its asymptotic properties, derive the error equations, and prove Theorems \ref{thmain} and \ref{thmain1} using linear stability estimates and a contraction argument. Section \ref{sec3} contains the detailed proof of the linear stability estimates. Finally, Section \ref{secappro} provides the construction of the approximate solution.

\section{Preliminary and proof of the main theorems }\label{sec2}

\indent
In this section, we first collect the properties of the approximate solution and then derive the error equations. Next, we present the main linear stability estimates for the reformulated error system. Finally, by the contraction mapping theorem, we establish the solvability of the nonlinear system, which yields Theorems \ref{thmain} and \ref{thmain1}.

 We first present the following approximate solution $(u^a,v^a,p^a)$, which will be constructed in Section \ref{secappro}:

\begin{eqnarray}\label{app equationd}
\left\{
\begin{array}{ll}
u^{a} u_\theta^{a}+r v^{a} u_r^{a}+u^{a} v^{a}+p_\theta^{a}-\epsilon^2\left(\frac{u_{\theta \theta}^{a}}{r}+r u_{r r}^{a}+u_r^{a}+\frac{2}{r} v_\theta^{a}-\frac{u^{a}}{r}\right)=R_u^a, &(\th,r)\in \bT\times[1,+\i),\\[5pt]
u^{a} v_\theta^{a}+r v^{a} v_r^{a}-\left(u^{a}\right)^2+r p_r^{a}-\epsilon^2\left(\frac{v_{\theta \theta}^{a}}{r}+r v_{r r}^{a}+v_r^{a}-\frac{2}{r} u_\theta^{a}-\frac{v^{a}}{r}\right)=R_v^a, &(\th,r)\in \bT\times[1,+\i),\\[5pt]
 \p_\th u^a+\p_r(rv^a)=0,  &(\th,r)\in \bT\times[1,+\i), \\[5pt]
( u^a,v^a)(\th,r)=( u^a,v^a)(\th+2\pi,r), &(\th,r)\in \bT\times[1,+\i),\\[5pt]
u^a(\th,1)=\o+\dl f(\th), \ v^a(\th,1)=0,\, ( u^a,v^a)(\th,r)|_{r\rightarrow+\i}=0, &\th\in \bT.
\end{array}
\right.
\end{eqnarray}
 where the forcing terms $(R^a_u,R^a_v)$ satisfy the following  pointwise-in-$r$ $L^2_\theta$ estimate:
\be\label{appuvreminder}
 \left\| \p^i_\th \p^j_rR_u^a(\cdot,r)\right\|_{L^2_\theta}+ \left\| \p^i_\th \p^j_r R_v^a(\cdot,r)\right\|_{L^2_\theta}\leq C_{i,j}\epsilon^{22-j}r^{-4},\qquad 0\leq i\leq 10,\quad 0\leq j\leq 1,\quad r\geq1.
\ee
Estimate \eqref{appuvreminder} is proved in Section \ref{secappro}.
\begin{lemma}\label{lemdetail}
The constructed approximate solution $(u^a,v^a)$ can be decomposed as
\bes
u^a:=u^a_e+\chi(r)u_p+\e^{22} h(\th,r),\q v^a:=v^a_e+\chi(r)v_p,
\ees
where $\chi(r)\in C^\i_c([0,+\i))$ is a cut-off function satisfying
\bes
\chi(r)=\lt\{
\bali
&1,\q r\in \lt[1,2\rt],\\
&0,\q r\geq 3.
\eali
\rt.
\ees
The Euler part $(u^a_e,v^a_e)$ are smooth functions harmonic functions, satisfying the following converged series expansions
\bes
u^a_e=\t{\o}r^{-1}+\mathcal{O}(\e\dl)\sum\limits_{n\in\bZ}\mathcal{A}_n r^{-(|n|+1)}e^{in\th}, \q v^a_e=\mathcal{O}(\e\dl)\sum\limits_{n\in(\bZ-\{0\})}\mathcal{B}_n r^{-(|n|+1)}e^{in\th},
\ees
where $\mathcal{O}(\e\dl)$ is a constant with order $\e\dl$ and $\mathcal{A}_n,, \mathcal{B}_n$ are constants depending only on $n$.
After subtracting the zero frequency, we have for $j,k\in\bN$,
\begin{align}
&\lt\|r^{2+k}\p^{j}_\th\p^{k}_r\lt(u^a_e-(\t{\o}+\mathcal{O}(\e\dl))r^{-1}\rt)\rt\|_\infty\leq C_{j,k}\epsilon\dl,\q \|r^{2+k}\p^j_\th\p^k_rv^a_e\|_\infty\leq C_{j,k} \epsilon \dl. \label{eulerapproxd}
\end{align}
The boundary layer part $(u_p,v_p)$ satisfies, for $j,k,\ell\in\bN$,
\begin{align}
\|\zeta^\ell\p^{j}_\th\partial_\zeta^k u_p\|_\infty\leq C_{j,k,\ell}(\dl+\epsilon), \ \|\zeta^\ell\p^{j}_\th\partial_\zeta^k {v}_p\|_\infty\leq C_{j,k,\ell} \epsilon(\dl+\epsilon). \label{prandtlapprox1d}
\end{align}
While $h(\th,r)$ satisfies
\be\label{happrox}
h(\th,1)=0,\q \|r^{100+k}\p^j_\th\p^k_r h(\th,r)\|_{L^\i}\leq C_{j,k}.
\ee
\qed
\end{lemma}
\begin{lemma} \label{lemdetail1}
By using the boundary conditions  in \eqref{app equationd} and the asymptotic behaviors in \eqref{eulerapproxd}, \eqref{prandtlapprox1d}, \eqref{happrox} satisfied by $(u^a,v^a)$, we have the following estimates.

For $j,k\in \bN$, $i\in\bZ$ and $i+k\geq 0$, we have
\be\label{appdetail0}
\bali
&  \lt|u^a-(\t{\o}+\mathcal{O}(\e\dl))r^{-1}\rt|\ls \dl r^{-2}, \\
&\lt|\p^{j}_{\th}\p^k_r \lt(u^a-u^a_e\rt)\rt|\ls \f{\e^{i}\dl}{(r-1)^{i+k}r^{10}},
\eali
\ee
and for $j,k\in \bN$, $i\in\bZ$ and $i+k\geq 0$,
\be\label{appdetail11}
\lt|\p^{j}_{\th}\p^k_r (v^a-v^a_e)\rt|\ls_{j,k} \f{\e^{i}\e\dl}{(r-1)^{i+k}r^{10}}.
\ee
\end{lemma}

\begin{proof}
From \eqref{eulerapproxd} and \eqref{prandtlapprox1d} in Lemma \ref{lemdetail}, we see that
\begin{align}
\lt|u^a-(\t{\o}+O(\e\dl))r^{-1}\rt| \ls {\e\dl}{r^{-2}}+ \chi(r)|u_p|+\e^{22}|h|\ls {\dl}{r^{-2}}.
\end{align}
 This is the first one of \eqref{appdetail0}.

Also from  Lemma \ref{lemdetail}, we have
\begin{align}
\lt|\p^{j}_{\th}\p^k_r\lt(u^a-u^a_e\rt)\rt| \leq C_{j,k} |\p^{j}_{\th}\p^k_r(\chi(r)u_p)|+\e^{22}|\p^{j}_{\th}\p^k_r h|, \label{detailesti1}
\end{align}
From \eqref{prandtlapprox1d}, by using Leibniz formula, we see that
\begin{align}
|\p^{j}_{\th}\p^k_r(\chi(r)u_p)|\ls& C_{j,k} |\chi(r)\p^{j}_{\th}\p^k_ru_p|+C_{j,k} \sum^k_{\t{k}=1}| \chi^{(\t{k})}(r)\p^{k-\t{k}}_ru_p|\nn\\
\ls&  C_{j,k} \lt |\p^{j}_{\th} \zeta^{k}\f{1}{(r-1)^k}\p^k_\zeta u_p\rt|+C_{j,k} \sum^k_{\t{k}=1}|{\bf1}_{\text{spt}\chi'}\zeta^{k-\t{k}}\p^{k-\t{k}}_\zeta u_p|\nn\\
=& C_{j,k}  \lt|\chi(r) \f{\e^i}{(r-1)^{k+i}}\zeta^{i+k} \p^{j}_{\th}\p^k_\zeta u_p\rt|+C_{j,k} \sum^k_{\t{k}=1}\lt|{\bf1}_{\text{spt}\chi'}\zeta^{k-\t{k}}\p^{k-\t{k}}_\zeta u_p\rt|\nn\\
\ls & C_{j,k}\f{\e^i\dl}{(r-1)^{k+i}r^{10}}. \label{detailesti2}
\end{align}
Inserting \eqref{detailesti2} into \eqref{detailesti1}, we obtain the second one of \eqref{appdetail0}.

The estimate of \eqref{appdetail11} is the same as the second one of \eqref{appdetail0}, we omit the details.
\end{proof}

From Lemma \ref{lemdetail} and Lemma \ref{lemdetail1}, we have the following estimate for the variables-change approximate solution.
\begin{corollary} \label{vclemdetail} Let $s=\ln r$. Denote $\hat{u}^a(\th,s)=u^a(e^s, \th)=u^a(r, \th)$ and $\hat{v}^a(\th,s)=v^a(e^s, \th)=v^a(r, \th)$. Then we have

For $j,k\in \bN$, $i\in\bZ$ and $i+k\geq 0$, we have
\be\label{vcappdetail0}
\bali
& |\hat{u}^a(\th,s)-(\t{\o}+\mathcal{O}(\e\dl))e^{-s}|\ls \dl e^{-s}, \\
&\hat{u}^a_e=\t{\o}e^{-s}+ \mathcal{O}(\e\dl)\sum\limits_{n\in\bZ}\mathcal{A}_n e^{-(|n|+1)s} e^{in\th},\q \lt|\p^{j}_{\th}\p^k_s \lt(\hat{u}^a_e-(\t{\o}+\mathcal{O}(\e\dl) \mathcal{A}_0)e^{-s}\rt)\rt|\ls \e\dl e^{-2s},\\
& \lt|\p^{j}_{\th}\p^k_s \lt(\hat{u}^a(\th,s)-\hat{u}^a_e\rt)\rt|\ls \f{\e^{i}\dl}{s^{i+k}}e^{(k-10)s},\\
&\lt|\p^{j}_{\th}\p^k_s \lt(\hat{u}^a-(\t{\o}+\mathcal{O}(\e\dl) \mathcal{A}_0)e^{-s}\rt)\rt|\ls \e\dl e^{-2s}+\f{\e^{i}\dl}{s^{i+k}}e^{(k-10)s}.
\eali
\ee
and for $j,k\in \bN$, $i\in\bZ$ and $i+k\geq 0$,
\be\label{vcappdetail1}
\bali
&\hat{v}^a_e=\mathcal{O}(\e\dl)\sum\limits_{n\in(\bZ-\{0\})}\mathcal{B}_n e^{-(|n|+1)s} e^{in\th},\q \lt|\p^{j}_{\th}\p^k_s \hat{v}^a_e\rt|\ls \e\dl e^{-2s}\\
&\lt|\p^{j}_{\th}\p^k_s \lt(\hat{v}^a(\th,s)-\hat{v}^a_e(\th,s)\rt)\rt|\ls_{j,k}  \f{\e^{i}\e\dl}{s^{i+k}}e^{(k-10)s},\\
&\lt|\p^{j}_{\th}\p^k_s \hat{v}^a\rt|\ls_{j,k}   \e\dl e^{-2s}+ \f{\e^{i}\e\dl}{s^{i+k}}e^{(k-10)s}
\eali
\ee

\end{corollary}

\pf  The proof is a direct consequence of Lemma \ref{lemdetail} and Lemma \ref{lemdetail1}. We omit the details. \qed

 From \eqref{appuvreminder}, by defining $(R^a_{\hat{u}},R^a_{\hat{v}}):=(R^a_{u},R^a_{v})(\th,e^s)$, we obtain
 \bes
  \left\| \p^i_\th R_{\hat{u}}^a(\cdot,s)\right\|_{L^2_\theta}+ \left\| \p^i_\th R_{\hat{v}}^a(\cdot,s)\right\|_{L^2_\theta}\leq C_i\epsilon^{22}e^{-4s},\qquad 0\leq i\leq10.
 \ees
By  cancelling the pressure in \eqref{app equationd}, the approximate vorticity $\o^a:=\f{v^a_\th-(ru^a)_r}{r}$ satisfies
\bes
(u^a \p_\th+ rv^a\p_r) \o^a -\e^2r\Dl \o^a=\f{1}{r}\p_\th R^a_v-\p_r R^a_u.
\ees
From our construction in Section \ref{secappro}, the approximate vorticity can be decomposed as
\bes
\o^a=\o^a_p+ \o^a_e, \text{ with } \o^a_e=\f{v^a_{e,\th}-(ru^a_e)_{,r}}{r}
\ees
where $\o^a_e$ is a constant by the harmonic properties of $(u^a_e,v^a_e)$.  From the representation of the approximate solution, $\o^a_p$ is smooth and supported in $\{1\leq r\leq 3\}$. So, we have
 \bes
R^a_{\o}:=\f{1}{r}\p_\th R^a_v-\p_r R^a_u=(u^a \p_\th+ rv^a\p_r) \o^a_p -\e^2r\Dl \o^a_p,
 \ees
is supported in $\{1\leq r\leq 3\}$. Combining this with \eqref{appuvreminder},  we have
\be\label{vcremainder1}
\left\|R^a_{\o}(\cdot,r)\right\|_{L^2_\theta}
=\left\|\left(\f{1}{r}\p_\th R^a_v-\p_r R^a_u\right)(\cdot,r)\right\|_{L^2_\theta}
\lesssim C\f{\e^{21}}{r^{100}},\qquad 1\leq r\leq3.
\ee
Set the error function by
\bes
u:=u^\epsilon-u^a,\ v:=v^\epsilon-v^a,\ p:=p^\epsilon-p^a,
\ees
 Then we have
\begin{align}\label{errorequ}
\left\{
\begin{array}{ll}
-\epsilon^2\left(\frac{u_{\theta \theta}}{r}+r u_{r r}+u_r+\frac{2}{r} v_\theta-\frac{u}{r}\right)+\p_\th p +S_u=R_u, &(\th,r)\in \bT\times[1,+\i),\\[5pt]
-\epsilon^2\left(\frac{v_{\theta \theta}}{r}+r v_{r r}+v_r-\frac{2}{r} u_\theta-\frac{v}{r}\right)+r\p_rp+S_v=R_v,&(\th,r)\in \bT\times[1,+\i),\\[5pt]
\p_\th u+\p_r(rv)=0,  &(\th,r)\in \bT\times[1,+\i),\\[5pt]
( u,v)(\th,r)=( u,v)(\th+2\pi,r), &(\th,r)\in \bT\times[1,+\i), \\[5pt]
u(\th,1)=v(\th,1)=u(\th,+\i)=v(\th,+\i)=0, &\th\in \bT,
 \end{array}
\right.
\end{align}
where
\bes
\begin{aligned}
S_u:&=u^a \p_\th u +v^ar\p_ru+u \p_\th u^a+vr\p_ru^a+uv^a+vu^a,\\
S_v:&=u^a \p_\th v +v^ar\p_rv+u \p_\th v^a+v r\p_r v^a-2u u^a,
\end{aligned}
\ees
and the remainders
\bes
R_u:=R_u^a-(u\p_\th+v r\p_r)u-uv,\ R_v:=R_v^a-(u\p_\th+v r\p_r)v+u^2.
\ees

In order to give the solvability of system \eqref{errorequ} and its $r^{-2}$ asymptotic behavior at the spatial infinity, we will first derive a weighted $H^3$ energy estimates in the domain $\O$. In order to avoid the trouble caused by the pressure, it is more convenient to reformulate our system in the vorticity form and eliminate the pressure.

In polar coordinates, the vorticity $\o$ is defined by
\bes
\o=\f{v_\th-(ru)_r}{r}.
\ees
From the incompressibility and the fact $\int^{2\pi}_0 rvd\th=0$, there exists a $\th$-periodic stream function $\Phi$ such that
\bes
\Phi_r=u,\q \Phi_\th=-rv,
\ees
Actually we can take
\bes
\Phi(\th,r):=\int^r_{1} u(\th,\bar{r})d\bar{r},
\ees
which satisfies $\Phi(\th,1)=\Phi_\th(\th,1)=0$. A direct calculation implies that $-\Dl\Phi=\o$. Subtracting $\p_r$\eqref{errorequ}$_1$ from $\f{1}{r}\p_\th$\eqref{errorequ}$_2$, we obtain that the vorticity $\o$ satisfies the following equation:

\begin{align}
-\epsilon^2r\Dl\o+(u^a\p_\th+rv^a\p_r)\o+(u\p_\th+rv\p_r)\o^a=&\f{1}{r}\p_\th R^a_v-\p_rR^a_u-(u\p_\th+rv\p_r)\o\nn\\
                                                                     &:=R^a_{\o}+R_{\o}. \label{errorequcut2}
\end{align}

Define the zero frequency and non-zero frquency of a $\th$-periodic function $f(\th,r)$ by
 \bes
 f_0=\f{1}{2\pi}\int^{2\pi}_0 f(\th,r)d\th, \q f_{\neq}=f-f_0.
 \ees
In our a priori setting, $rv=o(1)$ as $r\rightarrow+\i$, then from the incompressibility and $v(\th,1)=0$,
\bes
\p_\th\Phi(\th,+\i)=\int^1_{+\i} u_\th(\th,r)dr=-\int^1_{+\i} (rv)_rdr=-v(\th,1)=0,
\ees
 which implies that $\Phi(\th,\i)=\text{constant}$. We therefore have the following boundary conditions for $\Phi$:
\bes
\Phi|_{r=1}=0,\q \Phi(\th,\i)=\text{constant}, \q \p_\th \Phi|_{r=1 \text{ and } +\i}=0,\q \p_r\Phi|_{r=1\text{ and }+\i}=0.
\ees
and noting that $\Phi_{\neq}(\th,\i)=\int^\i_{1} (u-u_0)dr=0$, we have the following boundary condition for $\Phi_{\neq}$,
\bes
\Phi_{\neq}|_{r=1 \text{ and } +\i}=0,\q \p_r\Phi_{\neq}|_{r=1 \text{ and } +\i}=u_{\neq}|_{r=1 \text{ and } +\i}=0,\q \p_\th\Phi_{\neq}|_{r=1 \text{ and } +\i}=0.
\ees

Since the optimal decay rate of the nonzero Fourier mode of the approximate solution is $r^{-2}$, which comes from the Fourier 1-mode of the approximate Euler solution, we expect the error function to have the same decay order. To recover this optimal decay, we introduce the following change of variables and unknowns, whose linear structure is more convenient for later estimates. Let $s=\ln r$. Define
\be
\hat{\Phi}(\th,s)=\Phi(\th,e^s)=\Phi(\th,r), \q (\hat{u},\hat{v})(\th,s)=(u,v)(\th,e^s)=(u,v)(\th,r),\q \hat{\o}(\th,s)=\o(\th,e^s)=\o(\th,r).
\ee
In new coordinates $(\th,s)$, we have
\bes
r\p_r=\p_s,\, \hat{\Phi}_s=e^s \hat{u},\, \hat{\Phi}_\th=-e^s \hat{v},\, \Dl=e^{-2s}(\p^2_\th+\p^2_s):=e^{-2s}\Dl_s.
\ees
The system \eqref{errorequ}  is changed to
\begin{align}\label{vcerrorequuv}
\left\{
\begin{array}{ll}
-\epsilon^2e^{-s}\left(\Dl_s\hat{u}+2\hat{v}_\th-\hat{u}\right)+\hat{p}_\th +S_{\hat{u}}=R_{\hat{u}}, &(\th,s)\in \bT\times(0,+\i),\\[5pt]
-\epsilon^2e^{-s} \left(\Dl_s\hat{v}-2\hat{u}_\th-\hat{v}\right)+\hat{p}_s+S_{\hat{v}}=R_{\hat{v}},&(\th,s)\in \bT\times(0,+\i),\\[5pt]
\hat{u}_\th+\hat{v}_s+ \hat{v}=0,  &(\th,s)\in \bT\times(0,+\i),\\[5pt]
(\hat{u},\hat{v})(\th,s)=( \hat{u},\hat{v})(\th+2\pi,s), &(\th,s)\in \bT\times(0,+\i), \\[5pt]
\hat{u}(\th,0)=\hat{v}(\th,0)=\hat{u}(\th,+\i)=\hat{v}(\th,+\i)=0, &\th\in \bT,
 \end{array}
\right.
\end{align}
where
\bes
\begin{aligned}
S_{\hat{u}}:&=\hat{u}^a  \hat{u}_\th +\hat{v}^a \hat{u}_s+\hat{u} \hat{u}^a_\th+\hat{v}\hat{u}^a_s+\hat{u}\hat{v}^a+\hat{v}\hat{u}^a,\\
S_{\hat{v}}:&=\hat{u}^a  \hat{v}_\th +v^a\hat{v}_s+\hat{u} \hat{v}^a_\th+\hat{v} \hat{v}^a_s-2\hat{u} \hat{u}^a,
\end{aligned}
\ees
and the remainders
\bes
R_{\hat{u}}:=R_u^a(\th,e^s)-(\hat{u}\p_\th+\hat{v} \p_s)\hat{u}-\hat{u}\hat{v},\ R_{\hat{v}}:=R_v^a(\th,e^s)-(\hat{u}\p_\th+\hat{v} \p_s)\hat{v}+\hat{u}^2.
\ees
The boundary condition for $\hat{\Phi}$ is changed to
\bes
\hat{\Phi}|_{s=0}=0,\q \hat{\Phi}(\th,\i)=\text{constant}, \q \p_\th \Phi|_{s=0 \text{ and } +\i}=0,\q \p_s\Phi|_{s=0\text{ and }+\i}=0.
\ees
By defining
\bes
 \hat{R}^a_\o={R}^a_{\o}(\th,e^s), \text{ and }  \hat{R}_\o=R_\o(\th,e^s):=-(\hat{u}\p_\th+\hat{v}\p_s)\hat{\o},
\ees
 then from \eqref{errorequcut2}, we have
\begin{align}
&\e^2 e^{-s}\lt(\Dl^2_s\hat{\Phi}-4\Dl_s\hat{\Phi}_s+4\Dl_s\hat{\Phi}\rt)\nn\\
&-\lt[(\hat{u}^a\p_\th+\hat{v}^a\p_s)\Dl_s \hat{\Phi}-2v^a\Dl_s\hat{\Phi}\rt]-\lt[(\hat{u}\p_\th+\hat{v}\p_s)\Dl_s \hat{\Phi}^a-2\hat{v}\Dl_s\hat{\Phi}^a\rt]\nn\\
&=e^{2s}\hat{R}^a_{\o}+e^{2s}\hat{R}_\o:=e^{2s}\hat{R}_{\text{comb}}. \label{vcerrorequ}
\end{align}
From \eqref{vcremainder1}, $\hat{R}^a_{\o}$ satisfies
\be\label{appomegahat}
 \left\|\hat{R}^a_{\o}(\cdot,s)\right\|_{L^2_\theta}\ls C\e^{21} e^{-100s},\qquad 0\leq s\leq\log 3.
\ee
The main estimate for the linear system \eqref{vcerrorequuv} or \eqref{vcerrorequ} is the following.
\begin{proposition}\label{proplinearstability}
 Let $(\hat{u},\hat{v})$ be a smooth solution of (\ref{vcerrorequuv}) that decays sufficiently fast at infinity. Then there exist $\epsilon_0>0$ and $\dl_0>0$ such that, for any $\epsilon\in(0,\epsilon_0)$ and $\dl\in(0,\dl_0)$, there hold

\noindent i), The main weighted derivatives' estimates:
\begin{align}\label{linearstability}
&\e^6\|(e^s\hat{\Phi}_{h,\th\th\th\th},e^s\hat{\Phi}_{h,\th\th\th s},e^s\hat{\Phi}_{h,\th\th ss},e^{0.5s}\hat{\Phi}_{1,\th\th\th\th},e^{0.5s}\hat{\Phi}_{1,\th\th\th s}, e^{0.5s}\hat{\Phi}_{1,\th\th ss},\hat{\Phi}_{0,ss}, \hat{\Phi}_{0,s} )\|^2_{L^2(\hat{\O})}\nn\\
\ls& \e^{-2}\|e^{4s}\hat{R}_{\text{comb}}\|^2_{L^2{(\hat{\O})}}+\e^{-2}\|e^{1.5s} R_{\hat{u}}\|^2_{L^2{(\hat{\O})}}.
\end{align}
\noindent ii) The third and fourth order normal derivatives' estimates:
\begin{align}
&\e^{16}\|  \Phi_{0,sss}, \Phi_{0,ssss}, e^{0.5s}\Phi_{\th sss}, e^{0.5s}\Phi_{\neq,ssss} \|^2\nn\\
\ls& \e^{6}\lt(\|\Phi_{0,s},\Phi_{0,ss}\|^2+\|e^{0.5s}(\hat{\Phi}_{\th\th\th\th},\hat{\Phi}_{\th\th\th s},\hat{\Phi}_{\th\th s s} )\|^2\rt)+\e^6\|e^{2.5s}(\p_\th R_{\hat{u}},\p_\th R_{\hat{v}}, R_{\hat{u}})\|^2_{L^2(\hat{\O})}\nn\\
  &+\e^{12}\|e^{3.5s}\hat{R}_{\text{comb}}\|^2_{L^2(\hat{\O})}.\label{linearstability1}
\end{align}
\noindent iii) The Refined decay estimates for the Fourier-one mode:
\begin{align}
&\e^{20}\| (e^s \Dl_s\hat{\Phi}_1)_{,s}\|^2_{L^2{\hat{\O}}}+ \e^{18}\| (e^s \Dl_s\hat{\Phi}_1)_{,\th}\|^2_{L^2{(\hat{\O})}}\nn\\
\ls & \e^{16}\sum^4_{i=0}\|\hat{\Phi}^{(k)}_{0}\|^2_{L^2{(\hat{\O})}}+\e^{16}\sum_{\substack{0\neq i+j\leq 4 \\i\neq 0}}\| e^{0.5s}\p^i_\th\p^j_s\Phi\|^2_{L^2{(\hat{\O})}}+\e^{18}\|e^{4s}\hat{R}_{\text{comb}}\|^2_{L^2{(\hat{\O})}}.\label{linearstability2}
\end{align}
\end{proposition}
  \qed

Proof of Proposition \ref{proplinearstability} contains three parts: the first one is the main linear estimate in \eqref{linearstability}, which will be presented in Section \ref{subsec3.1} and Section \ref{subsec3.2.1}. Actually, the linear estimate \eqref{linearstability} is enough for closing the nonlinear system and obtain an existence result for the nonlinear system \eqref{vcerrorequuv}. However, From this a prior estimate, we can only obtain that the high frequency Fourier mode of the solution decay as $e^{-2s}$ as $s\rightarrow+\i$, i.e. $|(\hat{u}_{h},\hat{v}_h)|\ls e^{-2s}$ by Sobolev embedding, which correspond to as $r^{-2}$ decay for $({u}_{h},{v}_h)$. The Fourier-one mode only have $e^{-1.5s}$ order decay. In order to obtain the optimal $e^{-2s}$ decay for $\hat{u}_{1},\hat{v}_1$. A refined weighted estimate \eqref{linearstability2} is performed in Section \ref{subsec3.3}. Unfortunately, such refined estimate will produce some slowly decayed term up to the fourth order derivatives' of the stream function $\hat{\Phi}$. In order to control these slowly decayed term, we need some weighted estimates of the stream function $\hat{\Phi}$ up to fourth derivative, which is the stability \eqref{linearstability1}, which is realized in Section \ref{subsec3.2.2} and Section \ref{subsec3.2.3}.

\subsection{Proof of Theorem \ref{thmain}}
\indent

For system \ref{errorequ}, we have the following proposition.
\begin{proposition}\label{propexistenceuv}
There exist $\epsilon_0>0,\dl_0>0$ such that for any $\epsilon\in (0,\epsilon_0), \dl\in (0,\dl_0)$, the error equations (\ref{errorequ}) have a unique solution $(\hat{u},\hat{v})$ which satisfies
\bes
\lt\|r^2\lt(u-{u}_0,{v}\rt)\rt\|_\infty\leq C\e^{4},
\ees
where $u_0=\fint u(\th,r)d\th$ is the Fourier-zero mode of $u$, satisfying $|u_0|\leq C\e^4 r^{-1}$.
\end{proposition}

We first apply the contraction mapping theorem to prove the existence and decay estimate for the error equations (\ref{vcerrorequuv}). Then by changing the variables back from $(\th,s)$ to $(\th,r)$, we can give the the existence and decay estimate for the error equations (\ref{errorequ}). The result in Proposition \ref{propexistenceuv} is a direct consequence of the following Proposition.

\begin{proposition}\label{vcpropexistence}
There exist $\epsilon_0>0,\dl_0>0$ such that for any $\epsilon\in (0,\epsilon_0), \dl\in (0,\dl_0)$, the error equations (\ref{vcerrorequuv}) have a unique solution $(\hat{u},\hat{v})$ which satisfies
\be\label{vcerrorestix}
\lt\|e^{2s}\lt(\hat{u}-\hat{u}_0,\hat{v}\rt)\rt\|_\infty\leq C\e^{4},
\ee
where $\hat{u}_0=\fint \hat{u}(\th,s)d\th$ is the Fourier-zero mode of $\hat{u}$, satisfying $|\hat{u}_0|\leq C\e^4 e^{-s}$.
\end{proposition}
{\bf\noindent Proof of Proposition \ref{vcpropexistence}}
\begin{proof}
Define a smooth function class $\mathcal{C}:=\{(\t{\hat{u}},\t{\hat{v}})\in \mathcal{C}\}$, where $(\t{\hat{u}},\t{\hat{v}})$ satisfies
\begin{align}\label{iterative conditon}
\left\{
\begin{array}{ll}
 \t{\hat{u}}_\th+\t{\hat{v}}_s+ \t{\hat{v}}=0,\\[5pt]
(\t{\hat{u}},\t{\hat{v}})(\th,s)=(\t{\hat{u}},\t{\hat{v}})(\th+2\pi,s),\\[5pt]
(\t{\hat{u}},\t{\hat{v}})(\th,0)=(\t{\hat{u}},\t{\hat{v}})(\th,+\i)=0.
\end{array}
\right.
\end{align}
From \eqref{iterative conditon}$_1$, by defining the stream function $\t{\hat{\Phi}}:=\int^s_0 e^{\t{s}} \t{\hat{u}}(\th,\t{s})d\t{s}$, we have
\bes
\t{\hat{\Phi}}_s=e^s\t{\hat{u}},\q \t{\hat{\Phi}}_\th=-e^s\t{\hat{v}}.
\ees
Then we consider the following linear problem:
\begin{align*}
\left\{
\begin{array}{ll}
-\epsilon^2e^{-s}\left(\Dl_s\hat{u}+2\hat{v}_\th-\hat{u}\right)+\hat{p}_\th +S_{\hat{u}}=R_{\t{\hat{u}}}, &(\th,s)\in \bT\times(0,+\i),\\[5pt]
-\epsilon^2e^{-s} \left(\Dl_s\hat{v}-2\hat{u}_\th-\hat{v}\right)+\hat{p}_s+S_{\hat{v}}=R_{\t{\hat{v}}},&(\th,s)\in \bT\times(0,+\i),\\[5pt]
\hat{u}_\th+\hat{v}_s+ \hat{v}=0,  &(\th,s)\in \bT\times(0,+\i),\\[5pt]
(\hat{u},\hat{v})(\th,s)=( \hat{u},\hat{v})(\th+2\pi,s), &(\th,s)\in \bT\times(0,+\i), \\[5pt]
\hat{u}(\th,0)=\hat{v}(\th,0)=\hat{u}(\th,+\i)=\hat{v}(\th,+\i)=0, &\th\in \bT,
 \end{array}
\right.
\end{align*}
where
\bes
\begin{aligned}
S_{\hat{u}}:&=\hat{u}^a  \hat{u}_\th +\hat{v}^a \hat{u}_s+\hat{u} \hat{u}^a_\th+\hat{v}\hat{u}^a_s+\hat{u}\hat{v}^a+\hat{v}\hat{u}^a,\\
S_{\hat{v}}:&=\hat{u}^a  \hat{v}_\th +v^a\hat{v}_s+\hat{u} \hat{v}^a_\th+\hat{v} \hat{v}^a_s-2\hat{u} \hat{u}^a,
\end{aligned}
\ees
and the remainders
\bes
R_{\t{\hat{u}}}:=R_u^a(\th,e^s)-(\t{\hat{u}}\p_\th+\t{\hat{v}} \p_s)\t{\hat{u}}-\t{\hat{u}}\t{\hat{v}},\ R_{\t{\hat{v}}}:=R_v^a(\th,e^s)-(\t{\hat{u}}\p_\th+\t{\hat{v}} \p_s)\t{\hat{v}}+\t{\hat{u}}^2.
\ees

From the a priori linear stability estimates in \eqref{linearstability}, \eqref{linearstability1} and \eqref{linearstability2}, we define the following energy space for $\hat{\Phi}:=\int^s_0 e^{\t{s}}\hat{u}(\th,\t{s})d\t{s}$:
\begin{align}
\|\hat{\Phi}\|^2_{E}:=&\e^{20} \| (e^{s} \Dl_s\hat{\Phi}_1)_{,s}\|^2+\e^{18}\| (e^{s} \Dl_s\hat{\Phi}_1)_{,\th}\|^2+\e^{16}\|  \Phi_{0,sss}, \Phi_{0,ssss}, e^{0.5s}\Phi_{\th sss}, e^{0.5s}\Phi_{\neq,ssss} \|^2\nn\\
                                         +& \e^6\|(e^s\hat{\Phi}_{h,\th\th\th\th},e^s\hat{\Phi}_{h,\th\th\th s},e^s\hat{\Phi}_{h,\th\th ss},e^{0.5s}\hat{\Phi}_{1,\th\th\th\th},e^{0.5s}\hat{\Phi}_{1,\th\th\th s}, e^{0.5s}\hat{\Phi}_{1,\th\th ss},\hat{\Phi}_{0,ss}, \hat{\Phi}_{0,s} )\|^2.
\end{align}
According to the definition of the energy space $\|\hat{\Phi}\|^2_{E}$ and Sobolev embedding, we have the following $L^\i$ estimates.

{\bf\noindent $L^\i$ of the Fourier-zero mode} Directly by Sobolev embedding, we have
\begin{align}
&\|e^s\hat{u}_0\|_{L^\i}=\|\hat{\Phi}_{0,s}\|_{L^\i}\ls \|\hat{\Phi}_{0,s}\|_{L^2}+\|\hat{\Phi}_{0,ss}\|_{L^\i}\ls \e^{-6}\|\hat{\Phi}\|^2_{E}. \label{0modeinfinity}
\end{align}
{\bf\noindent $L^\i$ of the Fourier-one mode} By direct calculation, we can see that
\be\label{1moderelation}
(e^{2s}v_{1,s})_{,s}=-e^{s}(\Dl_s\hat{\Phi}_1)_{,\th},\q (e^{2s}u_{1,s})_{,s}=e^s\Dl_s\hat{\Phi}_{1,s}=(e^s\Dl_s\hat{\Phi}_1)_{,s}-e^s\Dl_s\hat{\Phi}_1.
\ee
From the definition of Fourier-one mode, by integration by parts and H\"{o}lder inequality, we see that
\begin{align}
&|e^{2s}A^{\hat{v}}_1(s)|= e^{2s} \lt|\int^\i_s \lt(A^{\hat{v}}_1\rt)_sd\t{s}\rt|=\lt|e^{2s}\int^\i_s e^{-2\t{s}}e^{2\t{s}}\lt(A^{\hat{v}}_1\rt)_s d\t{s}\rt|\nn\\
            =&\lt|-e^{2s}\f{1}{2}\int^\i_s e^{2\t{s}}\lt(A^{\hat{v}}_1\rt)_sde^{-2\t{s}}\rt|=\lt|e^{2s}\f{1}{2}\int^\i_s \lt(e^{2\t{s}}\lt(A^{\hat{v}}_1\rt)_s\rt)_{,s} e^{-2\t{s}}ds\rt|\nn\\
            \ls& \lt(\int^\i_s \lt|\lt(e^{2\t{s}}\lt(A^{\hat{v}}_1\rt)_s\rt)_{,s}\rt|^2d\t{s}\rt)^{1/2} \ls \lt\|(e^{2{s}}\hat{v}_{1,{s}})_{,s}\rt\|_{L^2(\hat{\O})} =\lt\|e^{s}(\Dl_s\hat{\Phi}_1)_{,\th}\rt\|_{L^2(\hat{\O})}\ls \e^{-18}\|\hat{\Phi}\|^2_{E}. \label{1modeinfinity}
\end{align}
Similarly, we obtain that $|e^{2s}B^{\hat{v}}_1(s)|\ls \e^{-18}\|\hat{\Phi}\|^2_{E}$. Hence,
\be
|e^{2s}\hat{v}_1(\th, s)|\ls \e^{-18}\|\hat{\Phi}\|^2_{E}.  \label{1modeinfinity1}
\ee

 Next, as in \eqref{1modeinfinity}, using \eqref{1moderelation} and the Poincar\'{e} inequality in the $\th$ variable, we obtain

\begin{align}
 |e^{2s}A^{\hat{u}}_1(s)|\ls& \lt\|\lt(e^{2s}\hat{u}_{1,s}\rt)_{,s}\rt\|_{L^2(\hat{\O})}\ls \lt\|(e^s\Dl_s\hat{\Phi}_1)_{,s}\rt\|_{L^2(\hat{\O})}+\lt\|e^s\Dl_s\hat{\Phi}_1\rt\|_{L^2(\hat{\O})}\nn\\
 \ls& \lt\|(e^s\Dl_s\hat{\Phi}_1)_{,s}\rt\|_{L^2(\hat{\O})}+\lt\|(e^s\Dl_s\hat{\Phi}_1)_{,\th}\rt\|_{L^2(\hat{\O})}\ls \e^{-20}\|\hat{\Phi}\|^2_{E}. \nn
\end{align}
 Similarly, we obtain that $|e^{2s}B^{\hat{u}}_1(s)|\ls \e^{-20}\|\hat{\Phi}\|^2_{E}$. The above two estimates imply that
\begin{align}
|e^{2s}\hat{u}_1|\ls \e^{-20}\|\hat{\Phi}\|^2_{E}.\label{1modeinfinity2}
\end{align}

{\bf\noindent $L^\i$ of the high frequency Fourier-zero mode} Direct Sobolev embedding, we see that
\begin{align}
&\|e^{2s} \hat{u}_{h}\|_{L^\i}=\|e^s \hat{\Phi}_{h,s}\|_{L^\i}\ls \|e^s (\hat{\Phi}_{h,\th s},\hat{\Phi}_{h,s s},\hat{\Phi}_{h,\th ss})\|_{L^2}\ls  \e^{-6}\|\hat{\Phi}\|^2_{E}\nn\\
&\|e^{2s} \hat{v}_{h}\|_{L^\i}=\|e^s \hat{\Phi}_{h,\th}\|_{L^\i}\ls \|e^s (\hat{\Phi}_{h,\th \th},\hat{\Phi}_{h,\th s},\hat{\Phi}_{h,\th \th s})\|_{L^2}\ls  \e^{-6}\|\hat{\Phi}\|^2_{E}.\label{hmodeinfinity}
\end{align}

By the linear stability estimate in Proposition \ref{proplinearstability}, we deduce that there exist $\epsilon_0>0$ and $\dl_0>0$ such that, for any $\epsilon\in(0,\epsilon_0)$ and $\dl\in(0,\dl_0)$, there holds
\begin{align}
&\|\hat{\Phi}\|^2_{E}\leq  \e^{-2}\|e^{4s}\t{\hat{R}}_{\text{comb}}\|^2+\e^{-2}\|e^{2.5s}(\p_\th R_{\t{\hat{u}}},\p_\th R_{\t{\hat{v}}},R_{\t{\hat{u}}})\|^2_{L^2(\hat{\O})},\label{energyfinal0}
\end{align}
where
\begin{align}
\t{\hat{R}}_{\text{comb}}=&\hat{R}^a_{\o}-(\t{\hat{u}}\p_\th+\t{\hat{v}}\p_s)\t{\o}=\hat{R}^a_{\o}+(\t{\hat{u}}\p_\th+\t{\hat{v}}\p_s)(e^{-2s}\Dl_s\t{\hat{\Phi}})\nn\\
                    =&\hat{R}^a_{\o}+e^{-2s} \t{\hat{u}}\Dl_s\t{\hat{\Phi}}_\th+e^{-2s}\t{\hat{v}}\Dl_s\t{\hat{\Phi}}_s-2e^{-2s}\t{\hat{v}}\Dl_s\t{\hat{\Phi}}.\nn
\end{align}
Now we give some estimates for the right hand of \eqref{energyfinal0}.  From \eqref{appomegahat}, the $L^\i$ estimate in \eqref{0modeinfinity}, \eqref{1modeinfinity1}, \eqref{1modeinfinity2} and \eqref{hmodeinfinity}, we see
\begin{align}
 \|e^{4s}\t{\hat{R}}_{\text{comb}}\|^2\ls&  \|e^{4s}{\hat{R}}^a_{\o}\|^2+  \|e^{2s} \t{\hat{u}}\Dl_s\t{\hat{\Phi}}_\th\|^2+ \|e^{2s} \t{\hat{v}}\Dl_s\t{\hat{\Phi}}_s\|^2+\|e^{2s} \t{\hat{v}}\Dl_s\t{\hat{\Phi}}\|^2\nn\\
 \ls& \e^{42}+  \e^{-20}\|\t{\hat{\Phi}}\|_{E}\lt(\|e^{s} \Dl_s\t{\hat{\Phi}}_\th\|^2+ \|\Dl_s\t{\hat{\Phi}}_s\|^2+\|\Dl_s\t{\hat{\Phi}}\|^2\rt)\nn\\
 \ls& \e^{42}+  \e^{-20}\|\t{\hat{\Phi}}\|^2_{E}\lt(\|e^{s} \Dl_s\t{\hat{\Phi}}_{1,\th}\|^2+\|e^{s} \Dl_s\t{\hat{\Phi}}_{h,\th}\|^2 + \|\Dl_s\t{\hat{\Phi}}_s\|^2+\|\Dl_s\t{\hat{\Phi}}\|^2\rt)\nn\\
 \ls&\e^{42}+\e^{-36}\|\t{\hat{\Phi}}\|^4_{E}. \label{energyfinal2}
\end{align}

From \eqref{appuvreminder}, the $L^\i$ estimate in \eqref{0modeinfinity}, \eqref{1modeinfinity1}, \eqref{1modeinfinity2} and \eqref{hmodeinfinity}, we see
\begin{align}
 \|e^{2.5s}{{R}}_{\t{\hat{u}}}\|^2\ls&  \|e^{2.5s}{\hat{R}}^a_{\hat{u}}\|^2+\|e^{2.5s} \t{\hat{u}}\p_\th \t{\hat{u}}\|^2+ \|e^{2.5s} \t{\hat{v}}\t{\hat{u}}_s\|^2+\|e^{2.5s} \t{\hat{u}}\t{\hat{v}}\|^2\nn\\
 \ls& \e^{44}+  \e^{-20}\|\t{\hat{\Phi}}\|_{E}\lt(\|e^{0.5s}\t{\hat{\Phi}}_{\th s}\|^2+ \|e^{-0.5s}(\t{\hat{\Phi}}_s,\t{\hat{\Phi}}_{ss})\|^2+\|e^{-0.5s}\t{\hat{\Phi}}_s\|^2\rt)\nn\\
 \ls& \e^{44}+\e^{-36}\|\t{\hat{\Phi}}\|^4_{E}. \label{energyfinal3}
\end{align}
Similarly, we can obtain that
\be
 \|e^{2.5s}\lt(\p_\th{{R}}_{\t{\hat{u}}}, \p_\th{{R}}_{\t{\hat{u}}} \rt)\|^2\ls \e^{44}+ \e^{-36}\|\t{\hat{\Phi}}\|^4_{E}. \label{energyfinal4}
\ee
Inserting \eqref{energyfinal2},  \eqref{energyfinal3} and \eqref{energyfinal4} into \eqref{energyfinal0}, we obtain
\be\label{energyfinal5}
\|\hat{\Phi}\|^2_{E}\leq  \e^{40}+\e^{-38}\|\t{\hat{\Phi}}\|^4_{E}.
\ee
Let $E=\{(\hat{u},\hat{v})\in C^\infty: (\hat{u},\hat{v})\ \text{ satisfies} \ (\ref{iterative conditon})\ \text{ and } \ \|(\hat{\Phi})\|_E<+\infty\}$.
Thus, due to \eqref{energyfinal5}, there exist $\epsilon_0>0, \dl_0>0$ such that for any $\epsilon\in (0,\epsilon_0), \dl\in(0,\dl_0)$, the operator
\beas
(\t{\hat{u}},\t{\hat{v}})=(e^{-s}\t{\hat{\Phi}}_s, -e^{-s}\t{\hat{\Phi}}_\th )\mapsto (\hat{u},\hat{v})=(e^{-s}{\hat{\Phi}}_s, -e^{-s}{\hat{\Phi}}_\th)
\eeas
maps the ball $B:=\{\t{\hat{\Phi}}: \|\t{\hat{\Phi}}\|^2_E\leq \epsilon^{39}\}$  in $B$ into itself.

Moreover, for every two pairs $(\t{\hat{u}}^1, \t{\hat{v}}^1)$ and $(\t{\hat{u}}^2, \t{\hat{v}}^2)$ in the ball, we have
\bes
\|(\hat{\Phi}^1-\hat{\Phi}^2\|^2_E\leq C\epsilon^{-38}(\|\t{\hat{\Phi}}^1\|_E^2+\|\t{\hat{\Phi}}^2\|_E^2)\|(\t{\hat{\Phi}}^1-\t{\hat{\Phi}}^2\|_E^2.
\ees
This estimate follows the same line as estimate (\ref{energyfinal5}) line by line. We omit the details.   Hence, there exist $\epsilon_0>0,\dl_0>0$ such that for any $\epsilon\in (0,\epsilon_0), \dl\in (0,\dl_0)$, the operator
\beas
(\t{u},\t{v}) \mapsto (u,v)
\eeas
maps the ball $B:=\{(\hat{u},\hat{v}): \|\hat{\Phi}\|^2_E\leq \epsilon^{39}\}$ into itself and is a contraction. By the fixed-point theorem, system \eqref{errorequ} has a solution satisfying
 \bes
 \|\hat{\Phi}\|^2_E\leq \e^{39}.
 \ees
  By estimates \eqref{0modeinfinity}, \eqref{1modeinfinity1}, \eqref{1modeinfinity2}, and \eqref{hmodeinfinity}, we obtain \eqref{vcerrorestix}. This completes the proof of Proposition \ref{vcpropexistence}.
\end{proof}

By changing variables from $(\th,s)$ back to $(\th,r)$ and unknowns $(\hat{u},\hat{v},\hat{\Phi})$ back to $({u},{v},{\Phi})$, we finish proof of Proposition \ref{propexistenceuv}. \qed

Now we can give the proof of Theorem \ref{thmain}.
\begin{proof}
Combining Proposition \ref{propexistenceuv} and the constructed approximate solution in Section \ref{secappro}, we easily obtain Theorem \ref{thmain}.
\end{proof}

\subsection{Proof of Theorem \ref{thmain1}}

Denote $\bl{u}^\la=\la^{-1} \bl{u}$. From \eqref{ns1} and \eqref{nsboundary}, we see that
\begin{equation}\label{ns1existence}
\left \{
\begin {array}{ll}
\bl{u}^\la\cdot\na \bl{u}^\la +\na p^\la-\la^{-1}\Dl \bl{u}^\la=0,\\ [5pt]
\na\cdot\bl{u}^\la=0,\\
\bl{u}^\la\cdot \boldsymbol{\tau}\big|_{\p B}=1 +\dl f(\th),\q \bl{u}^\la\cdot \boldsymbol{n}\big|_{\p B}=0,\\
\bl{u}^\la\big|_{r\rightarrow\i}=0,
\end{array}
\right.
\end{equation}

By using the result in Theorem \ref{thmain} with $\e=\f{1}{\s{\la}}$ and remembering estimate \eqref{solestix}, we see that there exists a $\la_0$, when $\la\geq \la_0$, system \eqref{ns1existence} have a solution $\bl{u}^\la$ satisfying
\bes
\lt|u^\la(\th,r)-(\t{\o}+\mathcal{O}(\la^{-1/2}\dl))r^{-1}\rt|\ls C\dl r^{-2},\q \lt|v^\la(\th,r)\rt|\leq C\la^{-1/2}\dl r^{-2},
\ees
where $\t{\o}:=\s{\fint^{2\pi}_0 (1+\dl f(\th))d\th}\approx 1+O(\dl)$. Now we go back to $\bl{u}$, we see that  there exists a $\la_0$, when $\la\geq \la_0$, system \eqref{ns1} have a solution $\bl{u}$ satisfying
\bes
\lt|u(\th,r)-(\la+\mathcal{O}(\la \dl)+\mathcal{O}(\la^{1/2}\dl))r^{-1}\rt|\ls C\la\dl r^{-2},\q \lt|v(\th,r)\rt|\leq C\la^{1/2}\dl r^{-2},
\ees
which is \eqref{xxx}. \qed

\section{Linear stability estimates of the error equation}  \label{sec3}

In this section, we mainly give the linear stability estimate in Proposition \ref{proplinearstability}. Before that, we give one useful   lemma concerning with one-dimensional Hardy inequality, which will be frequently used in the following proof.

\begin{lemma}
For function $f(s)\in C^1([0,+\i))$ with $f(0)=0$ and $\lim\limits_{s\rightarrow+\i} f^2(s)/s=0$, we have,
\be\label{vchardy1}
\int^{+\i}_{0} s^{-2} f^2(s) ds \leq 4\int^{+\i}_{0} (f_s)^2ds .
\ee
Let $f(s)\in C^1([0,+\i)$.  For any $0\neq\al\in\bR$, $\lim\limits_{s\rightarrow+\i} f^2(s) e^{\al s}=0$ and $f(0)=0$ for $\al<0$, then we have
\be\label{vchardy2}
\int^{+\i}_{1} e^{\al s}f^2 ds \leq \f{4}{\al^2}\int^{+\i}_{0}e^{\al s}(f_s)^2ds.
\ee
\end{lemma}
\begin{proof}
Using integration by parts, the Cauchy inequality and H\"{o}lder inequality, we have
\begin{align}
\int^{+\i}_{0}s^{-2} f^2 dr =&-\int^{+\i}_{0}  f^2(s) d s^{-1}=-\int^{+\i}_{0} 2s^{-1} f f_s ds\nn\\
                       \leq &2\lt(\int^{+\i}_{0} s^{-2} f^2 ds\rt)^{1/2}\lt(\int^{+\i}_{0} (f_s)^2 ds\rt)^{1/2},\label{vchardyproof2}
\end{align}
which indicates \eqref{vchardy1}. For the second one \eqref{vchardy2}, also by using integration by part, we have
\begin{align}
&\int^{+\i}_{0} e^{\al s}f^2 ds=\f{1}{\al}\int^{+\i}_{0}  f^2 de^{\al s} =\f{1}{\al}e^{\al s}f^2(s)\big|^{+\i}_0-\f{1}{\al}\int^{+\i}_{1}2e^{\al s}  f f_s ds=-\f{2}{\al}\int^{+\i}_{0} e^{\al s}  f f_s ds.\nn
\end{align}
Then the same as \eqref{vchardyproof2} by using H\"{o}lder inequality, we have \eqref{vchardy2}.
\end{proof}

\subsection{\texorpdfstring{$\dot{H}^1$ and $\dot{H}^2$ estimates for the stream function $\hat{\Phi}$}{Hdot1 and Hdot2 estimates for the stream function Phi}} \label{subsec3.1}

\indent

In this subsection, we give the weighted $\e$ independently tangential derivatives' estimate for the variable-change solution $(\hat{u},\hat{v})$, which is carried out in the stream function \eqref{vcerrorequ}. We have the following lemma.

\subsubsection{\texorpdfstring{The $\e$-independent positive estimate for the stream function $\hat{\Phi}$}{Epsilon-independent positive estimate for the stream function Phi}}\label{subspositive}

\begin{lemma}\label{lempositive}
 Let $(u,v)$ be a smooth solution of (\ref{vcerrorequ}). Then there exist $\epsilon_0>0$ and $\dl_0>0$ such that, for any $\epsilon\in(0,\epsilon_0)$ and $\dl\in(0,\dl_0)$, there holds
 \begin{align}\label{vcestipositive}
&\int_{\hat{\O}}\lt( e^{2s}\hat{\Phi}^2_{h,\th \th}+ e^{2s} \hat{\Phi}^2_{h,\th s}+ e^{s}\hat{\Phi}^2_{1,\th \th}+ e^s \hat{\Phi}^2_{1,\th s}\rt)\nn\\
\ls & \e^2\dl\int_{\hat{\O}} \lt(e^{2s}\hat{\Phi}^2_{h,ss}+e^{s}\hat{\Phi}^2_{1,ss}+\hat{\Phi}^2_{0,ss}+\hat{\Phi}^2_{0,s}\rt)+ \int_{\hat{\O}}e^{8s}\hat{R}^2_{\text{comb}}.
\end{align}
Here the constant $C$ is independent of $\e_0$, $\dl_0$ and $\hat{\O}:=\bT\times (0,+\i)$.
\end{lemma}

\pf Multiplying \eqref{vcerrorequ} by $e^{3s}\hat{\Phi}_{h,\th}+e^{2s}\hat{\Phi}_{1,\th}$ and integrating the resulted equation on $\hat{\O}:\bT\times[0,+\i)$ to obtain that
\begin{align}
&-\int_{\hat{\O}} [e^{3s}\hat{\Phi}_{h,\th}+e^{2s}\hat{\Phi}_{1,\th}]\lt[\hat{u}^a \Dl_s \hat{\Phi}_\th+ \hat{v}^a \Dl_s \hat{\Phi}_s -2 \hat{v}^a \Dl_s\hat{\Phi}\rt]\nn\\
&-\int_{\hat{\O}} [e^{3s}\hat{\Phi}_{h,\th}+e^{2s}\hat{\Phi}_{1,\th}\hat{u} ]\lt[\hat{u}\Dl_s \hat{\Phi}^a_\th+ \hat{v} \Dl_s \hat{\Phi}^a_s -2 \hat{v} \Dl_s\hat{\Phi}^a\rt]\nn\\
&+\e^2  [e^{3s}\hat{\Phi}_{h,\th}+e^{2s}\hat{\Phi}_{1,\th}]\lt(\Dl^2_s\hat{\Phi}-4\Dl_s\hat{\Phi}_s+4\Dl_s\hat{\Phi}\rt)=\int_{\hat{\O}}  [e^{3s}\hat{\Phi}_{h,\th}+e^{2s}\hat{\Phi}_{1,\th}] e^{2s}\hat{R}_{\text{comb}}.\label{vcpositive0}
\end{align}
Here and later, for simplification of notation, we omit the integral variables $d\th ds$ if no ambiguity is caused. We will estimate terms in \eqref{vcpositive0} one by one.

{\bf\noindent Estimates of $\boldsymbol{-\int_{\hat{\O}} [e^{3s}\hat{\Phi}_{h,\th}+e^{2s}\hat{\Phi}_{1,\th}]\hat{u}^a\Dl_s \hat{\Phi}_\th}$}: The left hand of \eqref{vcestipositive} comes from this term. We first deal with $-\int_{\hat{\O}} e^{3s}\hat{\Phi}_{h,\th}\hat{u}^a\Dl_s \hat{\Phi}_\th$, the other term is similar. By direct integration by parts and boundary condition for $\hat{\Phi}_{\th}$, we obtain that
\begin{align}
&-\int_{\hat{\O}} e^{3s}\hat{\Phi}_{h,\th}\hat{u}^a \Dl_s \hat{\Phi}_\th\nn\\
=& -\lt(\t{\o}+\mathcal{O}(\e\dl)\mathcal{A}_0\rt)\int_{\hat{\O}} e^{2s} \hat{\Phi}_{h,\th}  \Dl_s \hat{\Phi}_{h,\th}-\int_{\hat{\O}} \lt[\hat{u}^a-\lt(\t{\o}+\mathcal{O}(\e\dl)\mathcal{A}_0\rt)e^{-s}\rt] e^{3s} \hat{\Phi}_{h,\th} \Dl_s \hat{\Phi}_\th\nn\\
=&\lt(\t{\o}+\mathcal{O}(\e\dl)\mathcal{A}_0\rt)\int_{\hat{\O}}  e^{2s}\lt(\hat{\Phi}^2_{h,\th\th}+\hat{\Phi}^2_{h,\th s}-2\hat{\Phi}_{h,\th}\rt)-\int_{\hat{\O}} \lt[\hat{u}^a-\lt(\t{\o}+\mathcal{O}(\e\dl)\mathcal{A}_0\rt)e^{-s}\rt] e^{3s} \hat{\Phi}_{h,\th} \Dl_s \hat{\Phi}_\th. \label{vcpositive-1}
\end{align}
Using the Fourier series representation of $\hat{\Phi}$, we see that
\begin{align}
& \int_{\hat{\O}}  e^{2s}\lt(\hat{\Phi}^2_{h,\th\th}+\hat{\Phi}^2_{h,\th s}-2\hat{\Phi}^2_{h,\th}\rt)\geq\f{1}{2}\int_{\hat{\O}}  e^{2s}\lt(\hat{\Phi}^2_{h,\th\th}+\hat{\Phi}^2_{h,\th s}\rt). \label{vcpositive-2}
\end{align}
Here we have used the following facts
\begin{align}
 &\int_{\hat{\O}}  e^{2s}\lt(\hat{\Phi}^2_{h,\th\th}+\hat{\Phi}^2_{h,\th s}-2\hat{\Phi}^2_{h,\th}\rt)\nn\\
 =&\pi\int^\i_0e^{2s} \sum\limits_{k\geq 2}\lt(k^4|(A^{\hat{\Phi}}_{k}(s),B^{\hat{\Phi}}_{k}(s))|^2+k^2|(A^{\hat{\Phi}}_{k}(s)_{,s},B^{\hat{\Phi}}_{k}(s)_{,s})|^2-2k^2|(A^{\hat{\Phi}}_{k}(s),B^{\hat{\Phi}}_{k}(s))|^2\rt)ds\nn\\
 \geq &\pi\int^\i_0e^{2s} \sum\limits_{k\geq 2}\lt(\f{1}{2}k^4|(A^{\hat{\Phi}}_{k}(s),B^{\hat{\Phi}}_{k}(s))|^2+k^2|(A^{\hat{\Phi}}_{k}(s)_{,s},B^{\hat{\Phi}}_{k}(s)_{,s})|^2\rt)ds\geq \f{1}{2}\int_{\hat{\O}}  e^{2s}\lt(\hat{\Phi}^2_{h,\th\th}+\hat{\Phi}^2_{h,\th s}\rt). \nn
\end{align}
Using integration by parts, the fourth one of estimate in \eqref{vcappdetail0}, Cauchy inequality, Poincar\'{e} inequality in $\th$ variable, Hardy inequality \eqref{vchardy1} and \eqref{vchardy2}, we have
\begin{align}
&\lt|-\int_{\hat{\O}} \lt[\hat{u}^a-\lt(\t{\o}+\mathcal{O}(\e\dl)\mathcal{A}_0\rt)e^{-s}\rt] e^{3s} \hat{\Phi}_{h,\th} \Dl_s \hat{\Phi}_\th\rt|\nn\\
\ls &\dl \int_{\hat{\O}} e^{s} \lt( |\hat{\Phi}_{h,\th}|+|s^{-1}\hat{\Phi}_{h,\th}|+ |\hat{\Phi}_{h,\th\th}| + |\hat{\Phi}_{h,\th s}|\rt)\lt( |\hat{\Phi}_{\th\th}|+ |\hat{\Phi}_{\th s}| \rt)\nn\\
\ls &\dl \int_{\hat{\O}} e^{s} \lt(\hat{\Phi}^2_{\th\th}+ \hat{\Phi}^2_{\th s}\rt)=\dl \int_{\hat{\O}} e^{s} \lt(\hat{\Phi}^2_{h,\th\th}+ \hat{\Phi}^2_{h,\th s}\rt)+\dl\int_{\hat{\O}} e^{s} \lt(\hat{\Phi}^2_{1,\th\th}+ \hat{\Phi}^2_{1,\th s}\rt). \label{vcpositive-3}
\end{align}
Inserting \eqref{vcpositive-3} and \eqref{vcpositive-2} into \eqref{vcpositive-1}, we achieve that for small $\e_0$ and $\dl_0$, when $\e\leq \e_0$ and $\dl\leq \dl_0$,
\begin{align}
&-\int_{\hat{\O}} e^{3s}\hat{\Phi}_{h,\th} \hat{u}^a\hat{\Phi}_\th   \Dl_s \hat{\Phi}_{\th}\gs\int_{\hat{\O}} e^{2s} \lt(\hat{\Phi}^2_{h,\th\th}+\hat{\Phi}^2_{h,\th s}\rt)-\dl\int_{\hat{\O}} e^{s} \lt(\hat{\Phi}^2_{1,\th\th}+ \hat{\Phi}^2_{1,\th s}\rt).\label{vcpositive-4}
\end{align}
Similarly, we can obtain that  for small $\e_0$ and $\dl_0$, when $\e\leq \e_0$ and $\dl\leq \dl_0$,
\begin{align}
&-\int_{\hat{\O}} e^{2s}\hat{\Phi}_{1,\th}\hat{u}^a\hat{\Phi}_\th \Dl_s \hat{\Phi}_{\th}\gs\int_{\hat{\O}} e^{s} \lt(\hat{\Phi}^2_{1,\th\th}+\hat{\Phi}^2_{1,\th s}\rt)-\dl\int_{\hat{\O}} \lt(\hat{\Phi}^2_{h,\th\th}+ \hat{\Phi}^2_{h,\th s}\rt).\label{vcpositive1x}
\end{align}
Then combining \eqref{vcpositive-4} and \eqref{vcpositive1x}, we achieve that for small $\e_0$ and $\dl_0$, when $\e\leq \e_0$ and $\dl\leq \dl_0$,
\begin{align}
-\int_{\hat{\O}} [e^{3s}\hat{\Phi}_{h,\th}+e^{2s}\hat{\Phi}_{1,\th}]\hat{u}^a\Dl_s \hat{\Phi}_\th\gs \int_{\hat{\O}} \lt(e^{2s}\hat{\Phi}^2_{h,\th\th}+e^{2s}\hat{\Phi}^2_{h,\th s}\rt)+ \lt(e^{s}\hat{\Phi}^2_{1,\th\th}+e^{s}\hat{\Phi}^2_{1,\th s}\rt). \label{vcpositive1}
\end{align}

{\bf\noindent Estimates of $\boldsymbol{-\int_{\hat{\O}}[e^{3s}\hat{\Phi}_{h,\th} +e^{2s}\hat{\Phi}_{1,\th}]\hat{v}^a \Dl_s \hat{\Phi}_s+2\int_{\hat{\O}}[e^{3s}\hat{\Phi}_{h,\th} +e^{2s}\hat{\Phi}_{1,\th}] \hat{v}^a \Dl_s\hat{\Phi}}$}: We also only deal with the term concerning multiplier $e^{3s}\hat{\Phi}_{h,\th}$, the other is similar.  By integration by parts, we have
\begin{align}
&-\int_{\hat{\O}}e^{3s}\hat{\Phi}_{h,\th} \hat{v}^a \Dl_s \hat{\Phi}_s+2\int_{\hat{\O}}e^{3s}  \hat{\Phi}_{h,\th}\hat{v}^a \Dl_s\hat{\Phi}\nn\\
=&\int_{\hat{\O}}e^{3s}\lt(\hat{\Phi}_{h,\th s} \hat{v}^a+\hat{\Phi}_{h,\th } \hat{v}^a_s+ 5\hat{v}^a \hat{\Phi}_{h,\th}  \rt) \Dl_s \hat{\Phi}\nn\\
=&\int_{\hat{\O}}e^{3s}\lt(\hat{\Phi}_{h,\th s} \hat{v}^a+\hat{\Phi}_{h,\th } \hat{v}^a_s+ 5\hat{v}^a \hat{\Phi}_{h,\th}  \rt) \lt(\hat{\Phi}_{\th\th}+\hat{\Phi}_{0,ss}+\hat{\Phi}_{\neq,ss}\rt) .\label{vcpositive2}
\end{align}
From estimate \eqref{vcappdetail1}, we have
\be
\lt|\hat{v}^a\rt|\ls \e\dl e^{-2s}, \q \lt|\hat{v}^a_s\rt|\ls \e\dl (1+s^{-1})e^{-2s}. \label{vcpositive3}
\ee
Using the above estimate \eqref{vcpositive3}, Cauchy inequality, Poincar\'{e} inequality in $\th$ variable, Hardy inequality \eqref{vchardy1} and \eqref{vchardy2}, we obtain that
\begin{align}
&\lt|\int_{\hat{\O}}e^{3s}\lt(\hat{\Phi}_{h,\th s} \hat{v}^a+\hat{\Phi}_{h,\th } \hat{v}^a_s+ 5\hat{v}^a \hat{\Phi}_{h,\th s} \rt) \lt(\hat{\Phi}_{\th\th}+\hat{\Phi}_{\neq,ss}\rt)\rt|\nn\\
\ls& \e\dl\int_{\hat{\O}} e^{s}\lt(s^{-1}|\hat{\Phi}_{h,\th}|+|\hat{\Phi}_{h,\th}|+|\hat{\Phi}_{h,\th s}|\rt)\lt(|\hat{\Phi}_{h,\th \th}|+|\hat{\Phi}_{\neq,ss}|\rt)\nn\\
\ls&\dl\int_{\hat{\O}} e^{s}\lt(\hat{\Phi}^2_{\th s}+\hat{\Phi}^2_{\th \th}\rt)+\e^2\dl \int_{\hat{\O}} e^{s}\hat{\Phi}^2_{\neq,ss}\nn\\
\ls&\dl \int_{\hat{\O}} e^{s} \lt(\hat{\Phi}^2_{\th\th}+ \hat{\Phi}^2_{\th s}\rt)+\e^2\dl\int_{\hat{\O}}e^s\hat{\Phi}^2_{\neq,ss}ds. \label{vcpositive2x1}
\end{align}
Then by the series expansion representation of $v^a_e$ in the first one of \eqref{vcappdetail1}, we have
\begin{align}
&\int_{\hat{\O}}e^{3s}\lt(\hat{\Phi}_{h,\th s} \hat{v}^a+\hat{\Phi}_{h,\th } \hat{v}^a_s+ 5\hat{v}^a \hat{\Phi}_{h,\th} \rt)\hat{\Phi}_{0,ss}\nn\\
=&\underbrace{\int_{\hat{\O}}e^{3s}\lt(\hat{\Phi}_{h,\th s} (\hat{v}^a-\hat{v}^a_e)+\hat{\Phi}_h,{\th } (\hat{v}^a-\hat{v}^a_e)_s+ 5(\hat{v}^a-\hat{v}^a_e) \hat{\Phi}_{h,\th} \rt)\hat{\Phi}_{0,ss}}_{I_1}\nn\\
 &+\underbrace{\int_{\hat{\O}}e^{3s}\lt(\hat{\Phi}_{h,\th s} \hat{v}^a_{eh}+\hat{\Phi}_{h,\th} \hat{v}^a_{eh,s}+ 5\hat{v}^a_{eh} \hat{\Phi}_{h,\th} \rt)\hat{\Phi}_{0,ss}}_{I_2} \label{vcpositive2x2}
\end{align}
By using the second estimate in \eqref{vcappdetail1}, Cauchy inequality, Poincar\'{e} inequality in $\th$ variable, Hardy inequality \eqref{vchardy1} and \eqref{vchardy2}, we obtain that
\begin{align}
|I_1|\ls& \e\dl \int_{\hat{\O}} \lt(|\hat{\Phi}_{h,\th s}|+s^{-1}|\hat{\Phi}_{h,\th }| +  |\hat{\Phi}_{h,\th}| \rt)\hat{\Phi}_{0,ss}\ls \dl \int_{\hat{\O}}e^{s}\lt(\hat{\Phi}^2_{h,\th s}+\hat{\Phi}^2_{h,\th\th} \rt)+\e^2\dl \int_{\hat{\O}} \hat{\Phi}^2_{0,ss}. \label{vcpositive2x3}
\end{align}
By using the first estimate in \eqref{vcappdetail1}, Cauchy inequality, Poincar\'{e} inequality in $\th$ variable, we obtain that
\begin{align}
|I_2|\ls& \e\dl \int_{\hat{\O}}e^s \lt(|\hat{\Phi}_{h,\th s}|+|\hat{\Phi}_{h,\th }| +  |\hat{\Phi}_{h,\th}| \rt)\hat{\Phi}_{0,ss} \ls \dl \int_{\hat{\O}}e^{2s}\lt(\hat{\Phi}^2_{h,\th s}+\hat{\Phi}^2_{h,\th\th} \rt)+\e^2\dl \int_{\hat{\O}} \hat{\Phi}^2_{0,ss}. \label{vcpositive2x4}
\end{align}

Inserting estimates in \eqref{vcpositive2x3} and \eqref{vcpositive2x4} into \eqref{vcpositive2x2}, and then inserting the resulted estimate and estimate \eqref{vcpositive2x1} into \eqref{vcpositive2} to obtain that
\begin{align}
&\lt|-\int_{\hat{\O}}e^{3s}\hat{\Phi}_{h,\th} \hat{v}^a \Dl_s \hat{\Phi}_s +2\int_{\hat{\O}}e^{3s} \hat{\Phi}_{h,\th}\hat{v}^a  \Dl_s\hat{\Phi}\rt|\nn\\
\ls &\dl\int_{\hat{\O}}\lt(e^{2s}\hat{\Phi}^2_{h,\th s}+e^{2s}\hat{\Phi}^2_{h,\th\th}+e^{s}\hat{\Phi}^2_{1,\th s}+e^{s}\hat{\Phi}^2_{1,\th\th}\rt)+ \e^2\dl\int_{\hat{\O}} \lt(e^{s} \hat{\Phi}^2_{h,ss}+\hat{\Phi}^2_{0,ss}\rt). \nn
\end{align}
Similarly, we can obtain that
\begin{align}
&\lt|-\int_{\hat{\O}}e^{2s}\hat{\Phi}_{1,\th} \hat{v}^a \Dl_s \hat{\Phi}_s +2\int_{\hat{\O}}e^{2s} \hat{\Phi}_{1,\th}\hat{v}^a  \Dl_s\hat{\Phi}\rt|\nn\\
\ls &\dl\int_{\hat{\O}}\lt(e^{s}\hat{\Phi}^2_{h,\th s}+e^{s}\hat{\Phi}^2_{h,\th\th}+e^{s}\hat{\Phi}^2_{1,\th s}+e^{s}\hat{\Phi}^2_{1,\th\th}\rt)+ \e^2\dl\int_{\hat{\O}} \lt(e^{s} \hat{\Phi}^2_{h,ss}+\hat{\Phi}^2_{0,ss}\rt). \nn
\end{align}
Combining the above two, we achieve that
\begin{align}
&\lt|-\int_{\hat{\O}}[e^{3s}\hat{\Phi}_{h,\th} +e^{2s}\hat{\Phi}_{1,\th}]\hat{v}^a \Dl_s \hat{\Phi}_s+2\int_{\hat{\O}}[e^{3s}\hat{\Phi}_{h,\th} +e^{2s}\hat{\Phi}_{1,\th}] \hat{v}^a \Dl_s\hat{\Phi}\rt|\nn\\
\ls& \dl\int_{\hat{\O}}\lt(e^{2s}\hat{\Phi}^2_{h,\th s}+e^{2s}\hat{\Phi}^2_{h,\th\th}+e^{s}\hat{\Phi}^2_{1,\th s}+e^{s}\hat{\Phi}^2_{1,\th\th}\rt)+ \e^2\dl\int_{\hat{\O}} \lt(e^{s} \hat{\Phi}^2_{h,ss}+\hat{\Phi}^2_{0,ss}\rt).  \label{vcpositive2x6}
\end{align}

{\bf\noindent Estimates of $\boldsymbol{-\int_{\hat{\O}}[e^{3s}\hat{\Phi}_{h,\th} +e^{2s}\hat{\Phi}_{1,\th}]\hat{u} \Dl_s \hat{\Phi}^a_\th}$}:

We rewrite the term $-\int_{\hat{\O}}e^{3s} \hat{\Phi}_{h,\th}\hat{u} \Dl_s \hat{\Phi}^a_\th$ as
\begin{align}
&-\int_{\hat{\O}}e^{3s} \hat{\Phi}_{h\th}\hat{u} \Dl_s \hat{\Phi}^a_\th= -\int_{\hat{\O}}e^{2s}\hat{\Phi}_{h,\th} \hat{\Phi}_s \Dl_s \hat{\Phi}^a_\th=-\int_{\hat{\O}}e^{2s} \hat{\Phi}_{h,\th} \lt(\hat{\Phi}_{\neq,s}+\hat{\Phi}_{0,s}\rt) \Dl_s \hat{\Phi}^a_\th\nn\\
=&-\int_{\hat{\O}}e^{2s} \hat{\Phi}_{h,\th} \hat{\Phi}_{\neq,s} \Dl_s \hat{\Phi}^a_\th-\int_{\hat{\O}}e^{2s}\hat{\Phi}_{h,\th}\hat{\Phi}_{0,s}\Dl_s \hat{\Phi}^a_\th\nn\\
=&\underbrace{-\int_{\hat{\O}}e^{2s} \hat{\Phi}_{h,\th} \hat{\Phi}_{\neq,s} \Dl_s \hat{\Phi}^a_\th}_{I_3}\underbrace{-\int_{\hat{\O}}e^{2s} \hat{\Phi}_{h,\th}\hat{\Phi}_{0,s}\Dl_s\lt( \hat{\Phi}^a_\th-\hat{\Phi}^a_{e,\th}\rt)}_{I_4}\underbrace{-\int_{\hat{\O}}e^{2s} \hat{\Phi}_{h,\th}\hat{\Phi}_{0,s}\Dl_s \hat{\Phi}^a_{e,\th}}_{I_5}. \label{vcpositive3x1}
\end{align}
 Using the third estimate of \eqref{vcappdetail1}, we see that
\begin{align}
&\lt|\Dl_s\hat{\Phi}^a_\th\rt|=\lt|\Dl_s\lt(e^s\hat{v}^a\rt)\rt|\ls \e\dl e^{-s}(1+\f{1}{s^2}),\nn\\
&\lt|\Dl_s\lt(\hat{\Phi}^a_\th-\hat{\Phi}^a_{e,\th}\rt)\rt|=\lt|\Dl_s\lt(e^s\lt(\hat{v}^a-\hat{v}^a_{e}\rt)\rt)\rt|\ls \e\dl s^{-2}e^{-3s}.  \label{vcpositive3x2}
\end{align}
Using the first one of the above estimates \eqref{vcpositive3x2}, Cauchy inequality, Poincar\'{e} inequality in $\th$ variable, Hardy inequality \eqref{vchardy1} and \eqref{vchardy2}, we obtain that
\begin{align}
|I_3|\ls& \e\dl\int_{\hat{\O}}e^{s} (1+s^{-2}) |\hat{\Phi}_{h,\th}| |\hat{\Phi}_{\neq,s}|\ls \dl\int_{\hat{\O}}e^{s} \lt(\hat{\Phi}^2_{h,\th\th}+\hat{\Phi}^2_{h,\th s}\rt)+\e^2\dl\int_{\hat{\O}}e^{s} \hat{\Phi}^2_{\neq,ss}.  \label{vcpositive3x3}
\end{align}
Using the second one of the above estimates \eqref{vcpositive3x2}, Cauchy inequality, Poincar\'{e} inequality in $\th$ variable, Hardy inequality \eqref{vchardy1} and \eqref{vchardy2}, we obtain that
\begin{align}
|I_4|\ls& \e\dl\int_{\hat{\O}}e^{-s} s^{-2} |\hat{\Phi}_{h,\th}| |\hat{\Phi}_{0,s}|\ls \dl\int_{\hat{\O}}e^{-s}\hat{\Phi}^2_{h,\th s}+\e^2\dl\int_{\hat{\O}}e^{-s} \hat{\Phi}^2_{0,ss}\nn\\
 \ls&\dl \int_{\hat{\O}} e^{-s} \lt(\hat{\Phi}^2_{h,\th\th}+ \hat{\Phi}^2_{h,\th s}\rt)+\e^2\dl \int_{\hat{\O}} e^{-s}\hat{\Phi}^2_{0,ss}.  \label{vcpositive3x4}
\end{align}
Using the first one of \eqref{vcappdetail1} and Cauchy inequality and Poincar\'{e} inequality in $\th$ variable, we see that
\begin{align}
|I_5|\ls   \e\dl\int_{\hat{\O}}e^{s} |\hat{\Phi}_{h,\th}||\hat{\Phi}_{0,s}|\ls \dl \int_{\hat{\O}} e^{2s} \hat{\Phi}^2_{h,\th\th}+\e^2\dl \int_{\hat{\O}}\hat{\Phi}^2_{0,s}. \label{vcpositive3x5}
\end{align}
Inserting \eqref{vcpositive3x3}, \eqref{vcpositive3x4}, \eqref{vcpositive3x5} into \eqref{vcpositive3x1}, we achieve that
\begin{align}
&\lt|-\int_{\hat{\O}}e^{3s} \hat{\Phi}_{h,\th}\hat{u} \Dl_s \hat{\Phi}^a_\th\rt|\ls \dl\int_{\hat{\O}}e^{s}\lt(\hat{\Phi}^2_{h,\th s}+\hat{\Phi}^2_{h,\th\th}  \rt)+ \e^2\dl\int^{\i}_0 \lt(e^{s}\hat{\Phi}^2_{\neq,ss}+\hat{\Phi}^2_{0,ss} +\hat{\Phi}^2_{0,s}\rt). \nn
\end{align}
Similarly, we can obtain that
\begin{align}
&\lt|-\int_{\hat{\O}}e^{2s} \hat{\Phi}_{1,\th}\hat{u} \Dl_s \hat{\Phi}^a_\th\rt|\ls \dl\int_{\hat{\O}}e^{s}\lt(\hat{\Phi}^2_{1,\th s}+\hat{\Phi}^2_{1,\th\th}  \rt)+ \e^2\dl\int^{\i}_0 \lt(e^{s}\hat{\Phi}^2_{\neq,ss}+\hat{\Phi}^2_{0,ss}\rt). \nn
\end{align}
Then combining the above two, we see that
\begin{align}
&\lt|-\int_{\hat{\O}}[e^{3s} \hat{\Phi}_{h,\th}+e^{2s} \hat{\Phi}_{1,\th}]\hat{u} \Dl_s \hat{\Phi}^a_\th\rt|\ls \dl\int_{\hat{\O}}e^{s}\lt(\hat{\Phi}^2_{\th s}+\hat{\Phi}^2_{\th\th}  \rt)+ \e^2\dl\int_{\hat{\O}} \lt(e^{s}\hat{\Phi}^2_{\neq,ss}+\hat{\Phi}^2_{0,ss} +\hat{\Phi}^2_{0,s}\rt). \label{vcpositive3x6}
\end{align}

{\bf\noindent Estimates of $\boldsymbol{-\int_{\hat{\O}}[e^{3s} \hat{\Phi}_{h,\th}+e^{2s} \hat{\Phi}_{1,\th}] [\hat{v} \Dl_s \hat{\Phi}^a_s +2 v \Dl_s\hat{\Phi}^a]}$}:

Using the fourth one of estimates \eqref{vcappdetail0} and \eqref{vcappdetail1}, we see that
\begin{align}
&\lt|\Dl_s\hat{\Phi}^a_s\rt|=\lt|\Dl_s(e^s\hat{u}^a)\rt|=\lt|\Dl_s(e^s(\hat{u}^a-\t{\o}e^{-s}-\mathcal{O}(\e\dl)e^{-s})\rt|\ls \dl e^{-s} (1+s^{-2})\nn\\
&\lt|\Dl_s\hat{\Phi}^a\rt|=\lt|-(e^s\hat{v}^a)_\th+ \p_s(e^s(\hat{u}^a-\t{\o}e^{-s}-O(\e\dl)e^{-s}))\rt|\ls e^{-s}\dl (1+s^{-1}).  \label{vcpositive3x9}
\end{align}
Then by using Poincar\'{e} inequality in $\th$ variable, Hardy inequality \eqref{vchardy1} and \eqref{vchardy2}, we achieve that
\begin{align}
&\lt|-\int_{\hat{\O}}e^{3s}\hat{\Phi}_{h,\th} [\hat{v} \Dl_s \hat{\Phi}^a_s +2\hat{v}\Dl_s\hat{\Phi}^a]\rt|=\lt|\int_{\hat{\O}}e^{2s}\hat{\Phi}_{h,\th}[\hat{\Phi}_\th  \Dl_s \hat{\Phi}^a_s -2 \hat{\Phi}_\th \Dl_s\hat{\Phi}^a]\rt| \nn\\
\ls& \dl\int_{\hat{\O}}e^s (1+s^{-2})|\hat{\Phi}_\th\hat{\Phi}_{h,\th}|\ls\dl\int_{\hat{\O}} e^s\lt(\hat{\Phi}^2_{\th\th}+\hat{\Phi}^2_{\th s}\rt). \nn
\end{align}
Similarly,
\be
\lt|-\int_{\hat{\O}}e^{2s}\hat{\Phi}_{1,\th} [\hat{v} \Dl_s \hat{\Phi}^a_s +2\hat{v}\Dl_s\hat{\Phi}^a]\rt|\ls \dl\int_{\hat{\O}} e^s\lt(\hat{\Phi}^2_{\th\th}+\hat{\Phi}^2_{\th s}\rt). \nn
\ee
Then, combining the above two, we achieve that
\begin{align}
\lt|-\int_{\hat{\O}}[e^{3s} \hat{\Phi}_{h,\th}+e^{2s} \hat{\Phi}_{1,\th}] [\hat{v} \Dl_s \hat{\Phi}^a_s +2 v \Dl_s\hat{\Phi}^a]\rt|\ls \dl\int_{\hat{\O}} e^s\lt(\hat{\Phi}^2_{\th\th}+\hat{\Phi}^2_{\th s}\rt) \label{vcpositive4}
\end{align}

{\bf\noindent Estimates of $\boldsymbol{\e^2 \int_{\hat{\O}}[e^{2s} \hat{\Phi}_{h,\th}+e^{s} \hat{\Phi}_{1,\th}]\lt(\Dl^2_s\hat{\Phi}-4\Dl_s\hat{\Phi}_s+4\Dl_s\hat{\Phi}\rt)}$}:

Using integration by parts and series expansion for $\hat{\Phi}$, we obtain that
\begin{align}
&\e^2 \int_{\hat{\O}}e^{2s} \hat{\Phi}_{h,\th}\lt(\Dl^2_s\hat{\Phi}-4\Dl_s\hat{\Phi}_s+4\Dl_s\hat{\Phi}\rt)\nn\\
=& \e^2 \int_{\hat{\O}}  \Dl_s\hat{\Phi}_{h}\lt(8e^{2s}\hat{\Phi}_{h,\th s}+ 16e^{2s}\hat{\Phi}_{h,\th}\rt)=\e^2 \int_{\hat{\O}} \lt(8e^{2s}\hat{\Phi}_{h,\th s}\hat{\Phi}_{h,\th \th}+ 8e^{2s}\hat{\Phi}_{h,s s}\hat{\Phi}_{h,\th s}+32e^{2s} \hat{\Phi}_{h,\th s}\hat{\Phi}_{h}\rt)\nn\\
=&\underbrace{\e^2 \int_{\hat{\O}} \lt(8e^{2s}\hat{\Phi}_{h,\th s}\hat{\Phi}_{h,\th \th}+ 8e^{2s}\hat{\Phi}_{h, s s}\hat{\Phi}_{h,\th s}+32e^{2s} \hat{\Phi}_{h,\th s}\hat{\Phi}_{h}\rt)}_{I_6}\nn
\end{align}
Using Cauchy inequality, Poincar\'{e} inequality in $\th$ variable, we see that
\begin{align}
|I_6|\ls \dl  \int_{\hat{\O}} e^{2s}\lt(\hat{\Phi}^2_{h,\th s}+\hat{\Phi}^2_{h,\th \th}\rt)+\e^4\dl^{-1} \int_{\hat{\O}} e^{2s}\hat{\Phi}^2_{h,ss}. \nn
\end{align}
Similarly, we can obtain that
\begin{align}
&\e^2 \int_{\hat{\O}}e^{2s} \hat{\Phi}_{h,\th}\lt(\Dl^2_s\hat{\Phi}-4\Dl_s\hat{\Phi}_s+4\Dl_s\hat{\Phi}\rt)\ls\dl  \int_{\hat{\O}} e^{s}\lt(\hat{\Phi}^2_{1,\th s}+\hat{\Phi}^2_{1,\th \th}\rt)+\e^4\dl^{-1} \int_{\hat{\O}} e^{s}\hat{\Phi}^2_{1,ss}.\nn
\end{align}
Then the above two indicate that
\begin{align}
&\e^2 \int_{\hat{\O}}[e^{2s} \hat{\Phi}_{h,\th}+e^{s} \hat{\Phi}_{1,\th}]\lt(\Dl^2_s\hat{\Phi}-4\Dl_s\hat{\Phi}_s+4\Dl_s\hat{\Phi}\rt)\nn\\
\ls&  \dl  \int_{\hat{\O}}  \lt(e^{2s}\hat{\Phi}^2_{h,\th s}+e^{2s}\hat{\Phi}^2_{h,\th \th}+e^s\hat{\Phi}^2_{1,\th s}+e^s\hat{\Phi}^2_{1,\th \th}\rt)+\e^4\dl^{-1} \int_{\hat{\O}} \lt(e^{2s}\hat{\Phi}^2_{h,ss}+ e^{s}\hat{\Phi}^2_{1,ss}\rt). \label{vcpositive5}
\end{align}
Inserting \eqref{vcpositive1}, \eqref{vcpositive2x6}, \eqref{vcpositive3x6}, \eqref{vcpositive4}, and \eqref{vcpositive5} into \eqref{vcpositive0}, for sufficiently small $\e_0$ and $\dl_0$ (i.e., when $\e<\e_0$ and $\dl<\dl_0$), we obtain
\begin{align}
& \int_{\hat{\O}}\lt( e^{2s}\hat{\Phi}^2_{h,\th\th}+ e^{2s} \hat{\Phi}^2_{h,\th s}+ e^{s}\hat{\Phi}^2_{1,\th\th}+ e^s \hat{\Phi}^2_{1,\th s}\rt)\nn\\
\ls & \e^2\dl\int_{\hat{\O}} \lt(e^{2s}\hat{\Phi}^2_{h,ss}+e^{s}\hat{\Phi}^2_{1,ss}+\hat{\Phi}^2_{0,ss}+\hat{\Phi}^2_{0,s}\rt)+ \int_{\hat{\O}}[e^{5s} \hat{\Phi}_{h,\th}+e^{4s} \hat{\Phi}_{1,\th}]\hat{R}_{\text{comb}}|.\nn
\end{align}
This is \eqref{vcestipositive} by using Cauchy inequality for the last term of the above inequality. \qed

In order to close the second order derivative estimate for the stream function $\hat{\Phi}$, we still need to give the control of the first term on the right hand of \eqref{vcestipositive}, which will be achieved by the following lemma through basic energy estimates. The energy estimates are divided into the estimate for the Fourier zero mode of $\hat{\Phi}_0$, which comes from directly from \eqref{vcerrorequuv}$_1$, and the estimates for the non-zero frequency of $\hat{\Phi}_{\neq}$, which comes from \eqref{vcerrorequ}.

\subsubsection{ The $\dot{H}^2$ energy estimates for the stream function $\hat{\Phi}$} \label{subsenergy}

\begin{lemma}\label{lemlinear}
 Let $(\hat{u},\hat{v})$ be a smooth solution of (\ref{vcerrorequuv}). Then there exist $\epsilon_0>0$ and $\dl_0>0$ such that, for any $\epsilon\in(0,\epsilon_0)$ and $\dl\in(0,\dl_0)$, there holds
\begin{align}
&\e^2\int^\i_0 \lt(\hat{\Phi}^2_{0,ss}+\hat{\Phi}^2_{0,s}\rt)ds\ls \dl\int_{\hat{\O}} \lt(e^{2s}\hat{\Phi}^2_{h,\th\th}+e^{2s}\hat{\Phi}^2_{h,\th s}+e^{s}\hat{\Phi}^2_{1,\th\th}+e^{s}\hat{\Phi}^2_{1,\th s}\rt)+\e^{-2}\int_{\hat{\O}}e^{3s} R^2_{\hat{u}0}. \label{vczeroesti}
\end{align}
\end{lemma}

\pf Now taking $\th$ average of \eqref{vcerrorequuv}$_1$ in $\bT$, we see that in the new variables $(\th,s)$ and unknowns $(\hat{u},\hat{v})$,
\be\label{vcu0formula}
\lt\{
\bali
&\e^2\left(\hat{u}_{0,ss}-\hat{u}_{0,s}\right)=S_{\hat{u}0}-R_{\hat{u}0},\\
& \hat{u}_{0}(0)=\hat{u}_0 (+\i)=0,
\eali
\rt.
\ee
where
\bes
S_{\hat{u}0}=\fint_\bT (v^a \hat{u}_{\neq,s}+\hat{v} \hat{u}^a_{\neq,s}+\hat{u}_{\neq}\hat{v}^a+\hat{u}^a_{\neq}\hat{v})d\th,\q R_{\hat{u}0}=\fint_\bT R_{\hat{u}} d\th.
\ees
Multiplying \eqref{vcu0formula} by $-e^{s} \hat{\Phi}_{0,s}+e^s \hat{\Phi}_0$, then integrating  the resulted equations for $s$ variable and using integration by parts to obtain
\begin{align}
&\e^2\int^\i_0 \lt(\hat{\Phi}^2_{0,ss}+\hat{\Phi}^2_{0,s}\rt)ds+\hat{\Phi}^2_{0}\big|_{+\i}\nn\\
=&\f{1}{2\pi}\int_{\hat{\O}} (\hat{v}^a \hat{u}_{\neq,s}+\hat{v} \hat{u}^a_{\neq,s}+\hat{u}_{\neq}\hat{v}^a+\hat{u}^a_{\neq}\hat{v})\lt(-e^{s} \hat{\Phi}_{0,s}+e^s \hat{\Phi}_0\rt)+\f{1}{2\pi}\int_{\hat{\O}}R_{\hat{u}0}\lt(-e^{s} \hat{\Phi}_{0,s}+e^s \hat{\Phi}_0\rt)\nn\\
=&\underbrace{\f{1}{2\pi}\int_{\hat{\O}} (-\hat{v}^a_s \hat{u}_{\neq}+\hat{v} \hat{u}^a_{\neq,s}+\hat{u}^a_{\neq}\hat{v})\lt(-e^{s} \hat{\Phi}_{0,s}+e^s \hat{\Phi}_0\rt)+\f{1}{2\pi}\int_{\hat{\O}}\hat{v}^a \hat{u}_{\neq}e^s\hat{\Phi}_{0,ss} }_{J_1}+\f{1}{2\pi}\int_{\hat{\O}}R_{\hat{u}0}\lt(-e^{s} \hat{\Phi}_{0,s}+e^s \hat{\Phi}_0\rt). \label{vczeromode1}
\end{align}
From the fourth one of \eqref{vcappdetail0} and the third one of \eqref{vcappdetail1} in Lemma \ref{vclemdetail}, we obtain that
\be
|\hat{v}^a|\ls \e\dl e^{-2s},\q |\hat{u}^a_{\neq,s}|\ls\e\dl(1+s^{-2}) e^{-2s},\q |\hat{v}^a_s|+|\hat{u}^a_{\neq}|\ls \e\dl(1+s^{-1})e^{-2s}. \label{vczeromode2}
\ee
Then inserting \eqref{vczeromode2} into \eqref{vczeromode1}, using Cauchy Poincar\'{e} inequality in $\th$ variable, Hardy inequalities \eqref{vchardy1} and\eqref{vchardy2}, we can obtain that
\begin{align}
|J_1|\ls& \e\dl\int_{\hat{\O}}e^{-s}\lt( (1+s^{-2}) |\hat{v}|+(1+s^{-1}) |\hat{u}_{\neq}|\rt)\lt(|\hat{\Phi}_{0,s}|+|\hat{\Phi}_{0}|\rt)+\e\dl\int_{\hat{\O}}e^{-s}|\hat{u}_{\neq}| | \hat{\Phi}_{0,ss}|\nn\\
     \ls& \e^2\dl\int_{\hat{\O}} (1+s^{-2}) e^{-s}\lt(\hat{\Phi}^2_{0,s}+\hat{\Phi}^2_{0}\rt)+ \hat{\Phi}^2_{0,ss}+ \dl\int_{\hat{\O}}  e^{-s}\lt((1+s^{-2})\hat{v}^2+\hat{u}^2_{\neq}\rt)\nn\\
     \ls &\e^2\dl\int_{\hat{\O}} e^{-s}\hat{\Phi}^2_{0,ss}+ \dl\int_{\hat{\O}}  e^{-3s}\lt((1+s^{-2})\hat{\Phi}^2_\th+(e^s\hat{u}_{\neq})^2\rt)\nn\\
     =&\e^2\dl\int_{\hat{\O}} e^{-s}\hat{\Phi}^2_{0,ss}+ \dl\int_{\hat{\O}}  e^{-3s}\lt((1+s^{-2})\hat{\Phi}^2_\th+\Phi^2_{\neq,s}\rt)\nn\\
     \ls&\e^2\dl\int_{\hat{\O}} e^{-s}\hat{\Phi}^2_{0,ss}+ \dl\int_{\hat{\O}} e^{-3s}|\lt(\hat{\Phi}^2_{\th\th}+\hat{\Phi}^2_{\th s}\rt).  \label{vczeromode3}
\end{align}

Inserting \eqref{vczeromode3} into \eqref{vczeromode1} and using Cauchy inequality and Hardy inequality \eqref{vchardy2} for the last term, we obtain  \eqref{vczeroesti}.\qed

The weighted energy estimate for the non zero frequency of $\hat{\Phi}$ will be carried out in the framework of the stream function, under the reformulation of \eqref{vcerrorequ}. We have the following Lemma.

\begin{lemma}\label{lemhener}
 Let $(u,v)$ be a smooth solution of (\ref{errorequ}), then there exist $\epsilon_0>0, \dl_0>0$ such that for any $\epsilon\in (0,\epsilon_0)$ and $\dl\in(0,\dl_0)$, there holds
\begin{align}\label{estihener}
&\e^2 \int_{\hat{\O}}  \lt(e^{2s}\hat{\Phi}^2_{h,ss}+e^{2s}\hat{\Phi}^2_{h,\th s}+e^{2s}\hat{\Phi}^2_{h,\th\th}+e^{s}\hat{\Phi}^2_{1,ss}+e^{s}\hat{\Phi}^2_{1,\th s}+e^{s}\hat{\Phi}^2_{1,\th\th}\rt)\nn\\
\ls&  \int_{\hat{\O}}\lt(e^{2s}\hat{\Phi}^2_{h,\th\th}+e^{2s}\hat{\Phi}^2_{h,\th s}+e^{s}\hat{\Phi}^2_{1,\th\th}+e^{s}\hat{\Phi}^2_{1,\th s}\rt)+\e^2\dl \int_{\hat{\O}}\hat{\Phi}^2_{0,ss}+\hat{\Phi}^2_{0,s}+\e^{-2}\int_{\hat{\O}}e^{8s}\hat{R}^2_{\text{comb}}.
\end{align}
\end{lemma}
\pf

Multiplying \eqref{vcerrorequ} by $e^{3s} \hat{\Phi}_{h}+e^{2s} \hat{\Phi}_{1}$ and integrating the resulted equation on $\hat{\O}:\bT\times[0,+\i)$ to obtain that
\begin{align}
&\e^2 \int_{\hat{\O}} [e^{3s} \hat{\Phi}_{h}+e^{2s} \hat{\Phi}_{1}]\lt(\Dl^2_s\hat{\Phi}-4\Dl_s\hat{\Phi}_s+4\Dl_s\hat{\Phi}\rt)\nn\\
&-\int_{\hat{\O}}  [e^{3s} \hat{\Phi}_{h}+e^{2s} \hat{\Phi}_{1}][\hat{u}^a \Dl_s \hat{\Phi}_{\th}+ \hat{v}^a \Dl_s \hat{\Phi}_s-2  \hat{v}^a \Dl_s\hat{\Phi}]\nn\\
&-\int_{\hat{\O}}[e^{2s} \hat{\Phi}_{h}+e^{s} \hat{\Phi}_{1}][ \hat{u} \Dl_s \hat{\Phi}^a_\th+ \hat{v} \Dl_s \hat{\Phi}^a_s -2  \hat{v} \Dl_s\hat{\Phi}^a]=\int_{\hat{\O}} [e^{5s} \hat{\Phi}_{h}+e^{4s} \hat{\Phi}_{1}]\hat{R}_{\text{comb}}.\label{henergy0}
\end{align}
We will estimate terms in \eqref{henergy0} one by one.

{\bf\noindent Estimates of $\boldsymbol{\e^2 \int_{\hat{\O}}[e^{2s} \hat{\Phi}_{h}+e^{s} \hat{\Phi}_{1}]\lt(\Dl^2_s\hat{\Phi}-4\Dl_s\hat{\Phi}_s+4\Dl_s\hat{\Phi}\rt)}$}:

Direct integration by parts indicate that

\begin{align}
&\e^2 \int_{\hat{\O}} e^{2s} \hat{\Phi}_{h}\lt(\Dl^2_s\hat{\Phi}-4\Dl_s\hat{\Phi}_s+4\Dl_s\hat{\Phi}\rt)\nn\\
=& \e^2\int_{\hat{\O}}  e^{2s} \lt[ \hat{\Phi}^2_{h,\th\th}+2\hat{\Phi}^2_{h,\th s}+\hat{\Phi}^2_{h,ss}\rt]+ \e^2\int_{\hat{\O}} e^{2s}\lt[ -12\hat{\Phi}^2_{h,\th}-24\hat{\Phi}^2_{h,s}+32\hat{\Phi}^2_{h}\rt]\nn\\
\geq & \e^2 \int_{\hat{\O}}  e^{2s}\lt[ \hat{\Phi}^2_{h,\th\th}+2\hat{\Phi}^2_{h,\th s}+\hat{\Phi}^2_{h,ss}\rt]-C \e^2\int_{\hat{\O}}  e^{2s}\int_{\hat{\O}}  \lt[ \hat{\Phi}^2_{h,\th}+\hat{\Phi}^2_{h,s}+\hat{\Phi}^2_{h}\rt].\nn
\end{align}
Similarly, by integration by parts, we can see that
\begin{align}
&\e^2 \int_{\hat{\O}} e^{s} \hat{\Phi}_{1}\lt(\Dl^2_s\hat{\Phi}-4\Dl_s\hat{\Phi}_s+4\Dl_s\hat{\Phi}\rt)\nn\\
\geq & \e^2 \int_{\hat{\O}}  e^{s}\lt[ \hat{\Phi}^2_{1,\th\th}+2\hat{\Phi}^2_{1,\th s}+\hat{\Phi}^2_{1,ss}\rt]-C \e^2\int_{\hat{\O}}  e^{s}\int_{\hat{\O}}  \lt[ \hat{\Phi}^2_{1,\th}+\hat{\Phi}^2_{1,s}+\hat{\Phi}^2_{1}\rt].\nn
\end{align}
Combining the above two, we see that
\begin{align}
&\e^2 \int_{\hat{\O}}[e^{2s} \hat{\Phi}_{h}+e^{s} \hat{\Phi}_{1}]\lt(\Dl^2_s\hat{\Phi}-4\Dl_s\hat{\Phi}_s+4\Dl_s\hat{\Phi}\rt)\nn\\
\gs& \e^2 \int_{\hat{\O}} \lt[ e^{2s}\hat{\Phi}^2_{h,\th\th}+2 e^{2s}\hat{\Phi}^2_{h,\th s}+ e^{2s}\hat{\Phi}^2_{h,ss}+  e^{s}\hat{\Phi}^2_{1,\th\th}+2 e^{s}\hat{\Phi}^2_{1,\th s}+ e^{s}\hat{\Phi}^2_{1,ss}\rt]\nn\\
  &-C \e^2 \int_{\hat{\O}}  \lt[  e^{2s}\hat{\Phi}^2_{h,\th}+ e^{2s}\hat{\Phi}^2_{h,s}+ e^{2s}\hat{\Phi}^2_{h}+ e^{s}\hat{\Phi}^2_{1,\th}+ e^{s}\hat{\Phi}^2_{1,s}+ e^{s}\hat{\Phi}^2_{1}\rt]. \label{henergy1}
\end{align}

{\bf\noindent Estimates of $\boldsymbol{-\int_{\hat{\O}}  [e^{3s} \hat{\Phi}_{h}+e^{2s} \hat{\Phi}_{1}][\hat{u}^a \Dl_s \hat{\Phi}_{\th}+ \hat{v}^a \Dl_s \hat{\Phi}_s-2  \hat{v}^a \Dl_s\hat{\Phi}]}$}:
Using integration by parts, the incompressibility, the third estimates of \eqref{vcappdetail0} and \eqref{vcappdetail1}, Cauchy inequality, we obtain that
\begin{align}
&-\int_{\hat{\O}} e^{3s} \hat{\Phi}_{h}[\hat{u}^a \Dl_s \hat{\Phi}_{\th}+ \hat{v}^a \Dl_s \hat{\Phi}_s-2  \hat{v}^a \Dl_s\hat{\Phi}]\nn\\
=&\int_{\hat{\O}} e^{3s}\Dl_s\hat{\Phi} \lt(\hat{u}^a\hat{\Phi}_{h,\th}+ \hat{v}^a\hat{\Phi}_{h,s}+4 \hat{v}^a\hat{\Phi}_{h}\rt)\nn\\
=&(\t{\o}+\mathcal{O}(\e\dl)\mathcal{A}_0) \int_{\hat{\O}} e^{2s} \Dl_s\hat{\Phi}_h\hat{\Phi}_{h,\th}+ \int_{\hat{\O}} e^{3s}\lt(\hat{u}^a-(\t{\o}+\mathcal{O}(\e\dl)\mathcal{A}_0)e^{-s}\rt)\Dl_s\hat{\Phi} \hat{\Phi}_{h,\th}\nn\\
 &+ \int_{\hat{\O}} e^{3s} \Dl_s \hat{\Phi} \hat{v}^a\lt(\hat{\Phi}_{h,s}+4\hat{\Phi}_{h}\rt)\nn\\
 \ls&  \int_{\hat{\O}} e^{2s} \lt(\hat{\Phi}^2_{h,s}+\hat{\Phi}^2_{h,\th}\rt)+\e\dl \int_{\hat{\O}} e^{s}(1+s^{-1})|\Dl_s\hat{\Phi}| |\hat{\Phi}_{h,\th}|+\e\dl \int_{\hat{\O}} e^{s} |\Dl_s \hat{\Phi}| \lt(|\hat{\Phi}_{h,s}|+|\hat{\Phi}_{h}|\rt)\nn\\
 \ls&   \int_{\hat{\O}} e^{2s} \lt(\hat{\Phi}^2_{h,\th s}+\hat{\Phi}^2_{h,\th\th}\rt)+ \e^2\dl  \int_{\hat{\O}} \hat{\Phi}^2_{ss}.\nn
\end{align}
Similarly, we can obtain that
\begin{align}
&-\int_{\hat{\O}} e^{2s} \hat{\Phi}_{h}[\hat{u}^a \Dl_s \hat{\Phi}_{\th}+ \hat{v}^a \Dl_s \hat{\Phi}_s-2  \hat{v}^a \Dl_s\hat{\Phi}]\ls \int_{\hat{\O}} e^{s} \lt(\hat{\Phi}^2_{1,\th s}+\hat{\Phi}^2_{1,\th\th}\rt)+ \e^2\dl  \int_{\hat{\O}} \hat{\Phi}^2_{ss}.\nn
\end{align}
Combining the above two, we obtain that
\begin{align}
&\lt|-\int_{\hat{\O}}  [e^{2s} \hat{\Phi}_{h}+e^{s} \hat{\Phi}_{1}][\hat{u}^a \Dl_s \hat{\Phi}_{\th}+ \hat{v}^a \Dl_s \hat{\Phi}_s-2  \hat{v}^a \Dl_s\hat{\Phi}]\rt|\nn\\
\ls& \int_{\hat{\O}}  \lt(e^{2s}\hat{\Phi}^2_{h,\th s}+e^{2s}\hat{\Phi}^2_{h,\th\th}+ e^{s}\hat{\Phi}^2_{1,\th s}+e^s\hat{\Phi}^2_{1,\th\th}\rt)+ \e^2\dl  \int_{\hat{\O}} \hat{\Phi}^2_{ss}. \label{henergy2}
\end{align}

{\bf\noindent Estimates of $\boldsymbol{-\int_{\hat{\O}}[e^{3s} \hat{\Phi}_{h}+e^{2s} \hat{\Phi}_{1}][ \hat{u} \Dl_s \hat{\Phi}^a_\th+ \hat{v} \Dl_s \hat{\Phi}^a_s -2  \hat{v} \Dl_s\hat{\Phi}^a]}$}:

Using \eqref{vcpositive3x2}, \eqref{vcpositive3x9}, Cauchy inequality and Hardy inequality \eqref{vchardy1}, \eqref{vchardy2}, we see that
\begin{align}
&-\int_{\hat{\O}}e^{3s} \hat{\Phi}_{h}[ \hat{u} \Dl_s \hat{\Phi}^a_\th+ \hat{v} \Dl_s \hat{\Phi}^a_s -2  \hat{v} \Dl_s\hat{\Phi}^a]\nn\\
=&-\int_{\hat{\O}} e^{2s} \hat{\Phi}_{h} \hat{\Phi}_s \Dl_s \hat{\Phi}^a_\th+\int_{\hat{\O}} e^{2s}\hat{\Phi}_{h} \hat{\Phi}_\th \Dl_s \hat{\Phi}^a_s +2\int_{\hat{\O}} e^{2s} \hat{\Phi}_{h} \hat{\Phi}_\th \Dl_s\hat{\Phi}^a\nn\\
\ls& \e\dl\int_{\hat{\O}} e^{s} (1+s^{-2})|\hat{\Phi}_{h}| |\hat{\Phi}_s|+\dl\int_{\hat{\O}} e^{s}(1+s^{-2}) |\hat{\Phi}_{h}| |\hat{\Phi}_\th|\nn\\
\ls &\dl \int_{\hat{\O}} e^{2s} \lt(\hat{\Phi}^2_{h,\th s}+\hat{\Phi}^2_{h,\th\th}\rt)+\e^2\dl \int_{\hat{\O}} \lt(\hat{\Phi}^2_{s s}+\hat{\Phi}^2_{s}\rt).\nn
\end{align}

Similarly, we can obtain that
\begin{align}
&\lt|-\int_{\hat{\O}}e^{2s} \hat{\Phi}_{1}[ \hat{u} \Dl_s \hat{\Phi}^a_\th+ \hat{v} \Dl_s \hat{\Phi}^a_s -2  \hat{v} \Dl_s\hat{\Phi}^a]\rt|\ls \int_{\hat{\O}} e^{s} \lt(\hat{\Phi}^2_{1,\th s}+\hat{\Phi}^2_{1,\th\th}\rt)+ \e^2\dl  \int_{\hat{\O}} \hat{\Phi}^2_{ss}+\hat{\Phi}^2_{s}.\nn
\end{align}
Combining the above two, we obtain that
\begin{align}
&\lt|-\int_{\hat{\O}}[e^{3s} \hat{\Phi}_{h}+e^{2s} \hat{\Phi}_{1}][ \hat{u} \Dl_s \hat{\Phi}^a_\th+ \hat{v} \Dl_s \hat{\Phi}^a_s -2  \hat{v} \Dl_s\hat{\Phi}^a]\rt|\nn\\
\ls& \int_{\hat{\O}}  \lt(e^{2s}\hat{\Phi}^2_{h,\th s}+e^{2s}\hat{\Phi}^2_{h,\th\th}+ e^{s}\hat{\Phi}^2_{1,\th s}+e^s\hat{\Phi}^2_{1,\th\th}\rt)+ \e^2\dl  \int_{\hat{\O}} \lt(\hat{\Phi}^2_{ss}+ \hat{\Phi}^2_{s}\rt). \label{henergy3}
\end{align}

Inserting \eqref{henergy1}, \eqref{henergy2} and \eqref{henergy3} into \eqref{henergy0} and using Cauchy inequality for the last term,  we obtain \eqref{estihener}, \qed

\subsection{\texorpdfstring{$\dot{H}^3$ and $\dot{H}^4$ energy estimates for the stream function $\hat{\Phi}$}{Hdot3 and Hdot4 energy estimates for the stream function Phi}} \label{subsec3.2.1}

We first following the proof of Lemma \ref{lemhener} to give the tangential third and fourth derivatives' estimate of the solution. Then by using Bogoviskii Lemma and a cutting-gluing to give the weighted estimate for the tangential derivative of the pressure, i.e. $\hat{p}_\th$, which will induce the third normal derivative estimate for the stream function, i.e. $\hat{\Phi}_{sss}$. Then a direct application of equation of \eqref{vcerrorequ} give the fourth normal derivative estimate for the stream function, i.e. $\hat{\Phi}_{ssss}$.

\subsubsection{\texorpdfstring{The weighted tangential derivatives' estimates of $\hat{\Phi}$ up to the fourth order}{Weighted tangential derivative estimates of Phi up to fourth order}}

Almost the same as Lemma \ref{lemhener}, we have the following Lemma.
\begin{lemma}\label{propcombener2}
 Let $(\hat{u},\hat{v})$ be a smooth solution of (\ref{vcerrorequuv}), then there exist $\epsilon_0>0, \dl_0>0$ such that for any $\epsilon\in (0,\epsilon_0)$ and $ \dl\in(0,\dl_0)$, there holds
\begin{align}\label{esticombener}
&\e^4 \int_{\hat{\O}} \lt( e^{s}\hat{\Phi}^2_{1,\th ss}+ e^{s}\hat{\Phi}^2_{1,\th\th s}+ e^{s}\hat{\Phi}^2_{1,\th\th\th}+e^{2s}\hat{\Phi}^2_{h,\th ss}+e^{2s}\hat{\Phi}^2_{h,\th \th s}+e^{2s}\hat{\Phi}^2_{h,\th\th\th}\rt)\nn\\
\ls& C\e^2\int_{\hat{\O}} \lt( e^{s}\hat{\Phi}^2_{1,\th\th}+e^{s}\hat{\Phi}^2_{1,\th s}+e^{2 s} \hat{\Phi}^2_{h,\th\th}+e^{2 s}\hat{\Phi}^2_{h,\th s}\rt)\nn\\
+&\e^2 \int_{\hat{\O}} \lt(\hat{\Phi}^2_{ss}+\hat{\Phi}^2_{s}\rt)+\int_{\hat{\O}}e^{8s}\hat{R}^2_{\text{comb}},
\end{align}
and
\begin{align}\label{esticombener2}
&\e^6 \int_{\hat{\O}}  \lt( e^{s}\hat{\Phi}^2_{1,\th\th ss}+ e^{s}\hat{\Phi}^2_{1,\th\th\th s}+ e^{s}\hat{\Phi}^2_{1,\th\th\th\th}+e^{2s}\hat{\Phi}^2_{h,\th\th ss}+e^{2s}\hat{\Phi}^2_{h,\th\th \th s}+e^{2s}\hat{\Phi}^2_{h,\th\th\th\th}\rt)\nn\\
\ls& C\e^4\int_{\hat{\O}}  \lt( e^{s}\hat{\Phi}^2_{1,\th\th\th}+e^{s}\hat{\Phi}^2_{1,\th\th s}+e^{2 s} \hat{\Phi}^2_{h,\th\th\th}+e^{2 s}\hat{\Phi}^2_{h,\th\th s}\rt)\nn\\
+& \e^4\int_{\hat{\O}} \lt(\hat{\Phi}^2_{ss}+\hat{\Phi}^2_{s}\rt)+\e^2\int_{\hat{\O}} e^{8s}\hat{R}^2_{\text{comb}}.
\end{align}
\end{lemma}

\pf By multiplying \eqref{vcerrorequ} by $\e^2[e^{3s} \hat{\Phi}_{h,\th\th}+e^{2s} \hat{\Phi}_{1,\th\th}]$ and $\e^4[e^{3s} \hat{\Phi}_{h,\th\th\th\th}+e^{2s} \hat{\Phi}_{1,\th\th\th\th}]$ and then integrating the resulted equations in $\hat{\O}$, we can obtain \eqref{esticombener} and \eqref{esticombener2} by following the proof of Lemma \ref{lemhener}. \qed

{\noindent \bf Proof of \eqref{linearstability} in Proposition \ref{proplinearstability}. }

First by adding \eqref{vcestipositive} in Lemma \ref{lempositive}, \eqref{vczeroesti} in Lemma \ref{lemlinear} and \eqref{estihener} in Lemma \ref{lemhener}, we obtain that
\begin{align*}
&\|(e^s\hat{\Phi}_{h,\th\th},e^s\hat{\Phi}_{h,\th s},\e e^s\hat{\Phi}_{h, ss},e^{0.5s}\hat{\Phi}_{1,\th\th},e^{0.5s}\hat{\Phi}_{1,\th s}, \e e^{0.5s}\hat{\Phi}_{1, ss},\e \hat{\Phi}_{0,ss}, \e \hat{\Phi}_{0,s} )\|^2_{L^2(\hat{\O})}\nn\\
\ls& \e^{-2}\|e^{4s}\hat{R}_{\text{comb}}\|^2_{L^2{(\hat{\O})}}+\e^{-2}\|e^{1.5s} R_{\hat{u}}\|^2_{L^2{(\hat{\O})}}.
\end{align*}
Adding the above estimates with \eqref{esticombener} and \eqref{esticombener2} in Proposition \ref{propcombener2}, we achieve the first estimate \eqref{linearstability} in Proposition \ref{proplinearstability}.

\subsubsection{The weighted estimates of $\hat{\Phi}_{sss}$. }\label{subsec3.2.2}

 We first use the Bogovskii Lemma and a cutting-gluing  technique to give the weighted estimate for the tangential derivative of the pressure, i.e. $\hat{p}_\th$, and then give the third normal derivative estimate for the stream function, i.e. $\hat{\Phi}_{sss}$. We have the following Lemma.

\begin{lemma}\label{lemlinearss}
  Let $(\hat{u},\hat{v})$ be a smooth solution of (\ref{vcerrorequuv}). Then there exist $\epsilon_0>0$ and $\dl_0>0$ such that, for any $\epsilon\in(0,\epsilon_0)$ and $\dl\in(0,\dl_0)$, there hold
\begin{align}
&\e^{10}\|\hat{\Phi}_{0,sss},e^{0.5s} \hat{\Phi}_{\th sss} \|^2\ls  \e^6\|(\hat{\Phi}_{0,s},\hat{\Phi}_{0,ss}\|^2+\e^6\sum_{\substack{ i+j\leq 4\\i\neq0,j\leq 2 }}\|e^{0.5s}\p^i_\th\p^j_s\hat{\Phi}\|^2+\e^6\|e^{2.5s}(\p_\th R_{\hat{u}},\p_\th R_{\hat{v}}, R_{\hat{u}})\|^2_{L^2(\hat{\O})}. \label{sssesti8}
\end{align}
\end{lemma}

The third normal derivative estimates correspond to the second normal derivative of $\hat{u}$. From \eqref{vcerrorequuv}$_1$, we need  the weighted $L^2$ estimate of $\hat{p}_{\th}$, which can be obtained by applying the following Bogovskii Lemma.

\begin{lemma}\label{lembogov}
 Define the domain $O_k:=\bT \times[k,k+1]$, $0\leq k\in\bN$, then for $f\in L^2(O_k)$ with
\bes
\int_{O_k} f(\th,s)d\th ds=0,
\ees
there exists a vector $V=(V^\th,V^s)\in H^1_0(O_k)$, s. t.
\bes
\bali
&\partial_\theta V^\th+\partial_s V^s=f,\q \| V\|_{H^1(O_k)}\leq  C \|f\|_{L^2(O_k)}.
\eali
\ees
where $\na=(\p_\th,\p_s)$ and the constant $C$ is independent of $k$.
\end{lemma}

Here we only stated a special case of Lemma III.3.1 of \cite{Galdi:2011} which works for general domains and the constant $C$ depends on the ratio of the diameter and inner radius of the domain. For the above domain $O_k$, the diameter and inner radius are fixed, so we have an absolute constant. Note that the Bogovskii function may not be unique. For our purpose, we will choose and fix one for the relevant domain.

{\bf\noindent Proof of Lemma \ref{lemlinearss}} Using Lemma \ref{lembogov} for $f=\hat{p}_\th$, there exists a vector $V=(V^\th,V^s)\in H^1_0(O_k)$, s. t.
\be\label{bogov2}
\bali
&\partial_\theta V^\th+\partial_s V^s=\hat{p}_\th,\q \|V\|_{H^1(O_k)}\leq  C \|\hat{p}_\th\|_{L^2(O_k)}.
\eali
\ee

By taking $\th$ derivative on \ref{vcerrorequuv}$_{1,2}$, we obtain that
\be\label{sssesti1}
\begin{aligned}
&-\e^2e^{-s}\lt(\Dl_s \hat{u}_\th-\hat{u}_\th+2\hat{v}_{\th\th}\rt)+\hat{p}_{\th\th}+\p_\th S_{\hat{u}}=\p_\th R_{\hat{u}},\\
&-\e^2e^{-s}\lt(\Dl_s \hat{v}_\th-\hat{v}_\th-2\hat{u}_{\th\th}\rt)+ \hat{p}_{\th s}+\p_\th S_{\hat{v}}=\p_\th R_{\hat{v}},
\end{aligned}
\ee
By direct calculation, we have
\begin{align}
&\Dl_s \hat{u}-\hat{u}+2\hat{v}_{\th}=e^{-s}\lt(\hat{\Phi}_{\th\th s}+\hat{\Phi}_{sss}-2\hat{\Phi}_{s s}-2\hat{\Phi}_{\th\th}\rt),\label{sssesti2}\\
&\Dl_s \hat{v}-\hat{v}-2\hat{u}_{\th}=-e^{-s}\lt(\hat{\Phi}_{\th\th \th}+\hat{\Phi}_{\th ss}\rt). \label{sssesti3}
\end{align}
Multiplying \eqref{sssesti1}$_1$ by$-V^\th $, \eqref{sssesti1}$_2$ by$-V^s $, and integrations on $O_k$ to obtain
\begin{align}
& \int_{O_k} \hat{p}^2_{\theta}=\int \hat{p}_\th  \lt(\partial_\theta V^\th+\partial_s V^s\rt)=  \int_{O_k} \left(-\hat{p}_{\th\th}V^{\th}-\hat{p}_{s\theta} V^s\right)\nn\\
=& -\epsilon^2 \int_{O_k}e^{-s}\lt(\Dl_s \hat{u}_\th-\hat{u}_\th+2\hat{v}_{\th\th}\rt) V^\th-\epsilon^2\int_{O_k} e^{-s}\lt(\Dl_s \hat{v}_\th-\hat{v}_\th-2\hat{u}_{\th\th}\rt) V^s \nn\\
& + \int_{O_k}\lt(\partial_\theta S_{\hat{u}} V^\th+\partial_\theta S_{\hat{v}} V^s\rt) -\int_{O_k} \lt(\partial_\theta R_{\hat{u}} V^\th+\partial_\theta R_{\hat{v}} V^s\rt).\nn
\end{align}
By using \eqref{sssesti2}, integration by parts, Hardy inequality, Poincar\'{e} inequality and \eqref{bogov2}, we have
\begin{align}
&-\e^2\int_{O_k}e^{-s}\lt(\Dl_s \hat{u}_\th-\hat{u}_\th+2\hat{v}_{\th\th}\rt) V^\th   \nn\\
=&-\e^2\int_{O_k} e^{-2s}\lt(\hat{\Phi}_{\th\th\th s}+\hat{\Phi}_{\th sss}-2\hat{\Phi}_{\th s s}-2\hat{\Phi}_{\th \th\th}\rt) V^\th\nn\\
\ls& \e^2\|e^{-2s}(\hat{\Phi}_{\th\th s},\hat{\Phi}_{\th ss},\hat{\Phi}_{\th \th \th} )\|_{L^2(O_k)}\|\nabla V^\th\|_{L^2(O_k)}.\nn
\end{align}
The same by using \eqref{sssesti3}, integration by parts, Hardy inequality, Poincar\'{e} inequality and \eqref{bogov2}, we can obtain that
\begin{align}
&-\e^2\int_{O_k} e^{-s}\lt(\Dl_s \hat{v}_\th-\hat{v}_\th-2\hat{u}_{\th\th}\rt) V^s \nn\\
=&-\e^2\int_{O_k} e^{-2s}\lt(\hat{\Phi}_{\th\th\th \th}+\hat{\Phi}_{\th\th ss}\rt) V^s\ls \e^2\|e^{-2s}(\hat{\Phi}_{\th\th \th},\hat{\Phi}_{\th ss} )\|_{L^2(O_k)}\|\nabla V^s\|_{L^2(O_k)}.\nn
\end{align}
By using integration by parts and Hardy inequality, we have
\begin{align}
\begin{array}{ll}
\begin{aligned}
&\int_{O_k}\lt(\partial_\theta S_{\hat{u}} V^\th+\partial_\theta S_{\hat{v}} V^s\rt)\nn\\
\ls& \|S_{\hat{u}} \|_{L^2(O_k)}\|\p_\th V^\th\|_{L^2(O_k)}+\|S_{\hat{v}} \|_{L^2(O_k)}\|\p_\th V^s\|_{L^2(O_k)}\nn\\
\ls&  \|(S_{\hat{u}},S_{\hat{v}})\|_{L^2(O_k)}\|\hat{p}_\th \|_{L^2(O_k)}.
\end{aligned}
&
\begin{aligned}
&\int_{O_k}\lt(\partial_\theta R_{\hat{u}} V^\th+\partial_\theta R_{\hat{v}} V^s\rt)\nn\\
\ls& \|R_{\hat{u}} \|_{L^2(O_k)}\|\p_\th V^\th\|_{L^2(O_k)}+\|\p_\th R_{\hat{v}} \|_{L^2(O_k)}\| V^s\|_{L^2(O_k)}\nn\\
\ls&  \|(R_{\hat{u}},\p_\th R_{\hat{v}})\|_{L^2(O_k)}\|\hat{p}_\th \|_{L^2(O_k)}.
\end{aligned}
\end{array}
\end{align}
Combing the above estimates, we obtain that
\begin{align}
\|p_\th \|^2_{L^2(O_k)}\ls& \e^4 e^{-4k} \lt\|\lt(\hat{\Phi}_{\th\th\th },\hat{\Phi}_{\th\th s},\hat{\Phi}_{\th s s} \rt)\rt\|^2_{L^2(O_k)} \nn\\
                    &+\|(S_{\hat{u}},S_{\hat{v}},R_{\hat{u}},\p_\th R_{\hat{v}})\|^2_{L^2(O_k)}.\nn
\end{align}
Since in $O_k$, $e^{-s}\approx e^{-k}$.   By multiplying the above inequality by $e^{5k}$, we obtain that
\begin{align}
\|p_\th e^{2.5s}\|^2_{L^2(O_k)}\ls&\e^4  \lt\|e^{0.5s}\lt(\hat{\Phi}_{\th\th\th },\hat{\Phi}_{\th\th s},\hat{\Phi}_{\th s s} \rt)\rt\|^2_{L^2(O_k)}+\|e^{2.5s}(S_{\hat{u}},S_{\hat{v}},R_{\hat{u}},\p_\th R_{\hat{v}})\|^2_{L^2(O_k)}. \nn
\end{align}
Summing over $k$ from $0$ to $+\i$, we obtain

\begin{align}
\|p_\th e^{2.5s} \|^2 \ls& \e^4\lt\|e^{0.5s}\lt(\hat{\Phi}_{\th\th\th },\hat{\Phi}_{\th\th s},\hat{\Phi}_{\th s s} \rt)\rt\|^2_{L^2(\hat{\O})}+\|e^{2.5s}(S_{\hat{u}},S_{\hat{v}},R_{\hat{u}},\p_\th R_{\hat{v}})\|^2_{L^2(\hat{\O})}. \label{sssesti4}
\end{align}
Then from the equation \eqref{sssesti2} and \eqref{vcerrorequuv}$_1$, we have
\begin{align}
-\e^2 \Phi_{sss}=\epsilon^2\lt( \Phi_{\th\th s}-2 \Phi_{s s}-2 \Phi_{\th\th}\rt)+e^{2s}\lt( \hat{p}_\th+  S_{\hat{u}}- R_{\hat{u}}\rt),\label{sssesti9}
\end{align}
which, after using the Hardy inequality, implies that
\begin{align}
\e^4\|e^{0.5s} \Phi_{\neq,sss} \|^2\ls \e^4 \|e^{0.5s}(\hat{\Phi}_{\th\th s},\hat{\Phi}_{\neq, s s},\hat{\Phi}_{\th\th})\|^2+\|p_\th e^{2.5s}\|^2+\|e^{2.5s}(S_{\hat{u}},R_{\hat{u}})\|^2. \label{sssesti5}
\end{align}
Also direct by taking $L^2$ norm of \eqref{sssesti9}, we have
\begin{align}
\e^4\| \Phi_{0,sss} \|^2\ls \e^4 \|\hat{\Phi}_{0, s s}\|^2+\|e^{2s}(S_{\hat{u}0},R_{\hat{u}0})\|^2. \label{sssesti10}
\end{align}
Combining \eqref{sssesti4} and \eqref{sssesti5}, we arrive at
\begin{align}
\e^4\|e^{0.5s} \Phi_{\neq,sss} \|^2\ls \e^4 \lt(\|e^{0.5s} (\hat{\Phi}_{\th\th s},\hat{\Phi}_{\th s s},\hat{\Phi}_{\th\th\th})\|^2\rt)+\|e^{2.5s}(S_{\hat{u}},S_{\hat{v}},R_{\hat{u}},\p_\th R_{\hat{v}})\|^2_{L^2(\hat{\O})}. \label{sssesti6}
\end{align}
From the representations of $S_{\hat{u}}$ and $S_{\hat{v}}$, together with estimates \eqref{vcappdetail0} and \eqref{vcappdetail1}, we obtain
\be
\|e^{2s}S_{\hat{u}0}\|^2+ \|e^{2.5s}\lt(S_{\hat{u}},S_{\hat{v}}\rt) \|^2\ls \|(\hat{\Phi}_{0,s},\hat{\Phi}_{0,ss},e^{0.5s}\hat{\Phi}_{\th s},e^{0.5s}\hat{\Phi}_{\th\th} )\|^2. \label{sssesti7}
\ee
Combining \eqref{sssesti10}, \eqref{sssesti6}, and \eqref{sssesti7}, we arrive at
\begin{align}
&\e^4\| \Phi_{0,sss}\|^2+\e^4\|e^{0.5s} \Phi_{\neq,sss} \|^2\ls  \|(\hat{\Phi}_{0,s},\hat{\Phi}_{0,ss}\|^2+\sum_{\substack{ i+j\leq 3\\i\neq0,j\neq 3}}\|e^{0.5s}\p^i_\th\p^j_s\hat{\Phi}\|^2+\|e^{2.5s}(R_{\hat{u}},\p_\th R_{\hat{v}})\|^2_{L^2(\hat{\O})}.\nn
\end{align}
 By estimates similar to those for $e^{0.5s} \Phi_{\neq,sss}$, we obtain
\begin{align}
&\e^4\|e^{0.5s} \Phi_{\th sss} \|^2\ls  \|(\hat{\Phi}_{0,s},\hat{\Phi}_{0,ss}\|^2+\sum_{\substack{ i+j\leq 4\\i\neq0,j\leq 2}}\|e^{0.5s}\p^i_\th\p^j_s\hat{\Phi}\|^2+\|e^{2.5s}(\p_\th R_{\hat{u}},\p_\th R_{\hat{v}})\|^2_{L^2(\hat{\O})}.\nn
\end{align}
The above two indicate \eqref{sssesti8} after multiplying $\e^6$. \qed

\subsubsection{The weighted estimates of $\hat{\Phi}_{ssss}$. }\label{subsec3.2.3}
 We have the following Lemma.

\begin{lemma}\label{lemlinearssss}
Let $(\hat{u},\hat{v})$ be a smooth solution of (\ref{vcerrorequ}), then there exist $\epsilon_0>0, \dl_0>0$ such that for any $\epsilon\in (0,\epsilon_0)$  and $ \dl\in(0,\dl_0)$, there hold
\begin{align}
\e^{16}\| \Phi_{0,ssss}, e^{0.5s}\Phi_{\neq,ssss} \|^2\ls \e^{10}\lt(  \|\Phi_{0,s},\Phi_{0,ss},\Phi_{0,sss}\|^2+\sum_{\substack{i+j\leq 4\\i\neq 0,j\neq 4}}\|e^{0.5s}\p^i_\th\p^j_s\hat{\Phi}\|^2\rt)+\e^{12}\|e^{3.5s}\hat{R}_{\text{comb}}\|^2_{L^2(\hat{\O})}. \label{ssssesti}
\end{align}
\end{lemma}

\pf From \eqref{vcerrorequ}, we see that
\begin{align}
&\e^2 e^{-s}\lt(\Dl^2_s\hat{\Phi}-4\Dl_s\hat{\Phi}_s+4\Dl_s\hat{\Phi}\rt)\nn\\
&-\lt[(\hat{u}^a\p_\th+\hat{v}^a\p_s)\Dl_s \hat{\Phi}-2v^a\Dl_s\hat{\Phi}\rt]-\lt[(\hat{u}\p_\th+\hat{v}\p_s)\Dl_s \hat{\Phi}^a-2\hat{v}\Dl_s\hat{\Phi}^a\rt]\nn\\
&=e^{2s}\hat{R}^a_{\o}+e^{2s}\hat{R}_\o:=e^{2s}\hat{R}_{\text{comb}}.\nn
\end{align}
Then, we obtain that
\begin{align}
&\e^2\hat{\Phi}_{ssss}=\e^2\lt(-\hat{\Phi}_{\th\th\th\th}-2\hat{\Phi}_{\th\th ss}+4 \Dl_s\hat{\Phi}_s-\Dl_s\hat{\Phi}\rt)\nn\\
&+e^{s}\lt[(\hat{u}^a\p_\th+\hat{v}^a\p_s)\Dl_s \hat{\Phi}-2v^a\Dl_s\hat{\Phi}\rt]+e^{s}\lt[(\hat{u}\p_\th+\hat{v}\p_s)\Dl_s \hat{\Phi}^a-2\hat{v}\Dl_s\hat{\Phi}^a\rt]+e^{3s}\hat{R}_{\text{comb}}.\nn
\end{align}

 Then, by taking weighted $L^2$ norms and using estimates \eqref{vcappdetail0} and \eqref{vcappdetail1}, we obtain
\begin{align}
\e^4\| \Phi_{0,ssss}, e^{0.5s}\Phi_{\neq,ssss} \|^2\ls \e^{-2}\lt(  \|\Phi_{0,s}\Phi_{0,ss},\Phi_{0,sss}\|^2+\sum_{\substack{i+j\leq 4\\i\neq 0,j\neq 4}}\|e^{0.5s}\p^i_\th\p^j_s\hat{\Phi}\|^2\rt)+\|e^{3.5s}\hat{R}_{\text{comb}}\|^2_{L^2(\hat{\O})}.\nn
\end{align}
which is \eqref{ssssesti} after multiplying $\e^{12}$.

{\noindent \bf Proof of \eqref{linearstability1} in Proposition \ref{proplinearstability}. }

By adding \eqref{ssssesti} in Lemma \ref{lemlinearssss} and \eqref{sssesti8} in Lemma \ref{lemlinearss}, we obtain \eqref{linearstability1} in Proposition \ref{proplinearstability}.

\subsection{\texorpdfstring{Refined weighted estimates for the Fourier one mode of $\hat{\Phi}$}{Refined weighted estimates for the Fourier-one mode of Phi}} \label{subsec3.3}

\begin{lemma} \label{lem1fposi}
 Let $\hat{\Phi}$ be a smooth solution of (\ref{vcerrorequ}), then there exist $\epsilon_0>0, \dl_0>0$ such that for any $\epsilon\in (0,\epsilon_0)$ and $\dl\in(0,\dl_0)$, there holds,
\begin{align}
&\e^{20} \int_{\hat{\O}} | (e^{s} \Dl_s\hat{\Phi}_1)_{,s}|^2+ \e^{18}\int^{\i}_0 | (e^{s} \Dl_s\Phi_{1})_{,\th}|^2\nn\\
\ls & \e^{16}\sum^4_{i=0}\int^\i_0|\hat{\Phi}^{(k)}_{0}|^2ds+\e^{16}\sum_{\substack{0\neq i+j\leq 4 \\i\neq 0}}\int_{\hat{\O}} e^{s} (\p^i_\th\p^j_s\Phi)^2+\e^{18}\int_{\hat{\O}}e^{8 s}\hat{R}^2_{comb}. \label{vc1festiposi}
\end{align}
Here the constant $C$ is independent of $\e_0$, $\dl_0$.
\end{lemma}
\pf Remembering the Fourier series representation of $\hat{\Phi}$, we have
\bes
\hat{\Phi}_1=A^{\hat{\Phi}}_1(s)\cos\th+B^{\hat{\Phi}}_1(s)\sin\th.
\ees
 Then noting $\hat{\Phi}_{1,\th}=-e^{s} \hat{v}_1$, direct calculation indicates that
\bes
e^s\Dl_s\hat{\Phi}_{1,\th}=(e^{2s}(e^{-s}B^{\hat{\Phi}}_1)_{,s})_{,s}\cos\th-(e^{2s}(e^{-s}A^{\hat{\Phi}}_1)_{,s})_{,s}\sin\th=- (e^{2s} v_{1,s})_{,s}.
\ees
For notation simplification, we denote
\be\label{vc1posi-2}
\hat{\Phi}_1={\mathfrak{f}}(s)\cos\th+{\mathfrak{g}}(s)\sin\th,\q e^{2s}(e^{-s}\mathfrak{f})'=\t{\mathfrak{f}},\q e^{2s}(e^{-s}\mathfrak{g})'=\t{\mathfrak{g}}.
\ee
Then by this notations, we have
\begin{align}
&e^s\Dl_s\hat{\Phi}_{1,\th}=(\t{\mathfrak{g}}'(s)\cos\th-\t{\mathfrak{f}}'(s)\sin\th=- (e^{2s} v_{1,s})_{,s},\label{vc1posi-4}\\
&(e^s\Dl_s\hat{\Phi}_1)_{,s}=(e^{2s} (e^{-s}\mathfrak{f})_{,s})_{,s}\cos\th+(e^{2s} (e^{-s}\mathfrak{g})_{,s})_{,s}\sin\th=\t{\mathfrak{f}}^{''}\cos\th+\t{\mathfrak{g}}^{''}\sin\th. \label{vc1posi-5}
\end{align}
Multiplying \eqref{vcerrorequ} by $-e^{3s}\Dl_s\hat{\Phi}_{1,\th}=e^{2s}\lt(\t{\mathfrak{g}'}\cos\th-\t{\mathfrak{f}'}\sin\th\rt)$ and integrating the resulted equation on $\hat{\O}:=\bT\times[0,+\i)$ to obtain that
\begin{align}
&\int_{\hat{\O}} e^{3s}\Dl_s\hat{\Phi}_{1,\th} \hat{u}^a\Dl_s \hat{\Phi}_{\th}+\int_{\hat{\O}}e^{3s}\Dl_s\hat{\Phi}_{1,\th} \hat{v}^a \Dl_s \hat{\Phi}_s -2\int_{\hat{\O}}e^{3s}\Dl_s\hat{\Phi}_{1,\th}  \hat{v}^a\Dl_s\hat{\Phi}\nn\\
&+\int_{\hat{\O}}e^{3s}\Dl_s\hat{\Phi}_{1,\th} \hat{u} \Dl_s \hat{\Phi}^a_\th+\int_{\hat{\O}}e^{3s}\Dl_s\hat{\Phi}_{1,\th} \hat{v} \Dl_s \hat{\Phi}^a_s -2\int_{\hat{\O}}e^{3s}\Dl_s\hat{\Phi}_{1,\th} \hat{v}\Dl_s\hat{\Phi}^a\nn\\
&-\e^2 \int_{\hat{\O}}e^{2s}\Dl_s\hat{\Phi}_{1,\th}\lt(\Dl^2_s\hat{\Phi}-4\Dl_s\hat{\Phi}_s+4\Dl_s\hat{\Phi}\rt)=-\int_{\hat{\O}} e^{5s}\Dl_s\hat{\Phi}_{1,\th} \hat{R}_{\text{comb}}.\label{vc1posi0}
\end{align}

{\bf\noindent Estimates of $\boldsymbol{\int_{\hat{\O}} e^{3s}\Dl_s\hat{\Phi}_{1,\th} \hat{u}^a\Dl_s \hat{\Phi}_{\th}}$}: The first term of the left hand of \eqref{vc1festiposi} comes from this term. By direct integration by parts and boundary condition for $\hat{\Phi}_{\th}$, we obtain that
\begin{align}
&\int_{\hat{\O}}e^{3s}\Dl_s\hat{\Phi}_{1,\th} \hat{u}^a\Dl_s \hat{\Phi}_{\th}\nn\\
=& \lt(\t{\o}+\mathcal{\e\dl}\mathcal{A}_0\rt)\int_{\hat{\O}}e^{2s}\Dl_s\hat{\Phi}_{1,\th}  \Dl_s \hat{\Phi}_{1,\th}+\int_{\hat{\O}} \lt[\hat{u}^a+\lt(\t{\o}-\mathcal{\e\dl}\mathcal{A}_0\rt)e^{-s}\rt] e^{3s}\Dl_s\hat{\Phi}_{1,\th} \Dl_s \hat{\Phi}_\th.\nn\\
=& \lt(\t{\o}+\mathcal{\e\dl}\mathcal{A}_0\rt)\int_{\hat{\O}}|e^s\Dl_s\hat{\Phi}_{1,\th}|^2+\int_{\hat{\O}} \lt[\hat{u}^a-\lt(\t{\o}+\mathcal{\e\dl}\mathcal{A}_0\rt)e^{-s}\rt] e^{3s}\Dl_s\hat{\Phi}_{1,\th} \Dl_s \hat{\Phi}_\th. \label{vc1posi-1}
\end{align}
Using integration by parts, estimates in \eqref{vcappdetail0}, Cauchy inequality, Poincar\'{e} inequality in $\th$ variable, Hardy inequality \eqref{vchardy1} and \eqref{vchardy2}, we have
\begin{align}
&\lt|-\int_{\hat{\O}} \lt[\hat{u}^a-\lt(\t{\o}+\mathcal{\e\dl}\mathcal{A}_0\rt)e^{-s}\rt]e^{3s}\Dl_s\hat{\Phi}_{1,\th}\Dl_s \hat{\Phi}_\th\rt|\nn\\
\ls &\dl \int_{\hat{\O}} \lt|e^{s}\Dl_s\hat{\Phi}_{1,\th}\Dl_s \hat{\Phi}_\th\rt|\ls \dl \int_{\hat{\O}} |e^s\Dl_s\hat{\Phi}_{1,\th}|^2+ \dl \int_{\hat{\O}}  |\Dl_s \hat{\Phi}_\th|^2. \label{vc1posi-3}
\end{align}
Inserting \eqref{vc1posi-3} into \eqref{vc1posi-1}, we achieve that for small $\e_0$ and $\dl_0$, when $\e\leq \e_0$ and $\dl\leq \dl_0$,
\begin{align}
\int_{\hat{\O}} e^{3s}\Dl_s\hat{\Phi}_{1,\th} \hat{u}^a\Dl_s \hat{\Phi}_{\th} \gs  \int_{\hat{\O}} |e^s\Dl_s\hat{\Phi}_{1,\th}|^2-\dl \int_{\hat{\O}} |\Dl_s \hat{\Phi}_\th|^2.\label{vc1posi1}
\end{align}

{\bf\noindent Estimates of $\boldsymbol{\int_{\hat{\O}}e^{3s}\Dl_s\hat{\Phi}_{1,\th}[ \hat{v}^a \Dl_s \hat{\Phi}_s-2 \hat{v}^a \Dl_s\hat{\Phi}]}$}: By Cauchy inequality and the third one of \eqref{vcappdetail1}, we have
\begin{align}
&\lt|\int_{\hat{\O}}e^{3s}\Dl_s\hat{\Phi}_{1,\th}[ \hat{v}^a \Dl_s \hat{\Phi}_s-2 \hat{v}^a \Dl_s\hat{\Phi}]\rt|\ls \dl \int_{\hat{\O}}|e^s\Dl_s\hat{\Phi}_{1,\th}|^2+\e^2\dl \int_{\hat{\O}}\lt( |\Dl_s\hat{\Phi}|^2+|\Dl_s\hat{\Phi}_s|^2\rt).\label{vc1posi2}
\end{align}

{\bf\noindent Estimates of $\boldsymbol{\int_{\hat{\O}}e^{3s}\Dl_s\hat{\Phi}_{1,\th} \hat{u} \Dl_s \hat{\Phi}^a_\th}$}:
 Using estimates \eqref{vcpositive3x2},  after rewriting this term and applying Cauchy inequality, we see that
\begin{align}
&\lt|\int_{\hat{\O}}e^{3s}\Dl_s\hat{\Phi}_{1,\th}\hat{u} \Dl_s \hat{\Phi}^a_\th\rt|=\lt|\int_{\hat{\O}} e^{2s}\Dl_s\hat{\Phi}_{1,\th} \hat{\Phi}_s \Dl_s \hat{\Phi}^a_\th\rt|\nn\\
					=&\lt|\int_{\hat{\O}}e^{2s}\Dl_s\hat{\Phi}_{1,\th}\hat{\Phi}_{0,s} \Dl_s (e^{s}(\hat{v}^a-{\red \hat{v}_e^a}))\rt|+\lt|\int_{\hat{\O}}e^{2s}\Dl_s\hat{\Phi}_{1,\th} \hat{\Phi}_{\neq,s} \Dl_s (e^{s}\hat{v}^a)\rt|\nn\\
\ls& \dl\int_{\hat{\O}}e^{-s} (1+s^{-1}) |e^s\Dl_s\hat{\Phi}_{1,\th}|| \hat{\Phi}_{0,s}|+ \dl\int_{\hat{\O}} (1+s^{-1}) |e^s\Dl_s\hat{\Phi}_{1,\th}|| \hat{\Phi}_{\neq,s}|\nn\\
\ls& \dl\int_{\hat{\O}}|e^s\Dl_s\hat{\Phi}_{1,\th}|^2+ \dl\int_{\hat{\O}}\lt(|\hat{\Phi}_{0,s}|^2+|\hat{\Phi}_{0,ss}|^2+e^s|\hat{\Phi}_{\neq,ss}|^2\rt). \label{vc1posi3x1}
\end{align}

{\bf\noindent Estimates of $\boldsymbol{\int_{\hat{\O}}e^{3s}\Dl_s\hat{\Phi}_{1,\th}[\hat{v} \Dl_s \hat{\Phi}^a_s -2 \hat{v} \Dl_s\hat{\Phi}^a]}$}:

Using estimates \eqref{vcpositive3x9}, Cauchy inequality, Poincar\'{e} inequality in $\th$ variable, Hardy inequality \eqref{vchardy1} and \eqref{vchardy2}, we achieve that
\begin{align}
&\lt|\int_{\hat{\O}}e^{3s}\Dl\hat{\Phi}_{1,\th}[\hat{v} \Dl_s \hat{\Phi}^a_s -2 v \Dl_s\hat{\Phi}^a]\rt|=\lt|\int_{\hat{\O}}e^{2s}\Dl_s\hat{\Phi}_{1,\th}[\hat{\Phi}_\th \Dl_s \hat{\Phi}^a_s -2 \hat{\Phi}_\th  \Dl_s\hat{\Phi}^a]\rt| \nn\\
\ls& \dl \int_{\hat{\O}} |e^s\Dl_s\hat{\Phi}_{1,\th}|^2 +\int_{\hat{\O}} (1+s^{-4}) | \hat{\Phi}_{\th}|^2\ls\dl \int_{\hat{\O}} |e^s\Dl_s\hat{\Phi}_{1,\th}|^2+ \dl\int_{\hat{\O}} e^{s}\lt( | \hat{\Phi}_{\th}|^2+| \hat{\Phi}_{\th ss}|^2\rt).  \label{vc1posi4}
\end{align}

{\bf\noindent Estimates of $\boldsymbol{-\e^2 \int_{\hat{\O}} e^{2s}\Dl_s\hat{\Phi}_{1,\th}\lt(\Dl^2_s\hat{\Phi}-4\Dl_s\hat{\Phi}_s+4\Dl_s\hat{\Phi}\rt)}$}:

From \eqref{vc1posi-2} and \eqref{vc1posi-4}, by direct calculation, we see that
\begin{align}
\Dl^2_s\hat{\Phi}_1-4\Dl_s\hat{\Phi}_{1,s}+4\Dl_s\hat{\Phi}_1=& \lt(\mathfrak{f}^{(4)}-4 \mathfrak{f}^{(3)}+2 \mathfrak{f}^{(2)}+4 \mathfrak{f}^{'}-3 \mathfrak{f}\rt)\cos\th+\lt(\mathfrak{g}^{(4)}-4 \mathfrak{g}^{(3)}+2 \mathfrak{g}^{(2)}+4 \mathfrak{g}^{'}-3 \mathfrak{g}\rt)\sin\th.\nn\\
      =&e^{-s} \lt(\t{\mathfrak{f}}^{(3)}-6 \t{\mathfrak{f}}^{(2)}-8\t{ \mathfrak{f}}'\rt)\cos\th+e^{-s}\lt(\t{\mathfrak{g}}^{(3)}-6 \t{\mathfrak{g}}^{(2)}-8\t{ \mathfrak{g}}'\rt)\sin\th.\label{vcposi8}\\
 -e^s\Dl_s\hat{\Phi}_{1,\th}=& -\t{\mathfrak{g}}'\cos\th+\t{\mathfrak{f}}'\sin\th.\nn
\end{align}
Then using integration by parts, we obtain that
\begin{align}
&-\e^2 \int_{\hat{\O}}e^{2s}\Dl_s\hat{\Phi}_{1,\th} \lt(\Dl^2_s\hat{\Phi}-4\Dl_s\hat{\Phi}_s+4\Dl_s\hat{\Phi}\rt)=-\e^2 \int_{\hat{\O}}e^{2s}\Dl_s\hat{\Phi}_{1,\th} \lt(\Dl^2_s\hat{\Phi}_{1}-4\Dl_s\hat{\Phi}_{1,s}+4\Dl_s\hat{\Phi}_{1}\rt)\nn\\
=& \e^2 \int_{\hat{\O}}e^{s}\lt(-\t{\mathfrak{g}}'\cos\th+\t{\mathfrak{f}}'\sin\th \rt)\lt\{ \lt(\mathfrak{f}^{(4)}-4 \mathfrak{f}^{(3)}+2 \mathfrak{f}^{(2)}+4 \mathfrak{f}^{'}-3 \mathfrak{f}\rt)\cos\th\rt.\nn\\
  &\qq \lt.+\lt(\mathfrak{g}^{(4)}-4 \mathfrak{g}^{(3)}+2 \mathfrak{g}^{(2)}+4 \mathfrak{g}^{'}-3 \mathfrak{g}\rt)\sin\th\rt\}\nn\\
=& \e^2 \int_{\hat{\O}}\lt(-\t{\mathfrak{g}}'\cos\th+\t{\mathfrak{f}}'\sin\th  \rt)\lt\{ \lt(\t{\mathfrak{f}}^{(3)}-6 \t{\mathfrak{f}}^{(2)}-8\t{ \mathfrak{f}}'\rt)\cos\th+ \lt(\t{\mathfrak{g}}^{(3)}-6 \t{\mathfrak{g}}^{(2)}-8\t{ \mathfrak{g}}'\rt)\sin\th\rt\}\nn\\
=& \e^2\pi \int^\i_0 \t{\mathfrak{f}}'\lt(\t{\mathfrak{g}}^{(3)}-6 \t{\mathfrak{g}}^{(2)}-8\t{ \mathfrak{g}}'\rt)-\lt(\t{\mathfrak{f}}^{(3)}-6 \t{\mathfrak{f}}^{(2)}-8\t{ \mathfrak{f}}'\rt)\mathfrak{g}'.\nn
\end{align}
By integration by parts and Cauchy inequality, we see that
\begin{align}
&\lt|-\e^2 \int_{\hat{\O}} e^{2s}\Dl_s\hat{\Phi}_{1,\th} \lt(\Dl^2_s\hat{\Phi}-4\Dl_s\hat{\Phi}_s+4\Dl_s\hat{\Phi}\rt)\rt|\nn\\
\ls&\e^2 \lt(|\t{\mathfrak{f}}^{''}\t{\mathfrak{g}}^{'}|\big|_{s=0}+|\t{\mathfrak{g}}^{''}\t{\mathfrak{f}}^{'}|\big|_{s=0}\rt)+\dl \int^\i_0 \lt(|\t{\mathfrak{f}}^{'}|^2+|\t{\mathfrak{g}}^{'}|^2\rt)+\e^2\dl\int^\i_0 \lt(|\t{\mathfrak{f}}^{''}|^2+|\t{\mathfrak{g}}^{''}|^2\rt)\nn\\
\ls& \e^2 \sum_{0\neq i+j\leq 4}\int_{\hat{\O}} (\p^i_\th\p^j_s\Phi)^2+\dl\int_{\hat{\O}} | e^{s}\Dl_s\hat{\Phi}_{1,\th}|^2+ \e^2\dl \int_{\hat{\O}} |  (e^{s}\Dl_s\hat{\Phi}_1)_{,s}|^2. \label{vc1posi5}
\end{align}
Here at the last line of the above inequality, we used Sobolev embedding for the first term, and \eqref{vc1posi-4} and \eqref{vc1posi-5} for the second and third term.

Inserting  \eqref{vc1posi2}, \eqref{vc1posi3x1}, \eqref{vc1posi4}, \eqref{vc1posi5} and \eqref{vc1posi1} into \eqref{vc1posi0}, and for small $\dl_0$, when $\dl<\dl_0$, we can achieve that
\begin{align}
&\int^{\i}_0 |  e^{s}\Dl_s\hat{\Phi}_{1,\th}|^2\nn\\
\ls & \e^2\dl \int_{\hat{\O}} |  (e^{s}\Dl_s\hat{\Phi}_{1})_{,s}|^2+\sum^4_{i=0}\hat{\Phi}^{(k)}_{0}(s)+\sum_{\substack{0\neq i+j\leq 4 \\i\neq 0}}\int_{\hat{\O}} e^{s} (\p^i_\th\p^j_s\Phi)^2+\lt|\int_{\hat{\O}} e^{5s}\Dl_s\hat{\Phi}_{1,\th}\hat{R}_{comb}\rt|.\label{vc1posi6}
\end{align}

Now we go to estimate the first term on the righthand of \eqref{vc1posi6}. Multiplying \eqref{vcerrorequ} by $ -e^{3s}\Dl_s\hat{\Phi}_{1}=- e^{2s}\lt(\t{\mathfrak{f}'}\cos\th+\t{\mathfrak{g}'}\sin\th\rt)$ and integrating the resulted equation on $\hat{\O}$ to obtain that
\begin{align}
&-\e^2 \int_{\hat{\O}}e^{2s}\Dl_s\hat{\Phi}_{1} \lt(\Dl^2_s\hat{\Phi}-4\Dl_s\hat{\Phi}_s+4\Dl_s\hat{\Phi}\rt)\nn\\
&+\int_{\hat{\O}} e^{3s}\Dl_s\hat{\Phi}_{1} \lt\{\hat{u}^a\Dl_s \hat{\Phi}_{\th}+ \hat{v}^a \Dl_s \hat{\Phi}_s -2 \hat{v}^a\Dl_s\hat{\Phi}\rt\}\nn\\
&+\int_{\hat{\O}}e^{3s}\Dl_s\hat{\Phi}_{1} \lt\{\hat{u} \Dl_s \hat{\Phi}^a_\th+ \hat{v} \Dl_s \hat{\Phi}^a_s -2 \hat{v}\Dl_s\hat{\Phi}^a\rt\}\nn\\
&=-\int_{\hat{\O}} e^{5s}\Dl_s\hat{\Phi}_{1} \hat{R}_{\text{comb}}.\label{vc1xener0}
\end{align}
{\bf\noindent Estimates of $\boldsymbol{-\e^2 \int_{\hat{\O}}e^{2s}\Dl_s\hat{\Phi}_{1}\lt(\Dl^2_s\hat{\Phi}-4\Dl_s\hat{\Phi}_s+4\Dl_s\hat{\Phi}\rt)}$}:

Using \eqref{vc1posi-5} and \eqref{vcposi8}, we obtain that
\begin{align}
&-\e^2 \int_{\hat{\O}}e^{2s}\Dl_s\hat{\Phi}_{1}\lt(\Dl^2_s\hat{\Phi}-4\Dl_s\hat{\Phi}_s+4\Dl_s\hat{\Phi}\rt)\nn\\
=&-\e^2 \int_{\hat{\O}}e^{2s}\Dl_s\hat{\Phi}_{1}\lt(\Dl^2_s\hat{\Phi}_{1}-4\Dl_s\hat{\Phi}_{1,s}+4\Dl_s\hat{\Phi}_{1}\rt)\nn\\
=& -\e^2 \int_{\hat{\O}}\lt(\t{\mathfrak{f}}'\cos\th+\t{\mathfrak{g}}'\sin\th \rt)\lt\{ \lt(\t{\mathfrak{f}}^{(3)}-6 \t{\mathfrak{f}}^{(2)}-8\t{ \mathfrak{f}}'\rt)\cos\th+ \lt(\t{\mathfrak{g}}^{(3)}-6 \t{\mathfrak{g}}^{(2)}-8\t{ \mathfrak{g}}'\rt)\sin\th\rt\}\nn\\
=& -\e^2\pi \int^\i_0 \t{\mathfrak{f}}'\lt(\t{\mathfrak{f}}^{(3)}-6 \t{\mathfrak{f}}^{(2)}-8\t{ \mathfrak{f}}'\rt)+\lt(\t{\mathfrak{g}}^{(3)}-6 \t{\mathfrak{g}}^{(2)}-8\t{ \mathfrak{g}}'\rt)\t{\mathfrak{g}}'.
\end{align}
By integration by parts and Cauchy inequality, we see that
\begin{align}
&\lt|\e^2 \int_{\hat{\O}}e^{2s}\Dl_s\hat{\Phi}_{1} \lt(\Dl^2_s\hat{\Phi}-4\Dl_s\hat{\Phi}_s+4\Dl_s\hat{\Phi}\rt)\rt|\nn\\
\gs &\e^2\int^\i_0 \lt(|\t{\mathfrak{f}}^{''}|^2+|\t{\mathfrak{g}}^{''}|^2\rt) -C\e^2 \lt(|\t{\mathfrak{f}}^{''}\t{\mathfrak{f}}^{'}|+|\t{\mathfrak{g}}^{''}\t{\mathfrak{g}}^{'}|\rt)\big|_{s=0}-C\e^2 \int^\i_0\lt(|\t{\mathfrak{f}}^{'}|^2+|\t{\mathfrak{g}}^{'}|^2\rt)\nn\\
\gs & \e^2 \int_{\hat{\O}} | (e^{s}\Dl_s\hat{\Phi}_{1})_{,s}|^2  -C\e^2 \sum_{0\neq i+j\leq 4}\int_{\hat{\O}} (\p^i_\th\p^j_s\Phi)^2-C\e^2 \int^{\i}_0 | e^{s}\Dl_s\hat{\Phi}_{1,\th}|^2. \label{vc1fxener1}
\end{align}

{\bf\noindent Estimates of $\boldsymbol{\int_{\hat{\O}} e^{3s}\Dl_s\hat{\Phi}_{1}\lt[\hat{u}^a\Dl_s \hat{\Phi}_{\th}+ \hat{v}^a \Dl_s \hat{\Phi}_s -2 \hat{v}^a\Dl_s\hat{\Phi}\rt]}$}:
 In the following estimates, $\nu>0$ denotes a sufficiently small numerical constant chosen to be absorbed by the corresponding left-hand-side terms.

By Cauchy inequality, the third estimate in \eqref{vcappdetail1} and Poincar\'{e} inequality, we have

\begin{align}
&\lt|-\int_{\hat{\O}}e^{3s}\Dl_s\hat{\Phi}_{1}\lt\{\hat{u}^a\Dl_s \hat{\Phi}_{\th}+ \hat{v}^a \Dl_s \hat{\Phi}_s -2 \hat{v}^a\Dl_s\hat{\Phi}\rt\}\rt|\nn\\
=&\lt|-\int_{\hat{\O}}e^{3s}\Dl_s\hat{\Phi}_{1} \lt\{(\hat{u}^a-(\t{\o}+\mathcal{A}_0)e^{-s})\Dl_s \hat{\Phi}_{\th}+ \hat{v}^a \Dl_s \hat{\Phi}_s -2 \hat{v}^a\Dl_s\hat{\Phi}\rt\}\rt|\nn\\
 &+(\t{\o}+\mathcal{A}_0)\underbrace{\lt|\int_{\hat{\O}}e^{2s}\Dl_s\hat{\Phi}_{1} \Dl_s \hat{\Phi}_{1,\th}\rt|}_{=0}\nn\\
\ls &\lt|\int_{\hat{\O}} |(e^{s}\Dl_s\hat{\Phi}_{1})| \lt(|\Dl_s \hat{\Phi}_{\th}|+ |\Dl_s \hat{\Phi}_s|+|\Dl_s\hat{\Phi}|\rt\}\rt|\nn\\
\ls&  \nu \int_{\hat{\O}} |e^{s}\Dl_s\hat{\Phi}_{1,\th}|^2+  \int_{\hat{\O}} \lt(|\Dl_s \hat{\Phi}_{\th}|^2+|\Dl_s \hat{\Phi}_{s}|^2+|\Dl_s \hat{\Phi}|^2\rt) .\label{vc1fxener2}
\end{align}

{\bf\noindent Estimates of $\boldsymbol{\int_{\hat{\O}}e^{3s}\Dl_s\hat{\Phi}_{1} \lt[\hat{u} \Dl_s \hat{\Phi}^a_\th+ \hat{v} \Dl_s \hat{\Phi}^a_s -2 \hat{v}\Dl_s\hat{\Phi}^a\rt]}$}:
We rewrite this term, use estimates in \eqref{vcpositive3x2} and \eqref{vcpositive3x9} and Cauchy inequality to obtain that
\begin{align}
&\lt|-\int_{\hat{\O}}e^{3s}\Dl_s\hat{\Phi}_{1} \lt\{\hat{u} \Dl_s \hat{\Phi}^a_\th+ \hat{v} \Dl_s \hat{\Phi}^a_s -2 \hat{v}\Dl_s\hat{\Phi}^a\rt\}\rt|\nn\\
&= \lt|\int_{\hat{\O}}e^{2s}\Dl_s\hat{\Phi}_{1} \lt\{ \hat{\Phi}_s (\Dl_s \hat{\Phi}^a_\th-\Dl_s \hat{\Phi}^a_{e,\th})- \hat{\Phi}_\th \Dl_s \hat{\Phi}^a_s+2\hat{\Phi}_\th \Dl_s\hat{\Phi}^a\rt\}\rt|\nn\\
\ls& \dl\int_{\hat{\O}} e^{s}\Dl_s\hat{\Phi}_{1}  \lt( e^{-s}(1+s^{-1}) |\hat{\Phi}_s|+ (1+s^{-2}) |\hat{\Phi}_{\th}|\rt)\nn\\
\ls&  \nu \int_{\hat{\O}} |e^{s}\Dl_s\hat{\Phi}_{1,\th}|^2 + \int_{\hat{\O}} \lt(|\hat{\Phi}_s|^2+|\hat{\Phi}_{ss}|^2+e^{ s} \hat{\Phi}^2_{\th ss}\rt). \label{vc1fxener3}
\end{align}

Inserting \eqref{vc1fxener1}, \eqref{vc1fxener2}, and \eqref{vc1fxener3} into \eqref{vc1xener0}, for sufficiently small $\dl_0$ (i.e., when $\dl<\dl_0$), we obtain
\begin{align}
& \e^2 \int_{\hat{\O}} | (e^{s}\Dl_s\hat{\Phi}_{1})_{,s}|^2\nn\\
\ls &\sum^4_{i=0}\|\hat{\Phi}^{(k)}_{0}\|^2_{L^2(\hat{\O})}+\sum_{\substack{0\neq i+j\leq 4 \\i\neq 0}}\int_{\hat{\O}} e^{s} (\p^i_\th\p^j_s\Phi)^2+C\nu \int^{\i}_0 | e^{s}\Dl_s\hat{\Phi}_{1,\th}|^2+\lt|\int_{\hat{\O}}e^{5s}\Dl_s\hat{\Phi}_{1}\hat{R}_{comb}\rt|.\label{vc1fxener4}
\end{align}
Combining \eqref{vc1posi6} and \eqref{vc1fxener4} and using Cauchy inequality, we obtain that
\begin{align}
&\e^2 \int_{\hat{\O}} | (e^{s}\Dl_s\hat{\Phi}_{1})_{,s}|^2+ \int^{\i}_0 |(e^{s}\Dl_s\hat{\Phi}_{1})_{,\th}|^2\nn\\
\ls & \sum^4_{i=0}\|\hat{\Phi}^{(k)}_{0}\|^2_{L^2(\hat{\O})}+\sum_{\substack{0\neq i+j\leq 4 \\i\neq 0}}\int_{\hat{\O}} e^{s} (\p^i_\th\p^j_s\Phi)^2+\int_{\hat{\O}}e^{8 s}\hat{R}^2_{comb},\nn
\end{align}
which is \eqref{vc1festiposi} after multiplied by $\e^{18}$.

\section{Construction of approximate solutions}\label{secappro}
\indent
In this section, we construct an approximate solution of the Navier-Stokes equations (\ref{nspolar}) by the procedure of matched asymptotic expansion which can be proceeded as follows.

\bes
(u^{(0)}_e,v^{(0)}_e)\rightarrow (u^{(0)}_p,v^{(1)}_p)\rightarrow (u^{(1)}_e,v^{(1)}_e)\rightarrow (u^{(1)}_p,v^{(2)}_p)\rightarrow (u^{(2)}_e,v^{(2)}_e)\cdots.
\ees
\begin{itemize}
\item We use $(u_e,v_e)$ to denote the Euler asymptotic expansion away from the boundary, while $(u_p,v_p)$ to denote the asymptotic expansion near the boundary $\{r=1\}$.
\item The superscript on the Euler and the boundary layer asymptotic expansion represents the matching order of $\e$.
\item The equations satisfied by $(u_e,v_e)$ and $(u_p,v_p)$ will be obtained by solving \eqref{ns} by matching the order of $\e$. The asymptotic expansion will be solved step by step starting from $(u^{(0)}_e,v^{(0)}_e):=(\f{\t{\o}}{r}, 0)$.
\end{itemize}

Define the boundary layer variable $\zeta=\f{r-1}{\e}$. Now we are ready to write the formal asymptotic expansion away from the boundary and near the boundary.

{\noindent\bf Euler expansions away from the boundary}

Away from the boundary, we make the following formal expansions
\be\label{eulerextension}
\begin{aligned}
&u^{\e}(\th,r)=u^e(r)+\sum^{k_0}_{k=1}\e^{k}u^{(k)}_e(\th,r)+\text{h.o.t.},\\[5pt]
&v^{\e}(\th,r)=\sum^{k_0}_{k=1}\e^{k}v^{(k)}_e(\th,r)+\text{h.o.t.},\ p^{\e}(\th,r)= p_e^{(0)}(r)+\sum^{k_0}_{k=1}\e^{k}p^{(k)}_e(\th,r)+\text{h.o.t.}.
\end{aligned}
\ee
 Here \(p_e^{(0)}(r)=-\frac{\t{\o}^{\,2}}{2r^2}\), up to an additive constant. Here and in what follows, ``h.o.t." means higher order terms.

{\noindent\bf Boundary layer expansions near the boundary $\{r=1\}$}

Near the boundary, we make the following formal expansions
\be\label{boundaryextensionu}
\begin{aligned}
u^{\e}(\th,r)=&u^e(r)+u_p^{(0)}(\th,\zeta)+\sum^{k_0}_{k=1}\e^{k}\lt(u^{(k)}_e(\th,r)+u_p^{(k)}(\th,\zeta)\rt)+\text{h.o.t.},\\[5pt]
v^{\e}(\th,r)=&\sum^{k_0}_{k=1}\e^{k}\lt(v^{(k)}_e(\th,r)+v^{(k)}_p(\th,\zeta)\rt)+\text{h.o.t.},\q p^{\e}(\th,r)=\sum^{k_0}_{k=1}\e^{k}\lt(p^{(k)}_e(\th,r)+p^{(k)}_p(\th,\zeta)\rt)+\text{h.o.t.}.
\end{aligned}
\ee
The boundary condition is matched by
\begin{align}
&u^e(r)\big|_{r=1}+u_p^{(0)}\big|_{\zeta=0}={\o}+\dl f(\th),\q u_e^{(k)}\big|_{r=1}+u_p^{(k)}\big|_{\zeta=0}=0,\q 1\leq k\in\bN, \nn\\
&v_e^{(k)}\big|_{r=1}+v_p^{(k)}\big|_{\zeta=0}=0,\q 1\leq k\in\bN.\nn
\end{align}

Next we deduce the equations satisfied by these expansions.
\subsection{Equations for lower order expansions}\label{lowerexpansion}
\subsubsection{Equations for $(u_p^{(0)},v_p^{(1)},p_p^{(1)})$}
\indent

By substituting the outer boundary layer expansion \eqref{boundaryextensionu} into (\ref{nspolar}) and collecting the $\epsilon-0$th order terms, we obtain
 the following steady boundary layer equations for $(u_p^{(0)},v_p^{(1)},p_p^{(1)})$
\be\label{boundarylayeru0}
\left\{
\begin {array}{ll}
\big(\t{\o}+u_p^{(0)}\big)\partial_\th u_p^{(0)}+\big( v_p^{(1)}- v_p^{(1)}(\th,0)\big)\partial_\zeta u_p^{(0)}-\partial^2_{\zeta}u_p^{(0)}=0,\\[5pt]
\p_\zeta p^{(1)}_p=0,\\[5pt]
\partial_\th u_p^{(0)}+\partial_\zeta v_p^{(1)}=0,\\[5pt]
(u_p^{(0)},v_p^{(1)})(\th,\zeta)=(u_p^{(0)},v_p^{(1)})(\th+2\pi,\zeta),\\[5pt]
u_p^{(0)}\big|_{\zeta=0}=\o+\dl f(\th)-\t{\o},\q \lim\limits_{\zeta\rightarrow +\infty}(u_p^{(0)},v_p^{(1)},p_p^{(1)})=0.
\end{array}
\right.
\ee
From the \eqref{boundarylayeru0}$_{2,5}$ for the requirement of $p^{(1)}_p$, we have $p^{(1)}_p\equiv 0$.
\subsubsection{Equations for $(u_e^{(1)},v_e^{(1)},p_e^{(1)})$}
\indent

Inserting the Euler expansion \eqref{eulerextension} into \eqref{nspolar} and collecting the $\epsilon-1$th order terms, we deduce that $(u_e^{(1)},v_e^{(1)},p_e^{(1)})$ satisfies the following linearized Euler equations
\be\label{middleeuler1}
\left \{
\begin{array}{ll}
u^e(r) \partial_\th u_e^{(1)}+rv_e^{(1)}\p_r u^e(r)+u^e(r)v^{(1)}_e+\partial_\th p_e^{(1)}=0,\\[5pt]
u^e(r) \partial_\th v_e^{(1)}-2u^e(r) u^{(1)}_e+r\partial_r p_e^{(1)}=0,\\[5pt]
\partial_\th u_e^{(1)}+r\partial_r v_e^{(1)}+v_e^{(1)}=0,\\
(u_e^{(1)},v_e^{(1)})(\th,r)=(u_e^{(1)},v_e^{(1)})(\th+2\pi,r),\\
 v_e^{(1)}|_{r=1}=-v_p^{(1)}|_{\zeta=0},\ v_e^{(1)}|_{r=+\i}=0,
\end{array}
\right.
\ee
where $v_p^{(1)}$ is obtained from \eqref{boundarylayeru0}.

\subsubsection{Equations for $(u_p^{(1)},v_p^{(2)},p_p^{(2)})$}

\indent

By substituting the boundary layer expansion \eqref{boundaryextensionu} into (\ref{nspolar}) and collecting the $\epsilon-1$th order terms, we obtain
 the following steady boundary layer equations for $(u_p^{(1)},v_p^{(2)},p_p^{(2)})$
 \be\label{boundarylayeru1}
 \lt\{
 \begin{aligned}
 &(\t{\o}+u^{(0)}_p)\p_\th u^{(1)}_p+(v^{(2)}_p-v^{(2)}_p(\th,0))\p_\zeta u^{(0)}_p-\p^2_\zeta u^{(1)}_p\\
 &+u^{(1)}_p\p_\th u^{(0)}_p+(v^{(1)}_e(\th,1)+v^{(1)}_p)\p_\zeta u^{(1)}_p=f_{p,1}(\th,\zeta),\\
 &-\p_\zeta p^{(2)}_p=g_{p,1}(\th,\zeta),\\
 &\p_\th u^{(1)}_p+\p_\zeta v^{(2)}_p+\p_\zeta (\zeta v^{(1)}_p)=0,
 \end{aligned}
 \rt.
 \ee
 where the force terms ${f}_{p,1}$ and ${g}_{p,1}$ are defined as
\begin{align}
{f}_{p,1}(\theta, \zeta):= & \zeta \partial^2_{ \zeta} {u}_p^{(0)}+\partial_\zeta {u}_p^{(0)} \nn\\
& -{u}_p^{(0)}\left(\partial_\theta u_e^{(1)}\left(\theta, 1\right)+v_e^{(1)}\left(\theta, 1\right)+{v}_p^{(1)}\right)-\left(\zeta(u^e)^{\prime}\left(1\right) +u_e^{(1)}\left(\theta, 1\right)\right) \partial_\theta {u}_p^{(0)} \nn\\
& -\left(\partial_r v_e^{(1)}\left(\theta, 1\right)+v_e^{(1)}\left(\theta, 1\right)+{v}_p^{(1)}\right) \zeta \partial_\zeta u_p^{(0)}-\left[u^e\left(1\right)+ (u^e)^{\prime}\left(1\right)\right] {v}_p^{(1)},\nn\\
g_{p,1}(\th,\zeta):=&u^{(0)}_p \p_\th v^{(1)}_e(x,1)+(u^e(1)+u^{(0)}_p)\p_\th v^{(1)}_p\nn\\
                    &+(v^{(1)}_e(\th,1)+v^{(1)}_p)\p_\zeta v^{(1)}_p-\p^2_\zeta v^{(1)}_p\nn\\
                    &-2\lt(u^e(1)+u^{(0)}_p\rt)u^{(1)}_p-2u^{(0)}_p\lt[u^e(1)+(u^e)'(1)\rt].\nn
\end{align}

\subsubsection{Equations for $(u_e^{(2)},v_e^{(2)},p_e^{(2)})$}
\indent

Inserting the Euler expansion \eqref{eulerextension} into \eqref{ns} and collecting the $\epsilon-2$th order terms, we deduce that $(u_e^{(2)},v_e^{(2)},$ $ p_e^{(2)})$ satisfies the following linearized Euler equations

\be\label{middleeuler2}
\left \{
\begin{array}{ll}
u^e(r) \partial_\th u_e^{(2)}+rv_e^{(2)}\p_r u^e(r)+u^e(r)v^{(2)}_e+\partial_\th p_e^{(2)}=0,\\[5pt]
u^e(r) \partial_\th v_e^{(2)}-2u^e(r) u^{(2)}_e+r\partial_r p_e^{(2)}=0,\\[5pt]
\partial_\th u_e^{(2)}+r\partial_r v_e^{(2)}+v_e^{(2)}=0,\\
(u_e^{(2)},v_e^{(2)})(\th,r)=(u_e^{(2)},v_e^{(2)})(\th+2\pi,r),\\
 v_e^{(2)}|_{r=1}=-\hat{v}_p^{(2)}|_{\zeta=0},\ v_e^{(2)}|_{r=+\i}=0,
\end{array}
\right.
\ee

where ${v}_p^{(2)}$ is obtained from \eqref{boundarylayeru1}.

\subsection{Solvability of the lower order asymptotic expansions}
\indent

Now we give the solvability of the equations satisfied by the lower order asymptotic expansions in Section \ref{lowerexpansion}. The order in which we solve these equations are as follows
{\small
\begin{align*}
(u^e(r),0)\rightarrow (u_p^{(0)},v_p^{(1)})\rightarrow (u_e^{(1)},v_e^{(1)})\rightarrow(u_p^{(1)},v_p^{(2)})\rightarrow (u_e^{(2)},v_e^{(2)})\cdots.
\end{align*}
}
\subsubsection{The boundary layer system \eqref{boundarylayeru0} and its solvability}
\indent
A necessary condition for the solvability of the system \eqref{boundarylayeru0} is stated in the following Lemma. One can also refer to \cite[Lemma 2.1]{FeiGLT:2023CMP} or \cite[Lemma 2.1]{FeiGLT:2024ADV}.
\begin{lemma}\label{lembw}
If the system \eqref{boundarylayeru0} has a solution $(\bar{u}_p^{(0)}, v_p^{(1)})$, which satisfies
\beas
&&\t{\o}+{u}_p^{(0)}(\th,\zeta)>0, \q \|{u}_p^{(0)}\|_\infty<+\i,\ \forall\ \zeta\geq 0,
\eeas
then there holds
\begin{align}
\t{\o}^2=\o^2+\frac{\o\dl}{\pi}\int_0^{2\pi}f(\th)d\th+\frac{\dl^2}{2\pi}\int_0^{2\pi}f^2(\th)d\th.\nn
\end{align}
\end{lemma}
Next, we solve the steady boundary-layer system \eqref{boundarylayeru0}; one may refer to \cite[Corollary 2.4]{FeiGLT:2023CMP} for its solvability. We state the result below.
\begin{proposition}\label{propdcu0}
 There exists $\dl_0>0$ such that for any $\dl\in (0,\dl_0)$ and any $j,k,\ell\in \mathbb{N}\cup \{0\}$, the system (\ref{boundarylayeru0}) has a unique solution $(u^{(0)}_p,v^{(1)}_p)$ which satisfies, $\forall\ \zeta\geq 0$,
\bes
\begin{aligned}
&\lt\|\big<\zeta\big>^\ell\partial_\th^j\partial_\zeta^k( {u}_p^{(0)}, {v}_p^{(1)}) \rt\|_{L^\i}\leq C_{j,k,\ell}\dl,\q \int_0^{2\pi}{v}_p^{(1)}(\th,\zeta)d \th=0.
\end{aligned}
\ees
\end{proposition}

\subsubsection{The linearized Euler system for $(u_e^{(1)}, v_e^{(1)}, p_e^{(1)})$ and its solvability}

\begin{proposition}\label{propeulerorder1}
There exists $\dl_0>0$ such that for any $\dl\in(0,\dl_0)$, the linearized Euler equations (\ref{middleeuler1}) have a solution $(u_e^{(1)}, v_e^{(1)}, p_e^{(1)})$ which satisfies
\begin{align}\label{euler1esti}
&\lt\|r^{k+2}\partial^j_\th\partial^{k}_r \lt( u_e^{(1)}-\f{\t{A}_1}{r}, v_e^{(1)}\rt)\rt\|_{L^\i}\leq \dl C_{j,k}, \ \forall\ j,k\geq 0,
\end{align}
for some constant $\t{A}_1$ satisfying $|\t{A}_1|\ls\dl$.
\end{proposition}
\begin{proof}
Subtracting $\p_\th(\ref{middleeuler1})_2$ from $r\p_r(\ref{middleeuler1})_1$ to eliminate the pressure $p_e^{(1)}$, we obtain that
\bes
u^e(r)\Dl(r v^{(1)}_e)=0.
\ees
where $\Dl=\p^2_r+\f{1}{r^2}\p^2_{\th}+\f{1}{r}\p_r$. Since $u^e(r)$ is non-degenerate outside the disc, we obtain that
\bes
\Dl(r v^{(1)}_e)=0.
\ees
Then we can obtain $v_e^{(1)}$ by solving the following linear boundary value problem
\bes
\lt\{
\bali
&\Dl (rv^{(1)}_e)=0,\\
&rv_e^{(1)}\big|_{r=1}=-{v}_p^{(1)}\big|_{\zeta=0}, rv_e^{(1)}\big|_{r=+\i}=0,\\
&rv_e^{(1)}(\th,r)=rv_e^{(1)}(\th+2\pi,r).
\eali
\rt.
\ees

The above equations can be solved by using the method of variables separation.  The solution can be explicitly given by
\bes
rv_e^{(1)}(\th,r)=\sum_{n\in\bZ}\f{-\f{1}{2\pi}\int^{2\pi}_0{v}_p^{(1)}\big|_{\zeta=0}e^{-in\th}d\th}{r^{|n|}}e^{in\th}.
\ees
By noting that when $n=0$,
\bes
\int^{2\pi}_0{v}_p^{(1)}\big|_{\zeta=0}d\th=0,
\ees
then we obtain that
\be\label{eulerfirst3}
rv_e^{(1)}(\th,r)=\sum_{n\in{\bZ/\{0\}}}\f{-\f{1}{2\pi}\int^{2\pi}_0{v}_p^{(1)}\big|_{\zeta=0}e^{-in\th}d\th}{r^{|n|}}e^{in\th}.
\ee

It is easy to see that
\bes
 \|r^{k+2}\p^j_\th\p^k_r v_e^{(1)}\|_{L^\i}\leq C_{j,k}\dl \q \text{for } 0\leq k\in \bN.
\ees

After $v^{(1)}_e$ is given, we define
\bes
u^{(1)}_e(\th,r)=-\int^\th_0 \p_r(rv^{(1)}_e(\bar{\th},r))d\bar{\th},
\ees
which satisfies
\begin{eqnarray}
\left \{
\begin {array}{ll}
\partial_\th u_e^{(1)}+\partial_r(rv_e^{(1)})=0,\\[5pt]
u_e^{(1)}(\th,r)=u_e^{(1)}(\th+2\pi,r).\nonumber
\end{array}
\right.
\end{eqnarray}
Using \eqref{eulerfirst3}, we see that for $k\in\bN$
\bes
\|r^{k+2}\p^j_\th\p^k_r u_e^{(1)}\|_{L^\i}\leq \dl C_{j,k}.
\ees
After obtaining $(u_e^{(1)}, v_e^{(1)})$, we construct $p_e^{(1)}$ as following
\begin{align*}
p_e^{(1)}(\th,r):=\phi(r)- u^e(r)u_e^{(1)}(\th,r)-\p_r(ru^e(r)) \int^\th_0v_e^{(1)}(\bar{\th},r)d\bar{\th},
\end{align*}
which satisfies
\bes
\p_\th p_e^{(1)}+u^e(r)\p_{\th} u_e^{(1)}+\p_r(ru^e(r))v_e^{(1)}=0.
\ees
Let $\phi(r)$ be a function satisfying
\begin{align*}
r\phi'(r)+u^e(r)\p_\th v^{(1)}_e(0,r)-2u^e(r)u_e^{(1)}(0,r)=0.
\end{align*}
Combining the equations of $(u_e^{(1)}, v_e^{(1)})$, it's direct to obtain
\begin{align*}
u^e(r) \partial_\th v_e^{(1)}-2u^e(r) u^{(1)}_e+r\partial_rp_e^{(1)}=0.
\end{align*}
Hence, $(u_e^{(1)}, v_e^{(1)},p_e^{(1)})$ solves the equation (\ref{middleeuler1}) and satisfies (\ref{euler1esti}).

\end{proof}

\subsubsection{The linearized system for  $(u_p^{(1)},v_p^{(2)})$ and its solvability}
\indent

In this subsection, we consider the solvability of \eqref{boundarylayeru1}. We have the following Proposition.

\begin{proposition}\label{propdcu1}
There exists $\dl_0>0$ such that for any $\dl\in(0,\dl_0)$, the equations (\ref{boundarylayeru1}) have a unique solution $(u_p^{(1)},v_p^{(2)})$ which satisfies
\begin{align*}
\begin{aligned}
&\lt\|\partial_\th^j\partial_\zeta^k \big({u}_p^{(1)}-A_{1}\big)\big<\zeta\big>^\ell\rt\|_{L^\i}\leq C_{j,k,\ell}\dl,\\
&\lt\|\partial_\th^j\partial_\zeta^k v_p^{(2)}\big<\zeta\big>^\ell\rt\|_{L^\i}\leq C_{j,k,\ell}\dl,\q \int_0^{2\pi}v_p^{(2)}(\th,\zeta)dx=0, \ \forall\ \zeta\geq 0,
\end{aligned}
\end{align*}
where
$A_{1}:=\lim\limits_{\zeta\rightarrow +\infty} u_p^{(1)}(\th,\zeta)$ is a constant which satisfies $|A_{1}|\leq C\dl.$
\end{proposition}

Let $\kappa\in C_c^\infty ([0,+\i))$ satisfy
\begin{align*}
\kappa(0)=1,\ \int_0^{+\infty}\kappa(\zeta)d\zeta=0.
\end{align*}
For simplicity, we set
\begin{align*}
\bar{u}:&=\t{\o}+u_p^{(0)}, \ \bar{v}:=v_p^{(1)}+v^{(1)}_p(\th,1),\\
u:&=u_p^{(1)}+u_e^{(1)}(\th,1)\kappa(\zeta),\\
v:&=v_p^{(2)}-v_p^{(2)}(\th,0)+\zeta v_p^{(1)} -\p_\th u_e^{(1)}(\th,1)\int_0^\zeta \kappa(\bar{\zeta})d\bar{\zeta}.
\end{align*}
Then, the equations (\ref{boundarylayeru1}) reduce to
\be\label{boundarylayeru1r}
\left \{
\begin {array}{ll}
\bar{u}\partial_\th u+\bar{v}\partial_\zeta u+u\partial_\th \bar{u}+v\partial_\zeta\bar{u}-\partial^2_{\zeta}u=\bar{f},\\[7pt]
\partial_\th u+\partial_\zeta v=0,\\[5pt]
u(\th,\zeta)=u(\th+2\pi,\zeta),\ v(\th,\zeta)=v(\th+2\pi,\zeta)\\[5pt]
u|_{\zeta=0}=v|_{\zeta=0}=0,\  \lim\limits_{\zeta\rightarrow +\infty}\partial_\zeta u=0,
\end{array}
\right.
\ee
where $\bar{f}(\th,\zeta)$ is $2\pi$-periodic function and decays fast as $\zeta\rightarrow +\infty$ given as follows
\begin{align*}
\bar{f}=& f_{p,1}+\bar{u} \p_\th u_e^{(1)}(\th,1)\kappa(\zeta)\\
        &+\lt(\zeta v_p^{(1)} -\p_\th u_e^{(1)}(\th,1)\int_0^\zeta \kappa(\bar{\zeta})d\bar{\zeta}\rt)\p_\zeta \bar{u}\\
        &+ u_e^{(1)}(\th,1)\lt(\p_\th \bar{u}\kappa(\zeta)+\bar{v}\kappa'-\kappa^{''}(\zeta)\rt),
\end{align*}
 which satisfies
\bes
\lt\|\partial_\th^j\partial_\zeta^k \bar{f}\big<\zeta\big>^\ell\rt\|_{L^\i}\leq C_{j,k,\ell}\dl,\ \forall\ \zeta\geq 0.
\ees

This system \eqref{boundarylayeru1r} is exactly (2.47) in \cite{FeiGLT:2023CMP}. Reader can refer there for further detailed proof of Proposition \ref{propdcu1}. Here we omit the details.

Next, we construct the pressure $p_p^{(2)}(x,\zeta)$. Thanks to \eqref{boundarylayeru1}$_2$, we take
\bes
p_p^{(2)}(\th,\zeta)=\int^{+\i}_{\zeta} g_{p,1}(\th,\bar{\zeta})d\bar{\zeta}.
\ees
which satisfies
\begin{align*}
-\p_\zeta p_p^{(2)}(\th,\zeta)=g_1(\th,\zeta), \quad \lim_{\zeta\rightarrow +\infty}p_p^{(2)}(\th,\zeta)=0,
\end{align*}

\subsubsection{The linearized Euler system for $(u_e^{(2)},\ v_e^{(2)},\ p_e^{(2)})$ and its solvability}
\indent

Since the system \eqref{middleeuler2} is completely the same as the system \eqref{middleeuler1}, we omit the details to prove the solvability.

\subsubsection{Corrections of the expansions to the boundary layer equations and the Euler equations}
\indent

Since we want the solutions of the boundary layer equations decaying fast at infinity, we define
\bes
\t{u}^{(1)}_p:=u^{(1)}_p-A_1.
\ees
and use $\t{u}^{(1)}_p$ to replace the original ${u}^{(1)}_p$. Then we need to correct $u^{(1)}_e$ to $\t{u}^{(1)}_e$ such that they satisfy
\begin{align}
&\t{u}^{(1)}_e\big|_{r=1}=-u^{(1)}_p\big|_{\zeta=0}+A_1={u}^{(1)}_e\big|_{r=1}+A_1,\q \t{u}^{(1)}_e\big|_{r=+\i}=0,\nn
\end{align}
and $(\t{u}^{(1)}_e,v^{(1)}_e)$ is still a solution of the system \eqref{middleeuler1}.

Also, using divergence-free condition and a direct calculation, we see that
\bes
\p_\th \lt( r(\Dl-\f{1}{r^2})u_e^{(1)} +\f{2}{r}\p_\th v_e^{(1)}\rt)=0.
\ees
Let
\bes
\t{u}^{(1)}_e={u}^{(1)}_e+h_1(r),
\ees
where $\chi(r)\in C^\i_{c}([1,+\i))$ is a cut-off function to be determined later,  $h_1(r)$ satisfies the following ODE
\be \label{correction1}
\lt\{
\bali
&r(\Dl-\f{1}{r^2})h_1(r)=-r(\Dl-\f{1}{r^2})u_e^{(1)} +\f{2}{r}\p_\th v_e^{(1)}:=rf(r),\\
& h_1(1)=A_1,\q h_1(+\i)=0.
\eali
\rt.
\ee
By estimates of $u_e^{(1)}$ and $v_e^{(1)}$, we have
\bes
\|r^{k+4}\p^{k}_r f(r)\|_{L^\i}\leq C_k.
\ees
The solution of \eqref{correction1} is given by
\begin{align}
h_1(r)=&\f{A_1+\f{1}{2}\int^{+\i}_1 f(s)(1-s^2)ds}{r}+\f{1}{2r}\int^{+\i}_r f(s)s^2ds-\f{1}{2}r\int^{+\i}_r f(s)ds\nn\\
       :=& \f{\t{A}_1}{r}+\f{1}{2r}\int^{+\i}_r f(s)s^2ds-\f{1}{2}r\int^{+\i}_r f(s)ds.\nn
\end{align}
 It is not hard to deduce that
\bes
\lt\|r^{k+2}\p^{k}_r \lt(h_1(r)-\f{\t{A}_1}{r}\rt)\rt\|_{L^\i}\leq C_{k} \dl.
\ees
Obviously, $(\t{u}^{(1)}_e, v^{(1)}_e)$ is still a solution of the system \eqref{middleeuler1} and satisfies \eqref{euler1esti}. Also we have that
  \bes
 r(\Dl-\f{1}{r^2})\t{u}_e^{(1)} +\f{2}{r}\p_\th v_e^{(1)}\equiv 0.
  \ees
  By letting $\t{u}_p^{1}$ and $\t{u}_e^{(1)}$ to replace  ${u}_p^{1}$ and $u_e^{(1)}$, the equations for the first order boundary layer is invariant.

\subsection{Equations for higher order expansions and its solvability}
\indent

After performing the above lower order extensions, we can perform the higher order expansions inductively as follows.

\subsubsection{Equations for $(u_e^{(k)},v_e^{(k)},p_e^{(k)})$}\label{sec431}
\indent

For $k\geq 2$, by collecting the $\epsilon-k$th order terms of the Euler expansions in \eqref{eulerextension}, we deduce that $(u_e^{(k)},v_e^{(k)},p_e^{(k)})$ satisfies the following linearized Euler equations
\be\label{middleeulerk}
\left \{
\begin{array}{ll}
u^e(r) \partial_\th u_e^{(k)}+rv_e^{(k)}\p_r u^e(r)+u^e(r)v_e^{(k)} +\partial_\th p_e^{(k)}=f^{(k)}_e,\\[5pt]
u^e(r) \partial_\th v_e^{(k)}-2u^e(r)u_e^{(k)}+r\partial_r p_e^{(k)}=g^{(k)}_e,\\[5pt]
\partial_\th u_e^{(k)}+\partial_r(r v_e^{(k)})=0,\\
 v_e^{(k)}|_{r=1}=-{v}_p^{(k)}|_{\zeta=0},\ v_e^{(k)}|_{r=+\i}=0,\\
(u_e^{(k)}, v_e^{(k)})(\th,r)=(u_e^{(k)}, v_e^{(k)})(\th+2\pi,r),
\end{array}
\right.
\ee
where the forced term $f^{(k)}_e$ is defined as
\begin{align}
f^{(k)}_e:=&-\sum^{k}_{i=1} u_e^{(i)}\partial_\th u_e^{(k-i)}-r\sum^{k-1}_{i=0} v_e^{(i)}\partial_r u_e^{(k-i)}-\sum^{k-1}_{i=0} v_e^{(i)} u_e^{(k-i)}+r\lt(\Dl-\f{1}{r^2}\rt) u_e^{(k-2)}+\f{2}{r}\p_\th v_e^{(k-2)},\nn
\end{align}
and $g^{(k)}_e$ is defined as
\begin{align}
g^{(k)}_e:=&-\sum^{k-1}_{i=1} u_e^{(i)}\partial_\th v_e^{(k-i)}-r\sum^{k-1}_{i=1} v_e^{(i)}\partial_r v_e^{(k-i)}+\sum^{k-1}_{i=1} u_e^{(i)} u_e^{(k-i)}+\Dl (rv_e^{(k-2)}).\nn
\end{align}

{\noindent\bf Here we give another explanation why we choose the Euler shear flow satisfying}
\bes
u^e(r)=\f{B}{r}.
\ees

When $k=2$, the system \eqref{middleeulerk} is simplified to
\be\label{middleeuler3}
\left \{
\begin{array}{ll}
u^e(r) \partial_\th u_e^{(2)}+v_e^{(2)}\p_r(r u^e(r))+\partial_\th p_e^{(2)}=f^{(2)}_e,\\[5pt]
u_e(r) \partial_\th v_e^{(2)}-2u^e(r)u_e^{(2)}+r\partial_r p_e^{(2)}=g^{(2)}_e,\\[5pt]
\partial_\th u_e^{(2)}+\partial_r(r v_e^{(2)})=0,\\
 v_e^{(2)}|_{r=1}=-{v}_p^{(2)}|_{\zeta=0},\ v_e^{(2)}|_{r=+\i}=0,\\
(u_e^{(2)}, v_e^{(2)})(\th,r)=(u_e^{(2)}, v_e^{(2)})(\th+2\pi,r),
\end{array}
\right.
\ee
where the force terms $f^{(2)}_e$ and $g^{(2)}_e$ are defined by
\begin{align}
f^{(2)}_e:&=- u_e^{(1)}\partial_\th u_e^{(1)}-v_e^{(1)}\partial_r(r u_e^{(1)})+r(\Dl-\f{1}{r^2}) u^e(r),\label{bwx}\\
g^{(2)}_e:&=- u_e^{(1)}\partial_\th v_e^{(1)}-rv_e^{(1)}\partial_r v_e^{(1)}+(u_e^{(1)})^2.\nn
\end{align}

If $(u^e(r), 0)$ is the shear flow solution of the Euler equation, by taking $\th$ average of \eqref{middleeuler3}, we can see that it satisfies
\bes
\int^{2\pi}_0 f^{(2)}_e(\th,r)d\th\equiv 0.
\ees
Then from \eqref{bwx}, we see that

\begin{align}
&2\pi r\lt(\p^2_r+\f{1}{r}\p_r-\f{1}{r^2}\rt)u^e(r)=\int_0^{2\pi}v_e^{(1)}\partial_r(ru_e^{(1)})d\th\nn\\
=&\int_0^{2\pi}\partial_r(ru_e^{(1)})d(\int_0^\th v_e^{(1)}(\th',r)d\th')\nn\\
=&-\int_0^{2\pi}\partial_{ r}(r\p_\th u_e^{(1)})\int_0^\th v_e^{(1)}(\th',r)d\th' d\th\nn\\
=&\int_0^{2\pi}\partial_{r}(r\p_r(r v_e^{(1)}))\int_0^\th v_e^{(1)}(\th',r)d\th' d\th\nn\\
=&\int_0^{2\pi}r\lt(\Dl-\f{\p^2_\th}{r^2}\rt)(rv_e^{(1)})\int_0^\th v_e^{(1)}(\th',r)d\th' d\th\nn\\
=&-\int_0^{2\pi}\p^2_\th v_e^{(1)}\int_0^\th v_e^{(1)}(\th',r)d\th' d\th=\int_0^{2\pi}\p_\th v_e^{(1)} v_e^{(1)} d\th=0,\nn
\end{align}
which implies that $u^e(r)=Ar+\f{B}{r}$. Supplied with the boundary condition $u^e(+\i)=0$, we have $u^e_r=\f{B}{r}$. This also gives the explanation of the choice of the leading Euler equation. \qed

Now we show the solvability of the system \eqref{middleeulerk}.
\begin{proposition}\label{propeulerorderk}
There exists $\dl_0>0$ such that for any $\dl\in(0,\dl_0)$, the linearized Euler equations (\ref{middleeulerk}) has a solution $(u_e^{(k)}, v_e^{(k)}, p_e^{(k)})$ which satisfies
\be\label{eulerkesti}
\begin{aligned}
\lt\|r^{j+2}\partial^i_\th\partial^{j}_r \lt(u_e^{(k)}-\f{\t{A}_k}{r},v_e^{(k)}\rt)\rt\|_{L^\i}\leq C_{i,j}\dl, \ \forall\ i,j\geq 0.
\end{aligned}
\ee
for some constant $\t{A}_k$ satisfying $|\t{A}_k|\ls \dl$.
\end{proposition}
\indent

The idea of proof is the same as that of the system \eqref{middleeuler1}. First we cancel the pressure from \eqref{middleeulerk}$_1$ and \eqref{middleeulerk}$_2$ to obtain that
\be\label{middleeulerk1}
\lt\{
\bali
&-u^e(r) \Dl ( rv_e^{(k)}) =r\p_rf^{(k)}_e-\p_\th g^{(k)}_e,\\
&rv_e^{(k)}|_{r=1}=-{v}_p^{(k)}|_{\zeta=0},\ rv_e^{(k)}|_{r=+\i}=0,\\
&v_e^{(k)}(\th,r)=v_e^{(k)}(\th+2\pi,r).
\eali
\rt.
\ee
Noting in fact that $rv_e^{(k)}$ is harmonic, i.e.
\bes
r\p_rf^{(k)}_e-\p_\th g^{(k)}_e=0.
\ees
We give calculation here. Actually, by the procedure of the Euler expansions, we construct
\bes
r\lt(\Dl-\f{1}{r^2}\rt) u_e^{(k-2)}+\f{2}{r}\p_\th v_e^{(k-2)}=\Dl (rv_e^{(k-2)})=0,
\ees
Then a direct computation indicates that
\begin{align}
&r\p_rf^{(k)}_e-\p_\th g^{(k)}_e\nn\\
& =\underbrace{-\sum^{k}_{i=1} r\p_r u_e^{(i)}\partial_\th u_e^{(k-i)}}_{I_1}\underbrace{-\sum^{k}_{i=1} u_e^{(i)}r\p_r\partial_\th u_e^{(k-i)}}_{I_2}\nn\\
   &\quad\underbrace{-\sum^{k-1}_{i=0} r\p_rv_e^{(i)}\partial_r(r u_e^{(k-i)})}_{I_3}\underbrace{-\sum^{k-1}_{i=0} v_e^{(i)}r\p_r(\p_r(r u_e^{(k-i)}))}_{I_4}\nn\\
   &\quad\underbrace{+\sum^{k-1}_{i=1} \p_\th u_e^{(i)}\partial_\th v_e^{(k-i)}}_{I_5}+\underbrace{\sum^{k-1}_{i=1} u_e^{(i)}\partial^2_\th v_e^{(k-i)}}_{I_6}\nn\\
   &\quad+\underbrace{\sum^{k-1}_{i=1} r\p_\th v_e^{(i)}\partial_r v_e^{(k-i)}}_{I_7}+\underbrace{\sum^{k-1}_{i=1} rv_e^{(i)}\p_\th \partial_r v_e^{(k-i)}}_{I_8}.\nn\\
   &\quad-\underbrace{\sum^{k-1}_{i=1} \p_\th u_e^{(i)} u_e^{(k-i)}}_{I_{9}}-\underbrace{\sum^{k-1}_{i=1} u_e^{(i)}\p_\th  u_e^{(k-i)}}_{I_{10}}.\nn
\end{align}
By using the incompressibility and harmonic property of $r v_e^{(i)}$, we can obtain that
\bes
I_2+I_6+I_{10}=0.
\ees
By using the incompressibility and the property of
\bes
r\lt(\Dl-\f{1}{r^2}\rt) u_e^{(i)}+\f{2}{r}\p_\th v_e^{(i)}=0,
\ees
we can obtain that
\bes
I_1+I_{3}+I_{4}+I_{5}+I_{7}+I_{8}+I_{9}=0.
\ees
This indicates that $  \Dl (rv_e^{(k)})=0$. Then the system \eqref{middleeulerk1} is solved by solving the following linear boundary value problem
\bes
\lt\{
\bali
&- \Dl (rv_e^{(k)}) =0,\\
&rv_e^{(k)}|_{r=1}=-{v}_p^{(k)}|_{\zeta=0},\ rv_e^{(k)}|_{r=+\i}=0,\\
&v_e^{(k)}(\th,r)=v_e^{(k)}(\th+2\pi,r).
\eali
\rt.
\ees
The solution can be explicitly given by
\bes
rv_e^{(k)}(\th,r)=\sum_{n\in\bZ}\f{-\f{1}{2\pi}\int^{2\pi}_0{v}_p^{(k)}\big|_{\zeta=0}e^{-in\th}d\th}{r^{|n|}}e^{in\th}.
\ees
By noting that when $n=0$,
\bes
\int^{2\pi}_0{v}_p^{(k)}\big|_{\zeta=0}d\th=0,
\ees
then we obtain that
\bes
rv_e^{(k)}(\th,r)=\sum_{n\in{\bZ/\{0\}}}\f{-\f{1}{2\pi}\int^{2\pi}_0{v}_p^{(k)}\big|_{\zeta=0}e^{-in\th}d\th}{r^{|n|}}e^{in\th}.
\ees
It is easy to see that
\bes
 \|r^{k+2}\p^j_\th\p^k_r v_e^{(k)}\|_{L^\i}\leq C_{j,k}\dl \q \text{for } 0\leq k\in \bN.
\ees
The zero mean of $v_e^{(k)}$ in $\th$ variable implies, after integrating \eqref{middleeulerk}$_1$ in $\th\in [0,2\pi]$, the compatibility condition: for any $r\in[1,+\i)$,
\be\label{middleeulerk3}
\int^{2\pi}_0 f^{(k)}_e(\th,r)d\th=0.
\ee

Actually by integration by parts, we have
\begin{align}
 &\sum^{k}_{i=0}\int^{2\pi}_0 u_e^{(i)}\partial_\th u_e^{(k-i)}d\th=-\sum^{k}_{i=0}\int^{2\pi}_0 \p_\th u_e^{(i)} u_e^{(k-i)}d\th=-\sum^{k}_{i=0}\int^{2\pi}_0u_e^{(k)}\partial_\th u_e^{(i)}d\th,\label{middleeulerk4}
\end{align}
which implies that
\bes
\sum^{k}_{i=0}\int^{2\pi}_0 u_e^{(i)}\partial_\th u_e^{(k-i)}d\th=0.
\ees
Using the incompressibility and harmonic property of $rv_e^{(i)}$, we have
\begin{align}
 &\sum^{k}_{i=0}\int^{2\pi}_0 v_e^{(i)}\partial_r(r u_e^{(k-i)})d\th\nn\\
& =\sum^{k}_{i=0}\int^{2\pi}_0 v_e^{(i)}\int^\th_0\p_r(r\partial_\th u_e^{(k-i)})(\bar{\th},r)d\bar{\th}d\th\nn\\
& =-\sum^{k}_{i=0}\int^{2\pi}_0 v_e^{(i)}\int^\th_0\p_r(r\partial_r(r v_e^{(k-i)}))(\bar{\th},r)d\bar{\th}d\th\nn\\
 & =\sum^{k}_{i=0}\int^{2\pi}_0 v_e^{(i)}\int^\th_0\partial^2_\th v_e^{(k-i)}(\bar{\th},r)d\bar{\th}d\th \label{middleeulerk5}\\
 & =\sum^{k}_{i=0}\int^{2\pi}_0 v_e^{(i)}\lt(\partial_\th v_e^{(k-i)}(\th,r)-\partial_\th v_e^{(k-i)}(0,r)\rt)d\th\nn \\
 & =0.\nn
\end{align}
In the last but second line of \eqref{middleeulerk5},  integration by parts (the same as \eqref{middleeulerk4}) implies that the first term is zero, while the fact that the second term is zero is due to
\bes
\int^{2\pi}_0 v_e^{(i)}d\th=0, \text{ for } \forall\ r\in[1,+\i).
\ees
Combining \eqref{middleeulerk4} and \eqref{middleeulerk5}, we obtain \eqref{middleeulerk3}.

After $v^{(k)}_e$ is given, we define
\bes
u^{(k)}_e(\th,r)=-\int^\th_0 \p_r(rv^{(k)}_e(\bar{\th},r))d\bar{\th},
\ees
which satisfies
\begin{eqnarray}
\left \{
\begin {array}{ll}
\partial_\th u_e^{(k)}+\partial_r(rv_e^{(k)})=0,\\[5pt]
u_e^{(k)}(\th,r)=u_e^{(k)}(\th+2\pi,r).\nonumber
\end{array}
\right.
\end{eqnarray}
Using \eqref{eulerfirst3}, we see that for $i,j\in\bN$
\bes
\|r^{j+2}\p^i_\th\p^j_r u_e^{(k)}\|_{L^\i}\leq \dl C_{i,j}.
\ees
After obtaining $(u_e^{(k)}, v_e^{(k)})$, we construct $p_e^{(k)}$ as following
\begin{align*}
p_e^{(k)}(\th,r):=\phi(r)- u^e(r)u_e^{(k)}(\th,r)-\p_r(ru^e(r)) \int^\th_0v_e^{(k)}(\bar{\th},r)d\bar{\th}+\int^\th_0f_e^{(k)}(\bar{\th},r)d\bar{\th},
\end{align*}
which satisfies
\bes
\p_\th p_e^{(k)}+u^e(r)\p_{\th} u_e^{(k)}+\p_r(ru^e(r))v_e^{(k)}=f_e^{(k)}.
\ees
Let $\phi(r)$ be a function satisfying
\begin{align*}
r\phi'(r)+u^e(r)\p_\th v^{(k)}_e(0,r)-2u^e(r)u_e^{(k)}(0,r)+g_e^{(k)}(0,r)=0.
\end{align*}
Combining the equations of $(u_e^{(k)}, v_e^{(k)})$, it's direct to obtain
\begin{align*}
u^e(r) \partial_\th v_e^{(k)}-2u^e(r) u^{(k)}_e+r\partial_rp_e^{(k)}=0.
\end{align*}
Hence, $(u_e^{(k)}, v_e^{(k)},p_e^{(k)})$ solves the equation (\ref{middleeulerk}) and satisfies (\ref{eulerkesti}). \ef

\subsubsection{Higher order boundary expansions near the boundary $\{r=1\}$}

\indent

By substituting the the outer boundary layer expansion \eqref{boundaryextensionu} into (\ref{nspolar}) and collecting the $\epsilon-k$th $(k\geq 2)$ order terms, we obtain
 the following steady boundary equations for $(u_p^{(k)},v_p^{(k+1)},p_p^{(k+1)})$
 \be\label{boundarylayeruk}
 \lt\{
 \begin{aligned}
 &(\t{\o}+u^{(0)}_p)\p_\th u^{(k)}_p+u^{(k)}_p\p_\th u^{(0)}_p+(v^{(1)}_e(\th,1)+v^{(1)}_p)\p_\zeta u^{(k)}_p-\p^2_\zeta u^{(k)}_p\\
 &+(v^{(k+1)}_p-v^{(k+1)}_p(\th,0))\p_\zeta u^{(0)}_p=f_{p,k}(\th,\zeta),\\
 &-\p_\zeta p^{(k+1)}_p=g_{p,k}(\th,\zeta),\\
 &\p_\th u^{(k)}_p+ \p_\zeta v^{(k+1)}_p+\p_\zeta(\zeta v_p^{(k)})=0,
 \end{aligned}
 \rt.
 \ee
 where the force term $f_{p,k}(\th,\zeta)$ and $g_{p,k}(\th,\zeta)$ are lower order terms, which decay fast at $\zeta$ infinity. The exact representations of $f_{p,k}$ and $g_{p,k}$ are not important here, we choose to omit to write it out.

For the system \eqref{boundarylayeruk}, we have the following proposition.

\begin{proposition}\label{propdcuk}
There exists $\dl_0>0$ such that for any $\dl\in(0,\dl_0)$, the equations (\ref{boundarylayeruk}) has a unique solution $(u_p^{(k)},v_p^{(k+1)})$ which satisfies
\begin{align*}
\begin{aligned}
&\lt\|\partial_\th^i\partial_\zeta^j \big({u}_p^{(k)}-A_{k}\big)\big<\zeta\big>^{\ell}\rt\|_{L^\i}\leq C_{i,j,\ell}\dl,\\
&\lt\|\partial_\th^i\partial_\zeta^j v_p^{(k+1)}\big<\zeta\big>^{\ell}\rt\|_{L^\i}\leq C_{i,j,\ell}\dl ,\q \int_0^{2\pi}v_p^{(k+1)}(\th,\zeta)d\th =0, \ \forall\ \zeta\geq 0,
\end{aligned}
\end{align*}
where
$A_{k}:=\lim\limits_{\zeta\rightarrow +\infty} u_p^{(k)}(\th,\zeta)$ is a constant which satisfies $|A_{k}|\leq C\dl.$
\end{proposition}
\pf

Let $\kappa\in C_c^\infty ([0,+\infty))$ satisfy
\begin{align*}
\kappa(0)=1,\ \int_0^{+\infty}\kappa(\zeta)d\zeta=0.
\end{align*}
For simplicity, we set
\begin{align*}
\bar{u}:&=\t{\o}+u_p^{(0)}, \ \bar{v}:=v_p^{(1)}+v^{(1)}_e(\th,1),\\
u:&=u_p^{(k)}+u_e^{(k)}(\th,1)\kappa(\zeta),\\
v:&=v_p^{(k+1)}-v_p^{(k+1)}(\th,0)+\zeta v_p^{(k)} -\p_\th u_e^{(k)}(\th,1)\int_0^\zeta \kappa(\bar{\zeta})d\bar{\zeta}.
\end{align*}
Then, the equations (\ref{boundarylayeruk}) reduce to
\begin{eqnarray}\label{bluk1}
\left \{
\begin {array}{ll}
\bar{u}\partial_\th u+\bar{v}\partial_\zeta u+u\partial_\th \bar{u}+v\partial_\zeta\bar{u}-\partial^2_{\zeta}u=\bar{f}_{p,k},\\[7pt]
\partial_\th u+\partial_\zeta v=0,\\[5pt]
 u(\th,\zeta)=u(\th+2\pi,\zeta),\ v(\th,\zeta)=v(\th+2\pi,\zeta)\\[5pt]
u|_{\zeta=0}=v|_{\zeta=0}=0,\  \lim\limits_{{\zeta\rightarrow +\infty}}\partial_\zeta u=0,
\end{array}
\right.
\end{eqnarray}
where $\bar{f}_{p,k}$ is $2\pi$-periodic function and decays fast as $\zeta\rightarrow +\infty$ and have the following estimate
\bes
\lt\|\partial_\th^i\partial_\zeta^j \bar{f}_{p,k}\big<\zeta\big>^\ell\rt\|_{L^\i}\leq C_{i,j,\ell}\dl,\ \forall\ \zeta\geq 0.
\ees

This system \eqref{bluk1} is solved  as \eqref{boundarylayeru1} and is exactly (2.47) in \cite{FeiGLT:2023CMP}. One can refer to there for further detailed proof of Proposition \ref{propdcuk}. Here we omit the details.

Next, we construct the pressure $p_p^{(k+1)}(\th,\zeta)$. Thanks to \eqref{boundarylayeruk}$_2$, we take
\bes
p_p^{(k+1)}(\th,\zeta)=\int^{+\i}_{\zeta}g_{p,k}(\th,\bar{\zeta})d\bar{\zeta}.
\ees
which satisfies
\begin{align*}
-\p_\zeta p_p^{(k+1)}(\th,\zeta)=g_{p,k}(\th,\zeta), \quad \lim_{\zeta\rightarrow +\infty}p_p^{(k+1)}(\th,\zeta)=0.
\end{align*}

\ef

\subsubsection{Correction of the boundary layer solutions and the Euler equations of higher order}
\indent

Since we want the solutions of the boundary layer equations decaying fast at infinity, we define
\bes
\t{u}^{(k)}_p:=u^{(k)}_p-A_{k},
\ees
and use $\t{u}^{(k)}_p$  to replace the original ${u}^{(k)}_p$. Then we need to correct $u^{(k)}_e$ to $\t{u}^{(k)}_e,$ such that they satisfy
\bes
\t{u}^{(k)}_e\big|_{r=1}=-u^{(k)}_p\big|_{\zeta=0}+A_{k}={u}^{(k)}_e\big|_{r=1}+A_{k},
\ees
and $(\t{u}^{(k)}_e,v^{(k)}_e)$ is still a solution of the system \eqref{middleeulerk}.

Also, using divergence-free condition and the harmonic property of  $rv_e^{(k)}$, a direct calculation implies that
\bes
\p_\th \lt( r(\Dl-\f{1}{r^2})u_e^{(k)} +\f{2}{r}\p_\th v_e^{(k)}\rt)=0.
\ees
Let
\bes
\t{u}^{(k)}_e={u}^{(k)}_e+h_{k}(r),
\ees
where $h_{k}(r)$ satisfies the following ODE
\be\label{correctionk}
\begin{aligned}
&r(\Dl-\f{1}{r^2})h_{k}(r)=-r(\Dl-\f{1}{r^2})u_e^{(k)} +\f{2}{r}\p_\th v_e^{(k)}:=rf_k(r),\\
&h_{k}(1)=A_{k},\q h_k(+\i)=0.
\end{aligned}
\ee
By estimates of $u_e^{(k)},v_e^{(k)}$, we have
\bes
\|r^{j+4}\p^{j}_r f_k(r)\|_{L^\i}\leq C_j.
\ees
The solution of \eqref{correctionk} is given by
\begin{align}
h_k(r)=&\f{A_k+\f{1}{2}\int^{+\i}_1 f_k(s)(1-s^2)ds}{r}+\f{1}{2r}\int^{+\i}_r f_k(s)s^2ds-\f{1}{2}r\int^{+\i}_r f_k(s)ds\nn\\
      :=& \f{\t{A}_k}{r}+\f{1}{2r}\int^{+\i}_r f_k(s)s^2ds-\f{1}{2}r\int^{+\i}_r f_k(s)ds.\nn
\end{align}
It is not hard to deduce that
\bes
\lt\|r^{j+2}\p^{j}_r \lt(h_k(r)-\f{\t{A}_k}{r}\rt)\rt\|_{L^\i}\leq C_{j} \dl.
\ees
Obviously, $(\t{u}^{(k)}_e, v^{(k)}_e)$ is still a solution of the system \eqref{middleeulerk} and satisfies \eqref{eulerkesti}. Also we have that
  \bes
 r(\Dl-\f{1}{r^2})\t{u}_e^{(k)} +\f{2}{r}\p_\th v_e^{(k)}\equiv 0.
  \ees
  By letting $\t{u}_p^{(k)}$ and $\t{u}_e^{(k)}$ to replace  ${u}_p^{(k)}$ and $u_e^{(k)}$, the equations for the $k$th order boundary layer \eqref{boundarylayeruk} is invariant. \ef

\subsection{Approximate solutions}\label{Approximate solutions}

\indent

In this subsection, we construct an approximate solution of the Navier-Stokes equations (\ref{nspolar}). First let $ \chi(r)\in C^\i_c([1,+\i))$ be a function satisfying
\bes
\chi(r)=\lt\{
\bali
&1,\q r\in \lt[1,2\rt],\\
&0,\q r\geq 3.
\eali
\rt.
\ees
  Set
\begin{align*}
&{u}_{p}^a(\th,r):=\chi(r)\sum^{22}_{k=0}\e^{k}u^{(k)}_p:=\chi(r)u_p,\\
&{v}_{p}^a(\th,r):=\chi(r)\sum^{23}_{k=1}\e^{k}v^{(k)}_p:=\chi(r)v_p,\q{p}_{p}^a(\th,r):=\chi^2(r)\sum^{23}_{k=1}\e^{k}p^{(k)}_p:=\chi^2(r)p_p,
\end{align*}
and
\begin{align*}
u_e^a:=u^e(r)+\sum^{22}_{k=1}\e^{k}u^{(k)}_e, \q v_e^a:=\sum^{22}_{k=1}\e^{k}v^{(k)}_e,\q p_e^a:=\sum^{22}_{k=0}\e^{k}p^{(k)}_e.
\end{align*}

We construct an approximate solution $(u^a,v^a,p^a)$ by
\begin{align*}
u^a(\th,r):&=u_e^a+{u}_{p}^a+\epsilon^{22}h(\th,r),\q v^a(\th,r):=v_e^a+{v}_{p}^a,\q p^a(\th,r):=p_e^a+{p}_p^a,
\end{align*}
where the corrector $h(\th,r)$ will be given in Appendix, which satisfies
\begin{align*}
h(\th,1)=h(\th,+\i)=0, \ \|r^{j+100}\partial_\th^i\partial_r^j h\|_2\leq C_{i,j},
\end{align*}
and makes $(u^a, v^a)$ be divergence-free
\begin{align*}
\p_\th u^a+\p_r(rv^a)=0.
\end{align*}
Moreover, $(u^a, v^a)$  satisfies the following boundary conditions
\begin{align*}
&( u^a,v^a)(\th,r)=( u^a,v^a)(\th+2\pi,r),\\[5pt]
&u^a(\th, 1)=\o+\dl f(\th), \ v^a(\th,1)=0,\\[5pt]
&u^a(\th,+\i)=0,\ v^a(\th,+\i)=0.
\end{align*}
By collecting the estimates in Proposition \ref{propeulerorder1} and Proposition \ref{propeulerorderk},  we deduce that
\bes
\lt \|r^{k+2}\p^j_\th\p^{k}_r \lt(u^a_e-u^e(r)-\f{O(\e\dl)}{r},v^a_e \rt)\rt \|_\infty\leq C_{j,k}\epsilon(\dl+\epsilon),
\ees
where $O(\e\dl)=\sum^{22}_{k=1} \e\t{A}_k$.

By collecting the estimates in Proposition \ref{propdcu0}, Proposition \ref{propdcu1} and Proposition \ref{propdcuk}, one has
\begin{align}
\|\zeta^\ell\p^{j}_\th\partial_\zeta^k u_p^a\|_\infty\leq C_{j,k,\ell} (\dl+\epsilon), \ \|\zeta^\ell\p^{j}_\th\partial_\zeta^k v_p^a\|_\infty\leq C_{j,k,\ell} \epsilon(\dl+\epsilon).\nn
\end{align}

Finally, set
 \begin{align*}
 R_u^a:&=u^{a} u_\theta^{a}+r v^{a} u_r^{a}+u^{a} v^{a}+p_\theta^{a}-\epsilon^2\left(\frac{u_{\theta \theta}^{a}}{r}+r u_{r r}^{a}+u_r^{a}+\frac{2}{r} v_\theta^{a}-\frac{u^{a}}{r}\right), \\
R_v^a:&=u^{a} v_\theta^{a}+r v^{a} v_r^{a}-\left(u^{a}\right)^2+r p_r^{a}-\epsilon^2\left(\frac{v_{\theta \theta}^{a}}{r}+r v_{r r}^{a}+v_r^{a}-\frac{2}{r} u_\theta^{a}-\frac{v^{a}}{r}\right),
 \end{align*}
  then, for $ 0\leq i\leq 10,\ 0\leq j\leq 1$, there holds
 \begin{align*}
  \left\|\p^i_\th\p^j_r R_u^a(\cdot,r)\right\|_{L^2_\theta}+\left\|\p^i_\th\p^j_r R_v^a(\cdot,r)\right\|_{L^2_\theta}\leq C_{i,j}\e^{22-j}r^{-4},\qquad r\geq1,
 \end{align*}
and $(u^a, v^a, p^a)$ satisfies
\begin{eqnarray*}
\left\{
\begin{array}{lll}
u^{a} u_\theta^{a}+r v^{a} u_r^{a}+u^{a} v^{a}+p_\theta^{a}-\epsilon^2\left(\frac{u_{\theta \theta}^{a}}{r}+r u_{r r}^{a}+u_r^{a}+\frac{2}{r} v_\theta^{a}-\frac{u^{a}}{r}\right)=R_u^a,\\[5pt]
u^{a} v_\theta^{a}+r v^{a} v_r^{a}-\left(u^{a}\right)^2+r p_r^{a}-\epsilon^2\left(\frac{v_{\theta \theta}^{a}}{r}+r v_{r r}^{a}+v_r^{a}-\frac{2}{r} u_\theta^{a}-\frac{v^{a}}{r}\right)=R_v^a,\\[5pt]
 \p_\th u^a+\p_r(rv^a)=0, \\[5pt]
( u^a,v^a)(\th,r)=( u^a,v^a)(\th+2\pi,r),\\[5pt]
u^a(\th, 1)=(\o+\dl f(\th)), \ v^a(\th,1)=0,\\[5pt]
u^a(\th,+\i)=0,\ v^a(\th,+\i)=0.
\end{array}
\right.
\end{eqnarray*}
This is exactly the constructed approximate solution that we stated in the beginning of Section \ref{sec2}. \qed

\section{Appendix}

In this section, we construct the corrector $h(\th,r)$ defined in Section \ref{Approximate solutions}. First, we state a simple lemma, similar to Lemma 6.1 in Appendix B of \cite{FeiGLT:2023CMP}; therefore, we omit the proof.

\begin{lemma}\label{corector equation}
Assume that $K(\th,r)$ is a $2\pi$-periodic smooth function which satisfies
\bes
\int_0^{2\pi}K(\th,r)d\th=0, \ \forall\ r\in [1,+\i); \quad K(\th,1)=K(\th, +\i)=0,
\ees
then there exists a $2\pi$-periodic function $h(\th,r)$ such that
\begin{align}
&\partial_\th h(\th,r)=K(\th,r); \ h(\th,1)=h(\th, +\i)=0;\nonumber\\
&\int_0^{2\pi}h(\th,r)d\th=0, \ \|\partial_\th^j\partial_r^kh\|_2\leq C\|\partial_\th^j\partial_r^kK\|_2.\nn
\end{align}
\end{lemma}

Next, we construct the corrector $h(\th,r)$ by the above lemma.
Direct computation gives
\begin{align*}
&\p_\th (u^a_e+ u^a_{p})+\p_r(rv^a_e+rv^a_{p})=\p_\th  u^a_{p}+\p_r(rv^a_{p})\\
=&\chi(r)(\p_\th {u}_p+\p_r(r{v}_p))+r\chi^{\prime}(r){v}_p,\nn\\
=&r\chi^{\prime}(r){v}_p=r\e^{22}\chi^{\prime}(r)(r-1)^{-22}\zeta^{22}{v}_p\nn\\
:=&-\e^{22} K(\th,r).
\end{align*}

Since $\chi'(r)=0$ for $r\in [1,2]\cup [3,+\i)$ and ${v}_p$ has the required properties, $K(\th,r)$ satisfies the assumptions of Lemma \ref{corector equation}. Hence, there exists $h(\th,r)$ such that
\beas
&\p_\th (u^a_e+ u^a_{p})+\p_r(rv^a_e+rv^a_{p})=-\e^{22}\p_\th h(\th,r),
\eeas
which implies that
\bes
\p_\th u^a+\p_r(rv^a)=0.
\ees

\section*{Data availability statement}

Data sharing is not applicable to this article as no datasets were generated or analysed during the current study.

\section*{Conflict of interest statement}

The authors declare that they have no conflict of interest.

\section*{Acknowledgments}

 X. Pan is supported by NSF of China  under Grant  No.12471222. We thank Dr. Chen Gao for helpful discussions on this topic.

\end{document}